\documentclass[reqno,11pt]{amsart}
\usepackage{amsmath,amsfonts,amssymb,amsxtra,latexsym,amscd,enumerate,amsthm,verbatim}

\usepackage{graphicx}

\usepackage[margin=1.39in]{geometry}
\numberwithin{equation}{section}

\newcommand{\R}{\mathbb{R}}

\newcommand{\E}{\mathcal{E}}

\newcommand{\B}{\mathcal{B}}

\newcommand{\G}{\mathcal{G}}
\newcommand{\T}{\mathbb{T}}
\newcommand{\C}{\mathbb{C}}
\newcommand{\Z}{\mathbb{Z}}
\newcommand{\eps}{\epsilon}

\numberwithin{equation}{section} %pour numeroter les equations par section

\newcommand{\obv}{\overline{\upsilon}}
\newcommand{\ubv}{\underline{\upsilon}}
\newtheorem{theorem}{Theorem}[section]
\newtheorem{lemma}[theorem]{Lemma}
\newtheorem{proposition}[theorem]{Proposition}

\newtheorem{remark}[theorem]{Remark}

\begin{document}

\title[]{Axi-symmetrization near point vortex solutions for the 2D Euler equation}

\author{Alexandru D. Ionescu}
\address{Princeton University}
\email{aionescu@math.princeton.edu}

\author{Hao Jia}
\address{University of Minnesota}
\email{jia@umn.edu}

\thanks{The first author was supported in part by NSF grant DMS-1600028 and by NSF-FRG grant DMS-1463753.  The second author was supported in part by DMS-1600779}

\begin{abstract}
{\small}
We prove a definitive theorem on the asymptotic stability of point vortex solutions to the full Euler equation in 2 dimensions. More precisely, we show that a small, Gevrey smooth, and compactly supported perturbation of a point vortex leads to a global solution of the Euler equation 2D, which converges weakly as $t\to\infty$ to a radial profile with respect to the vortex. The position of the point vortex, which is time dependent, stabilizes rapidly and becomes the center of the final, radial profile. The mechanism that leads to stabilization is mixing and inviscid damping.

\end{abstract}

\maketitle

\setcounter{tocdepth}{1}

\tableofcontents

\section{Introduction}\label{Introd}
The presence of coherent vortices is a prominent feature in two dimensional fluid flows, such as viscous flows with high Reynolds number and  perfect fluid flows. These vortices are believed to play an important role in the 2D turbulence theory (see  for example \cite{Ba1,Ba2,Benzi,Brachet,McW,McW1}). The stability analysis of vortices is among the oldest problems studied in hydrodynamics, starting with Rayleigh \cite{Ray}, Kelvin \cite{Kelvin}, and Orr \cite{Orr} and continuing to the present day, see for example \cite{Gallay, Gallay2, Hall, Schecter} and references therein.

To motivate the problem, recall that smooth solutions to the 2D Euler equation are global, due to the Beale-Kato-Majda criteria (see also \cite{Yudovich2} for the case of bounded initial data). The long time behavior of solutions is, however, very hard to understand and describe. Arnold (see \cite{Arnold}) proved an important criteria for nonlinear stability of some steady states using monotonicity formulae, but the precise dynamics near these solutions are not known.  There have been some attempts in building a theory of ``weak turbulence" for the two dimensional Euler equation, to explain the appearance of coherent structures, see e.g., chapter 34 of \cite{SverakNotes}. A proposed mathematical explanation is that generically, the vorticity $\omega(t)$ converges {\it weakly} but not strongly as $t\to\infty$. This would explain the local chaos versus global structure phenomenon. It is an attractive conjecture, but it  seems hard
to rigorously formulate, let alone prove such a conjecture. A more realistic approach is to consider the 2D Euler equations in physically relevant perturbative regimes, which is the problem we consider in this paper.

Vortices are radial solutions of the vorticity equation in 2D, and are stationary under the flow. Since vortices often self organize and may become the dominant feature in two dimensional fluid flows, it is important to understand their dynamical properties.

Numerical and physical experiments and formal asymptotic analysis (see \cite{Ba1,Ba2,Schecter} and references therein) suggest that small perturbations of vortices form spirals around the center of the vortex and the angle-dependent modes of the vorticity vanish in the weak sense as $t\to\infty$, which leads to ``axi-symmetrization" of the vorticity. This has been studied rigorously at the linearized level around a strictly decreasing vortex profile by Bedrossian--Coti Zelati--Vicol \cite{Bed2},  who established axi-symmetrization of the vorticity and precise rates of decay of the associated stream function. Coti Zelati and Zillinger  \cite{Zillinger3} studied a similar phenomenon around a class of degenerate circular flows that include the important class of point vortices. 

\subsection{Nonlinear asymptotic stability} In this paper we initiate the rigorous study of the full {\it{nonlinear}} asymptotic stability problem for vortices of the 2D Euler equation. We work with the simplest class of vortices, called {\it{point vortices}}, which are $\delta$-functions centered at points in $\mathbb{R}^2$. Such solutions (and more generally the so called $N-$vortex solutions) are models of general solutions with vorticity concentrated sharply in small neighborhoods, and have been studied by many authors. See for instance Kirchhoff \cite{Kirchhoff}, C.C. Lin \cite{CCLin}, see also a recent development \cite{JCWei}, and the book \cite{Majda} for more references.

To state our main conclusions, consider solutions to the 2D Euler equations of the form
\begin{equation}\label{IE1}
{\rm vorticity\,\,field}=\kappa\, \delta(P(t))+\omega,\qquad {\rm velocity\,\,field}=\nabla^{\perp}\Delta^{-1}\delta(P(t))+u,
\end{equation}
where $\kappa\in \mathbb{R}\backslash\{0\}$ is the strength of the point vortex, $\delta(P(t))$ is the Dirac mass centered at $P(t)=(P_1(t),P_2(t))\in\mathbb{R}^2$. We assume that $P(t)$ is not in the support of $\omega$, which will be satisfied as part of our analysis. Then the perturbation $\omega,u$ satisfy the equation
\begin{equation}\label{Euler}
\partial_t\omega+U\cdot\nabla\omega+u\cdot\nabla \omega=0,\qquad{\rm for}\,\,(x,y,t)\in \mathbb{R}^2\times[0,\infty),
\end{equation}
where 
\begin{equation}\label{U}
U=\nabla^{\perp}\Delta^{-1}\delta(P(t))=\frac{\kappa}{2\pi}\nabla^{\perp}\log{|(x,y)-P(t)|},
\end{equation}
and the velocity field $u$ and the stream function $\psi$ are determined through
\begin{equation}\label{St00}
u=\nabla^{\perp} \psi=(-\partial_y\psi,\partial_x\psi),\qquad\Delta\psi=\omega,\qquad{\rm for}\,\,(x,y)\in\mathbb{R}^2.
\end{equation}
In addition, the center $P(t)$ satisfies the ODE
\begin{equation}\label{PO}
P'(t)=\nabla^{\perp}\psi(t,P(t)).
\end{equation}
Denote
\begin{equation}\label{c0}
c_0:=\int_{\mathbb{R}^2}\omega(t,x,y) dxdy,
\end{equation}
which is preserved for all times, as long as the support of $\omega(t)$ is away from $P(t)$. The equations  \eqref{Euler}--\eqref{PO} can be derived rigorously when the vortex $P(t)$ lies outside of the support of $\omega(t)$, see for example \cite{Marchioro}. In our case, this support condition is propagated dynamically by the flow, as a consequence of the proof of stability.

In this paper we prove axi-symmetrization around a point vortex. More precisely we prove that small, Gevrey smooth,  and compactly supported perturbations symmetrize around the point vortex whose location changes in time and converges fast as $t\to\infty$. 

\begin{theorem}\label{Thorem1}
Assume that $\kappa\in\R\backslash\{0\}, \lambda\in(0,\infty), M\in(1,\infty)$, and $\omega_0\in C^{\infty}_0(\R^2)$ satisfies the support property ${\rm supp}\,\omega_0\subseteq\{x\in\mathbb{R}^2:\,|x|\in[1/M,M]\}$. Assume that
\begin{equation}\label{M1}
\int_{\R^2}e^{\lambda\langle \xi,\eta\rangle^{1/2}}\left|\widetilde{\omega_0}(\xi,\eta)\right|^2\,d\xi d\eta\leq \epsilon^2,
\end{equation}
for a sufficiently small constant $\epsilon\leq\epsilon(\kappa,M,\lambda)$. Then there is a unique smooth global solution $(\omega,P)$ of the system \eqref{Euler}--\eqref{PO} such that $P(t)$ stays outside of the support of $\omega(t)$ for all $t\ge0$. Moreover, for some $c=c(\kappa,M,\lambda)>0$,
\begin{equation}\label{M2}
|P(t)-P_{\infty}|\lesssim \epsilon \, e^{-c\langle t\rangle^{1/2}}, \qquad{\rm for\,\,some\,\,}P_{\infty}\in\R^2\,\,{\rm and\,\,all}\,\,t\ge0,
\end{equation}
and $\omega(t)$ converges weakly to a Gevrey-2 regular function $\omega_{\infty}\in C^{\infty}(\R^2)$ which is radial with respect to $P_{\infty}$, as $t\to\infty$.
\end{theorem}

In the above $P_{\infty}$ can be determined by the initial data through conservation laws. More precisely using the conservation of $\int_{\R^2}(x,y)\,\omega(t,x,y)\,dxdy+\kappa P(t)$ and $\int_{\R^2}\omega(t,x,y)\,dxdy$, we conclude that $P_{\infty}=(\kappa+c_0)^{-1}\int_{\R^2}(x,y)\,\omega_0(x,y)\,dxdy$ assuming without loss of generality $P(0)$ is the origin.

See also Theorem \ref{Thm} below for a more quantitative statement of our results. 

In the inviscid case,  this appears to be the first nonlinear asymptotic stability result for vortices in the plane, with general initial data. In the viscous case, Gallay and Wayne \cite{Gallay2} proved global stability of vortex solutions to the two dimensional Navier-Stokes equations, and Gallay \cite{Gallay} obtained the enhanced dissipation and axi-symmetrization for the vortex solutions.  

The vortex stabilization we prove here has similarities to the inviscid damping near Couette flows in $\mathbb{T}\times \R$, established in the remarkable paper \cite{BM} by Bedrossian-Masmoudi, and extended to the bounded periodic channel in \cite{IOJI} (see the longer discussion in subsection \ref{MainIde} below). However, the phenomenon of inviscid damping in the case of vortex solutions seems to be more natural than the inviscid damping near shear flows. Indeed, the mechanism leading to stability in the neighborhood of vortex solutions is the mixing in the angular variable $\theta$, which is naturally a periodic variable. On the other hand, for shear flows $(b(y),0)$ in the domain $\mathbb{T}\times[-L, L]$, the stability is due to the mixing in the $x$ direction by the velocity $b(y)$, and periodicity in $x$ (but not in $y$) has to be imposed as part of the setup.

\subsection{The general inviscid damping problem}
The mechanism of axi-symmetrization considered in this paper, also called ``inviscid damping", has been studied at the linearized level in many works. For example, in \cite{dongyi} Wei--Zhang--Zhao  proved sharp decay estimates for a general class of monotone shear flows in a channel. See also a recent refinement \cite{JiaL} where more precise asymptotics for the stream function was obtained. In \cite{Bed2} Bedrossian-Coti Zelati-Vicol established sharp linear inviscid damping around vortices with strictly decreasing profile in the plane. These works, and many others such as \cite{Grenier,Faddeev,Orr,Stepin,Dongyi2,Zillinger1,Zillinger2,Zillinger3}, provide                                                                                                                                                               a rather complete picture of the linear inviscid damping, under suitably general conditions.

However, the nonlinear asymptotic stability of steady flows is much more subtle and challenging. So far, the only nonlinear asymptotic stability results are those of Bedrossian and Masmoudi \cite{BM} for plane Couette flows, and the extension to a finite channel by the authors \cite{IOJI}. The method introduced in \cite{BM} for proving nonlinear asymptotic stability is based on the use of  time dependent imbalanced weights, which are designed carefully to control frequency dependent resonances. In contrast, the linear stability analysis in \cite{Bed2,JiaL,dongyi} is based on the regularity analysis of resolvents, which, through an oscillatory integral, implies decay in time of the stream function. It is not clear at the moment whether the linear decay estimates can be applied to the nonlinear analysis.\footnote{See however the recent result \cite{JiaG} which established linear inviscid damping in Gevrey spaces and appears to be more promising for nonlinear applications. The methods introduced in this paper played an important role in \cite{JiaG}.}

\subsection{Adapted polar coordinates and precise results}
 It is convenient to study (\ref{Euler}) in the polar coordinates, re-centered at $P(t)$. Let
\begin{equation}\label{COV}
(x,y)=P(t)+r(\cos{\theta},\sin{\theta}).
\end{equation}
In $(r,\theta)$ coordinate, we set the functions  $u'_r,u'_{\theta},\psi',\omega'$ as follows
\begin{equation}\label{u'In}
\begin{split}
&\omega'(t,\theta,r)=\omega(t,x,y), \qquad \psi'(t,\theta,r)=\psi(t,x,y),\\
&u_r'(t,\theta,r)e_r+u'_{\theta}(t,\theta,r)e_{\theta}=u(t,x,y),\qquad e_r:=(\cos\theta, \sin\theta),\qquad e_{\theta}:=(-\sin\theta,\cos\theta).
\end{split}
\end{equation}
Then equation (\ref{Euler}) can be rewritten as
\begin{equation}\label{PEu}
\partial_t\omega'-(P'(t),e_r)\partial_r\omega'-\frac{1}{r}(P'(t),e_{\theta})\partial_{\theta}\omega'+\frac{\kappa}{2\pi r^2}\partial_{\theta}\omega'-\frac{\partial_{\theta}\psi'\partial_r\omega'-\partial_{r}\psi'\partial_{\theta}\omega'}{r}=0,
\end{equation}
where the stream function $\psi'(t,\theta,r)$ can be calculated through
\begin{equation}\label{St}
\partial_{r}^2\psi'+\frac{1}{r}\partial_r\psi'+\frac{1}{r^2}\partial_{\theta}^2\psi'=\omega'.
\end{equation} 
In the above, 
\begin{equation}\label{PtS2}
P'(t)=\frac{1}{2\pi}\int_0^{\infty}\int_0^{2\pi}(\sin{\theta},-\cos{\theta})\,\omega'(t,\theta,r)d\theta dr.
\end{equation}
 $(P'(t),e_r),(P'(t),e_{\theta})$ are dot products between vectors $P'(t)$ and the vectors $e_r$ and $e_\theta$.

To state our main theorem we define the Gevrey spaces $\mathcal{G}^{\lambda,s}\big(\mathbb{T}\times \R\big)$ as the space of $L^2$ functions $f$ on $\mathbb{T}\times \R$ defined by the norm
\begin{equation}\label{Gev}
\|f\|_{\G^{\lambda,s}(\mathbb{T}\times \R)}:=\big\|e^{\lambda \langle k,\xi\rangle^s}\widetilde{f}(k,\xi)\big\|_{L^2(\mathbb{Z}\times\R)}<\infty.
\end{equation}
In the above, $\widetilde{f}$ denotes the Fourier transform of $f$ in $(\theta,r)\in\T\times\R$, $s\in(0,1]$, and $\lambda>0$. More generally, for any interval $I\subseteq\R$ we define the Gevrey spaces $\mathcal{G}^{\lambda,s}\big(\mathbb{T}\times I\big)$ by
\begin{equation}\label{Gev2}
\|f\|_{\G^{\lambda,s}(\mathbb{T}\times I)}:=\|Ef\|_{\G^{\lambda,s}(\mathbb{T}\times \R)},
\end{equation}
where $Ef(r):=f(r)$ if $r\in I$ and $Ef(r):=0$ if $r\notin I$. 

For any function $H(\theta,r)$ let $\langle H\rangle(r)$ denote the average of $H$ in $\theta$. Our main theorem in this paper is the following:

\begin{theorem}\label{Thm}
Assume that $\beta_0,\vartheta_0\in (0,1/8]$, $\kappa\in (0,\infty)$, and assume $\omega'_0$ is smooth initial data, satisfying the support condition ${\rm supp}\,\omega'_0\subseteq\T\times[\vartheta_0, 1/\vartheta_0]$ and the smallness condition
\begin{equation}\label{Eur0}
\|\omega_0'\|_{\G^{\beta_0,1/2}(\mathbb{T}\times\R)}=\epsilon\leq\overline{\epsilon},
\end{equation}
where $\overline{\eps}=\overline{\eps}(\beta_0,\vartheta_0,\kappa)>0$ is sufficiently small. We have the following conclusions.

(i) (global regularity) There exists a unique global solution $\omega'\in C([0,\infty):\G^{\beta_1,1/2}(\mathbb{T}\times\R))$ of the system \eqref{PEu}--\eqref{PtS2} with initial data $\omega'(0)=\omega'_0$, where $\beta_1=\beta_1(\beta_0,\vartheta_0,\kappa)>0$, such that ${\rm supp}\,\omega'(t)\subseteq \T\times[\vartheta_0/2, 2/\vartheta_0]$ and $|P(t)|<\vartheta_0/100$ for any $t\in[0,\infty)$.

(ii) (asymptotic stability) There exist $\Omega_{\infty}\in \G^{\beta_1,1/2}(\T\times\R)$ and $P_{\infty}=(P_{\infty}^1,P_{\infty}^2)\in\R^2$ with ${\rm supp}\,\Omega_{\infty}\subseteq \mathbb{T}\times [\vartheta_0/2,2/\vartheta_0]$ and $|P_{\infty}|\leq\vartheta_0/100$ such that 
\begin{equation}\label{convergence1}
\left\|\omega'(t,\theta+\kappa t/(2\pi r^2)+\Phi(t,r),r)-\Omega_{\infty}(\theta,r)\right\|_{\G^{\beta_1,1/2}(\mathbb{T}\times\R)} \lesssim\epsilon\langle t\rangle^{-1},
\end{equation}
\begin{equation}\label{convergence2}
|P(t)-P_{\infty}|\lesssim \epsilon\, e^{-\beta_1t^{1/2}},
\end{equation}
for any $t\ge 0$. Here 
\begin{equation}\label{DefPhi}
\Phi(t,r):=\int_0^t\frac{\langle u'_{\theta}\rangle(\tau,r)}{r}\,d\tau=\int_0^t\frac{\langle\partial_r\psi'\rangle(\tau,r)}{r}d\tau.
\end{equation}

(iii) (control of the velocity field) There exists $u'_{\infty}\in \G^{\beta_1,1/2}(\R)$ such that
\begin{equation}\label{convergenceofux}
\left\|\langle u'_{\theta}\rangle(t,r)-u'_{\infty}(r)\right\|_{\G^{\beta_1,1/2}(\R)}\lesssim\epsilon\langle t\rangle^{-2}.
\end{equation}
The function $u'_{\infty}$ satisfies the additional properties
\begin{equation}\label{AsymPhi2}
\partial_r( ru'_{\infty}(r))=r \Omega_{\infty}(r),\qquad u'_{\infty}(r)=0\text{ if }r\leq\vartheta_0/2,\qquad u'_{\infty}(r)=c_0/(2\pi)\text{ if }r\geq 2/\vartheta_0.
\end{equation}

Finally, the velocity field $u'$ satisfies the pointwise bounds
\begin{equation}\label{convergencetomean}
\left\|u'_{\theta}(t,\theta,r)-\langle u'_{\theta}\rangle(t,r)\right\|_{L^{\infty}(\mathbb{T}\times \R)}\lesssim\epsilon\langle t\rangle^{-1},
\end{equation}
\begin{equation}\label{convergenceuy}
\left\|u'_{r}(t,\theta,r)\right\|_{L^{\infty}(\mathbb{T}\times \R)}\lesssim\epsilon\langle t\rangle^{-2}.
\end{equation}
\end{theorem}

We conclude this subsection with a few remarks.

\begin{remark} The inviscid damping is generated by the term $\frac{\kappa}{2\pi r^2}\partial_{\theta}\omega'$ in the equation (\ref{PEu}). Indeed, at the linearized level the equation \eqref{PEu} is
\begin{equation}\label{PEuLin}
\partial_t\omega^{lin}+\frac{\kappa}{2\pi r^2}\partial_{\theta}\omega^{lin}=0,
\end{equation}
with the explicit solution 
\begin{equation}\label{PEuLin2}
\omega^{lin}(t,\theta,r)=\omega^{lin}_0(\theta-\kappa t/(2\pi r^2),r).
\end{equation}
Using now \eqref{St} we can express $\psi^{lin}_k$, $k\in\mathbb{Z}\backslash\{0\}$, as 
\begin{equation}\label{PEuLin3}
\psi^{lin}_k(t,r)=\int_{\mathbb{R}}G_k(r,\rho)\omega^{lin}_{0,k}(\rho)e^{-ik\kappa t/(2\pi \rho^2)}\,d\rho,
\end{equation}
where $\psi^{lin}_k$ and $\omega^{lin}_{0,k}$ denote the $k$--th Fourier modes of the functions $\psi^{lin}$ and $\omega^{lin}_{0}$ in $\theta$ and $G_k$ defined as in \eqref{G.k} is the associated Green function for the operator $\partial_r^2+\partial_r/r-k^2/r^2$. 

If $\omega_0^{lin}$ is smooth and compactly supported away from $0$ then one can integrate by parts in $\rho$ in \eqref{PEuLin3} to prove pointwise decay in time for the velocity field $u^{lin}=(u^{lin}_\theta,u^{lin}_r)=(\partial_r\psi^{lin},\partial_\theta\psi^{lin}/r)$, consistent with the bounds \eqref{convergencetomean}--\eqref{convergenceuy}. 

In other words the main conclusions of Theorem \ref{Thm} can be verified for the linearized flow as a consequence of the explicit formula \eqref{PEuLin}, as expected.
\end{remark}

\begin{remark}
The assumption that the point vortex lies outside the support of the perturbation is necessary for the inviscid damping in Gevrey spaces, due to ``boundary effect" the point vortex can exert (but the compact support assumption can probably be replaced by suitable decay). For shear flows in a finite channel, it was demonstrated by Zillinger \cite{Zillinger2} that scattering in high Sobolev spaces cannot hold if the initial perturbation does not vanish at the boundary then. The effect of boundary can also be seen in the precise asymptotics for the stream function in \cite{JiaL}. In our case, the point vortex plays a similar role as the boundary effect, in the polar coordinate, with $r=0$ as the ``boundary".
\end{remark}

\begin{remark}
The requirement that the perturbation is Gevrey-2 regular, as in \eqref{Eur0}, is natural by analogy with the case of the Couette flow. See the recent counter-examples of Deng--Masmoudi \cite{Deng} in slightly larger Gevrey spaces, and the more definitive counter-examples to inviscid damping in low Sobolev spaces by Lin--Zeng \cite{ZhiWu}. 
\end{remark}

\subsection{Main ideas of the proof}\label{MainIde} Our proof is based on analysis of the equations \eqref{PEu}--\eqref{PtS2}. Assuming first, for simplicity that $P'(t)=0$, we notice that the equations \eqref{PEu}--\eqref{St} are more complicated versions of the main equations in the analysis of the stability of the Couette flow on $\mathbb{T}\times\mathbb{R}$ (see, for example, equation (1.14) in \cite{IOJI}). 

We describe below some of the main ingredients of the proof, noting that the proof in the point vortex case contains all the difficulties of the Couette problem, in a  significantly more complicated variable coefficient setting.

\subsubsection{Renormalization and time-dependent weights} These are two key ideas introduced by Bedrossian--Masmoudi \cite{BM} in the case of Couette flow. 

As in \cite{BM} and \cite{IOJI}, we define a new system of coordinates $z,v$ by 
\begin{equation}\label{newCoordInt}
v(t,r):=\frac{\kappa}{2\pi r^2}+\frac{1}{t}\int_0^t\frac{\langle \partial_r\psi'\rangle(s,r)}{r}\,ds,\qquad z(t,\theta,r):=\theta-tv(t,r).
\end{equation}
The change of variable (\ref{newCoordInt}) is a nonlinear refinement of the linear change of coordinates $z=\theta-t\kappa/(2\pi r^2)$ in \eqref{PEuLin2}. This change of variables automatically ``adapts" to the asymptotic profile $\Omega_{\infty}(\theta,r)$, which has to be determined by the nonlinear flow, as $t\to\infty$. The main point is to remove the terms containing the non-decaying components $\kappa/(2\pi r^2)$ and  $\langle\partial_r\psi\rangle$ from the evolution equation satisfied by the renormalized vorticity, compare equations \eqref{PEu} and \eqref{Pef}. 

We can then analyze the resulting evolution equations pertubatively, establishing simultaneously smoothness and decay in time of certain components. For this we set up several energy functionals for the main functions (vorticity, stream function, and coordinate functions) and use the equations to estimate the increment in time. 

There is one significant difficulty here, due to the  ``reaction term" $\partial_v\mathbb{P}_{\neq0}\phi\cdot\partial_zF$ in \eqref{Pef}, which cannot be estimated using standard weights, due to loss of derivatives around certain ``critical" or ``resonant" times. The idea of Bedrossian-Masmoudi \cite{BM}, which we used in \cite{IOJI} and we use here as well,  is to employ time dependent norms associated to \emph{imbalanced weights}, which are designed carefully to absorb the large factors due to resonances, at the critical times. See the longer discussion in \cite{BM} and \cite[Section 1.3]{IOJI}. We note that the special weights we use here are refinements of the weights of \cite{BM}, but with additional smoothness that is important at other stages of the argument.

\subsubsection{Variable coefficients} A key new difficulty compared to the case of the Couette flow is that linearizing around vortex solutions leads to equations with variable coefficients, which are not small perturbations of constant functions. This is an important issue to understand for future applications as well, since the analysis of more general shear flows or vortices will likely lead to equations with variable coefficients.{\footnote{For general shear flows or vortices there is one more issue, namely the presence of an additional nonlocal linear term in the equation for the vorticity deviation, which requires different ideas.}}  

Since we rely heavily on Fourier analysis, this difficulty is present at all levels of the proof. We illustrate it in the case of the normalized stream function, for which we would like to use the elliptic equation \eqref{Pephi}. In order to prove the precise elliptic estimates (with the carefully designed weights), see \eqref{rec3.5} and Proposition \ref{Imboot3'} below, we take a two-step approach. First we consider the low frequency component of the normalized stream function, for which the weight is not important, and the elliptic estimates reduce to standard elliptic regularity. 

For high frequencies, we define suitable Fourier multiplier operators associated with the weights, and pass the operator through the variable coefficients. We then apply the standard elliptic estimates to the resulting elliptic equations. Since the coefficients are not constant functions, we need to bound a number of commutator terms, by showing that these terms are small. The key issue here is the smoothness of the weights. For example, for standard weights, one gains derivatives in commutator terms through inequalities of the form
\begin{equation*}
|K(\xi)-K(\eta)|\lesssim \langle\xi\rangle^{-\gamma}\max\{K(\xi),K(\eta)\}\langle\xi-\eta\rangle, \qquad {\rm if \,\,\langle\xi-\eta\rangle\ll \min\{\langle\xi\rangle, \langle\eta\rangle\}},
\end{equation*}
for some $\gamma>0$, where we assume $K$ is the weight function we use to commute. This is the case for example if $K(\xi)=\langle\xi\rangle^s$ (Sobolev type regularity) with $\gamma=1$ or $K(\xi)=e^{\langle\xi\rangle^{1/2}}$ (Gevrey regularity) with $\gamma=1/2$. 

In our case such a gain is not possible. Our main idea is to construct the weights to depend on an additional small parameter $\delta$ and prove a weaker bound of the form
\begin{equation}\label{S1}
\begin{split}
&\left|A_k(t,\xi)-A_k(t,\eta)\right|\\
&\lesssim \left[\frac{C(\delta)}{\langle k,\xi\rangle^{1/2}}+\sqrt{\delta}\right]\max\{A_k(t,\xi),A_k(t,\eta)\}, \quad {\rm if\,\, \langle \xi-\eta\rangle\lesssim 1\ll \min\{\langle k,\xi\rangle, \langle k,\eta\rangle\}}.
\end{split}
\end{equation}
Such bounds are still suitable to control the commutators, due to the gain of $\sqrt{\delta}$ for large frequencies. The parameter $\delta$ should be thought of as small relative to the structural constants $\beta_0,\vartheta_0,\kappa$, but much larger that $\epsilon$ (the size of the solution).

Smoothness conditions such as \eqref{S1} impose additional constraints on the weights $A_\ast(t,\xi)$, which have to be satisfied simultaneously with the resonance conditions described earlier. Because of all of these constraints, the precise design of the weights is very important. We use similar weights as in our paper \cite{IOJI}, and provide all the details and the proofs in section \ref{weights}.

\subsubsection{Interaction with the point vortex} We also need to quantify the effect of the global shift in coordinates caused by the movement of the vortex $P(t)$. On the one hand, $P'(t)$ decays fast, see \eqref{Pboot4}, which implies rapid stabilization of the point vortex. On the other hand, the global change of variables, in the presence of mixing, causes loss of regularity of functions. 

Roughly speaking, we can understand the problem as follows: the change of coordinates is expected to normalize the transport in $\theta$ with a {\it fixed} center. Thus a change in the center disrupts the renormalization. The issue can be seen already for simple functions such as $\cos\theta$ which is very smooth in the $\theta, r$ coordinate, while in the new $z,v$ coordinate (see \eqref{PCoC} for the change of coordinates formula) becomes $\cos(z+tv)$, which loses regularity rapidly. As a result, for terms such as $P'(t)\cos{(z+tv)}$ which are related to the effect of the global shift of coordinates, we have exactly the right balance between the fast decay of $P'(t)$ in time and the loss of regularity in the function $\cos(z+tv)$ to close the estimates. The situation is somewhat analogous to the boundary effect when studying the asymptotic stability of the Couette flow in the finite channel considered in \cite{IOJI}.

\subsection{Organization}\label{Organize} The rest of the paper is organized as follows. In section 2 we define our renormalized variables and set up the main bootstrap Proposition \ref{MainBootstrap}. In section 3 we show how the main theorems follow from the conclusions of the bootstrap proposition. In section 4 we prove bounds on the main structural functions in the problem. In sections 5 and 6 we prove the main improved bounds required to complete the bootstrap argument. In section 7 we  review the construction of the main weights $A_k(t,\xi), A_R(t,\xi), A_{NR}(t,\xi)$ from \cite{IOJI} and prove some more precise bounds on these weights.

\section{Renormalization, energy functionals, and the main proposition}\label{sec:variables}

\subsection{The main change of variables}\label{sec:cha}

Assume $T\geq 1$ and $\omega'$ is a sufficiently smooth solution of the system \eqref{PEu}--\eqref{PtS2} on some time interval $[0,T]$, which satisfies $\|\langle\omega'\rangle(t)\|_{H^{10}}\ll 1$ and is supported away from the origin for all $t\in[0,T]$, i.e. 
\begin{equation}\label{omeCom}
{\rm supp}\,\omega'(t)\subseteq B_{2/\vartheta_0}\backslash B_{\vartheta_0/2}.
\end{equation}

In analogy with the inviscid damping for the Couette flow, we now introduce variables that unwind the mixing in $\theta$. More precisely, we define
\begin{equation}\label{NewV}
v(t,r):=\frac{\kappa}{2\pi r^2}+\frac{1}{t}\int_0^t\frac{\langle \partial_r\psi'\rangle(s,r)}{r}\,ds,\qquad z(t,\theta,r):=\theta-tv(t,r),
\end{equation}
where $\langle h\rangle(r)$ denotes the average in $\theta\in\mathbb{T}$ of the function $h(\theta,r)$. It follows from (\ref{St}) that
\begin{equation}\label{psiInf2}
\partial_r\big(r\langle\partial_r\psi'\rangle\big)=r\langle \omega'\rangle.
\end{equation}
In particular, the map $r\to v$ defines a bijective change of variables, provided that $\|\langle\omega'\rangle(t)\|_{H^{10}}$ is sufficiently small (depending only on $\kappa$ and $\vartheta_0$). In view of \eqref{omeCom} and \eqref{c0} we have
\begin{equation}\label{psiInf4}
r\langle\partial_r\psi'\rangle(t,r)=0\,\,{\rm if\,\,}r<\vartheta_0/2\qquad{\rm and}\qquad r\langle\partial_r\psi'\rangle(t,r)=\frac{c_0}{2\pi}\,\,{\rm if\,\,}r>2/\vartheta_0.
\end{equation}

We define the functions $F$ and $\phi$ by the identities
\begin{equation}\label{fphi}
F(t,z(t,\theta,r),v(t,r)):=\omega'(t,\theta,r),\qquad\phi(t,z(t,\theta,r),v(t,r)):=\psi'(t,\theta,r).
\end{equation}
Direct calculations show that
\begin{equation}\label{DNC}
\begin{split}
&\partial_t\omega'(t,\theta,r)=\partial_tF(t,z,v)-\frac{\kappa}{2\pi r^2}\partial_zF(t,z,v)-\frac{\langle \partial_r\psi'\rangle(t,r)}{r}\partial_zF(t,z,v)+\partial_tv(t,r)\,\partial_vF(t,z,v),\\
&\partial_r\omega'(t,\theta,r)=-t\partial_rv(t,r)\,\partial_zF(t,z,v)+\partial_rv(t,r)\partial_vF(t,z,v),\\
& \partial_{\theta}\omega'(t,\theta,r)=\partial_zF(t,z,v). 
\end{split}
\end{equation}
Completely analogous formulas hold connecting $\psi'$ and $\phi$. 

We also define the functions $V',V'',\dot{V},\varrho$ by the formulas
\begin{equation}\label{V}
V'(t,v):=\partial_rv(t,r),\qquad V''(t,v):=\partial_r^2v(t,r),\qquad \dot{V}(t,v):=\partial_tv(t,r),\qquad\varrho(t,v):=\frac{1}{r}.
\end{equation}
Using \eqref{DNC}--\eqref{V} we can rewrite the equation (\ref{PEu}) in terms $F,\phi$ as follows,
\begin{equation}\label{EM}
\partial_tF+\dot{V}\partial_vF-\varrho V'\partial_z\phi\partial_vF+\varrho V'\partial_v\mathbb{P}_{\neq0}\phi\partial_zF-P'_v\,V'(\partial_v-t\partial_z)F-\varrho P'_z\,\partial_zF=0.
\end{equation}
In the above, 
\begin{equation}\label{Pr}
\begin{split}
&P'_v(t,z,v)=(P'(t),e_r)=P'_1(t)\cos(z+tv)+P'_2(t)\sin(z+tv),\\
& P'_{z}(t,z,v)=(P'(t),e_{\theta})=-P'_1(t)\sin(z+tv)+P'_2(t)\cos(z+tv), \\
& \mathbb{P}_{\neq0}\phi=\phi-\langle\phi\rangle.
\end{split}
\end{equation}
Moreover, using \eqref{St}, the normalized stream function $\phi$ satisfies the equation
\begin{equation}\label{ell}
\varrho^2\partial_z^2\phi+(V')^2(\partial_v-t\partial_z)^2\phi+V''(\partial_v-t\partial_z)\phi+\varrho V'(\partial_v-t\partial_z)\phi=F.
\end{equation}

\subsubsection{Equations for the coordinate functions}

We would also like to derive equations for the change-of-coordinates functions. We notice that the functions $V',V'',\varrho$ defined in \eqref{V} are not ``small"; to derive useful equations we need to construct suitable combinations of these functions, which are small, and then derive evolution equations for these combinations.

From (\ref{NewV}) and (\ref{St}), it follows 
\begin{equation}\label{prv}
\partial_rv(t,r)=-\frac{\kappa}{\pi r^3}-\frac{2}{t}\int_0^t\frac{\langle \partial_r\psi'\rangle(s,r)}{r^2}ds+\frac{1}{t}\int_0^t\frac{\langle \omega'\rangle(s,r)}{r}ds.
\end{equation}
The identities (\ref{psiInf4}) and (\ref{NewV}) imply that
\begin{equation}\label{psiInf5}
v(t,r)=\frac{\kappa}{2\pi r^2}\,\,{\rm if}\,\,r<\vartheta_0/2,\qquad {\rm and}\qquad v(t,r)=\frac{\kappa+c_0}{2\pi r^2}\,\,{\rm if}\,\,r>2/\vartheta_0,
\end{equation}
therefore
\begin{equation}\label{psiInf6}
\partial_rv(t,r)=-2v/r\,\,\,{\rm if}\,\,\,r<\vartheta_0/2\,\,\,\text{ or }\,\,\,r>2/\vartheta_0. 
\end{equation}
Let
\begin{equation}\label{cf1}
V_\ast(t,v):=V'(t,v)+2\varrho(t,v) v.
\end{equation}
It follows from \eqref{psiInf6} that $V_\ast(t,v)=0$ if $r\in\mathbb{R}^+\backslash[\vartheta_0/2,2/\vartheta_0]$. Moreover, using \eqref{NewV} and \eqref{prv},
\begin{equation}\label{cf1.5}
V_\ast(t,v(t,r))=(\partial_rv)(t,r)+\frac{2v(t,r)}{r}=\frac{1}{t}\int_0^t\frac{\langle \omega'\rangle(s,r)}{r}ds.
\end{equation}
Therefore, for any $t\in[0,T]$,
\begin{equation*}
\frac{\langle \omega'\rangle(t,r)}{r}=\frac{d}{dt}\big[tV_\ast(t,v(t,r))\big]=V_\ast(t,v(t,r))+t(\partial_tV_\ast)(t,v(t,r))+t(\partial_tv)(t,r)(\partial_vV_\ast)(t,v(t,r)),
\end{equation*}
which can be rewritten in the form
\begin{equation}\label{cf2}
(\partial_tV_\ast)(t,v)+\dot{V}(t,v)(\partial_vV_\ast)(t,v)=\frac{1}{t}\big(\langle F\rangle(t,v)\varrho(t,v)-V_\ast(t,v)\big).
\end{equation}

We define now
\begin{equation}\label{cf4}
\varrho_\ast(t,v):=\varrho(t,v)-\sqrt{\frac{2\pi v}{\kappa}}.
\end{equation}
Using the definitions we have
\begin{equation*}
\varrho_\ast(t,v(t,r))=\frac{1}{r}-\sqrt{\frac{2\pi v(t,r)}{\kappa}}.
\end{equation*}
Therefore, after taking time derivatives
\begin{equation}\label{cf6}
(\partial_t\varrho_\ast)(t,v)+\dot{V}(t,v)(\partial_v\varrho_\ast)(t,v)=-\sqrt{\frac{\pi}{2\kappa v}}\dot{V}(t,v).
\end{equation}

%It follows from \eqref{NewV} that
%\begin{equation*}
%tv(t,r)=\frac{\kappa t}{2\pi r^2}+\int_0^t\frac{\langle \partial_r\psi'\rangle(s,r)}{r}\,ds.
%\end{equation*}
%Therefore, for any $t\in[0,T]$,
%\begin{equation*}
%\begin{split}
%v(t,r)+t(\partial_t v)(t,r)=\frac{\kappa}{2\pi r^2}+\frac{\langle \partial_r\psi'\rangle(t,r)}{r}.
%\end{split}
%\end{equation*}
%In view of (\ref{DNC}), we also have
%\begin{equation}\label{avp2}
%\langle \partial_r\psi'\rangle(t,r)=\langle \partial_v\phi\rangle(t,v(t,r))V'(t,v(t,r)).
%\end{equation}
%Therefore we derive the equation
%\begin{equation}\label{cf7}
%t\dot{V}(t,v)=-v+\frac{\kappa}{2\pi}\varrho(t,v)^2+\langle \partial_v\phi\rangle(t,v)V'(t,v)\varrho(t,v).
%\end{equation}

Next we calculate
\begin{equation}\label{ptv}
\partial_tv(t,r)=\frac{1}{t}\left[-\frac{1}{t}\int_0^t\frac{\langle \partial_r\psi'\rangle(s,r)}{r}ds+\frac{\langle\partial_r\psi'\rangle(t,r)}{r}\right],
\end{equation}
and
\begin{equation}\label{ptrv}
\partial^2_{rt}v(t,r)=\frac{1}{t}\left[\frac{2}{t}\int_0^t\frac{\langle \partial_r\psi'\rangle(s,r)}{r^2}ds-\frac{1}{t}\int_0^t\frac{\langle \omega'\rangle(s,r)}{r}ds-2\frac{\langle \partial_r\psi'\rangle(t,r)}{r^2}+\frac{\langle\omega'\rangle(t,r)}{r}\right].
\end{equation}
Therefore
\begin{equation}\label{ptrvC}
\partial^2_{rt}v(t,r)+\frac{2}{r}\partial_tv(t,r)=\frac{1}{t}\left[-\frac{1}{t}\int_0^t\frac{\langle \omega'\rangle(s,r)}{r}ds+\frac{\langle\omega'\rangle(t,r)}{r}\right].
\end{equation}
Moreover
\begin{equation}\label{cfv} 
\langle \omega'\rangle(t,r)=\langle F\rangle(t,v(t,r)),\qquad \partial^2_{tr}v(t,r)=V'(t,v(t,r))(\partial_v\dot{V})(t,v(t,r)),
\end{equation}
Using (\ref{ptrvC}), (\ref{cfv}), and \eqref{cf1.5} we see that
\begin{equation}\label{ptdv}
t[V'(t,v)(\partial_v\dot{V})(t,v)+2\varrho(t,v)\dot{V}(t,v)]=-V_\ast(t,v)+\langle F\rangle(t,v)\varrho(t,v).
\end{equation}

Finally, we define\footnote{The function $W_\ast$ is, of course, ``small" since both $V_\ast$ and $F$ are ``small". However, the precise combination we choose here has the additional critical property that it decays as $t\to\infty$ essentially at the rate of $t^{-1}$, compare with the energy functional $\mathcal{E}_{W_\ast}$. This decay property is not shared by the functions $V_\ast$ and $F$.}
\begin{equation}\label{cf10}
W_\ast(t,v):=-V_\ast(t,v)+\langle F\rangle(t,v)\varrho(t,v).
\end{equation}
Then we calculate, using \eqref{EM}, \eqref{cf2}, and $\partial_t\varrho+\dot{V}\partial_v\varrho=0$,
\begin{equation}\label{cf11}
\begin{split}
&\partial_tW_\ast+\dot{V}\partial_vW_\ast=\varrho(\partial_t\langle F\rangle+\dot{V}\partial_v \langle F\rangle)-(\partial_tV_\ast+\dot{V}\partial_v V_\ast)\\
&=-W_\ast/t+\varrho^2 V'\langle \partial_z\phi\partial_vF\rangle-\varrho^2 V'\langle\partial_v\mathbb{P}_{\neq0}\phi\partial_zF\rangle+\varrho V'\langle P'_v(\partial_v-t\partial_z)F\rangle+\varrho^2 \langle P'_z\,\partial_zF\rangle.
\end{split}
\end{equation}

We summarize our calculations so far in the following:

\begin{proposition}\label{ChangedEquations}
Assume that $\omega',\psi':[0,T]\times\T\times\R_+\to\R$ and $P':[0,T]\to\R^2$ are smooth functions that satisfy the system \eqref{PEu}--\eqref{PtS2}.
Assume that $\omega'(t)$ is supported in $\T\times[\vartheta_0/2,2/\vartheta_0]$ and $\|\langle\omega'\rangle(t)\|_{H^{10}}\ll 1$ for any $t\in[0,T]$. 

We define the change-of-coordinate functions $(z,v):\mathbb{T}\times\mathbb{R}^+\to \mathbb{T}\times \mathbb{R}^+$ by
\begin{equation}\label{PCoC}
v:=\frac{\kappa}{2\pi r^2}+\frac{1}{t}\int_0^t\frac{\langle \partial_r\psi'\rangle(s,r)}{r}\,ds,\qquad z:=\theta-tv,
\end{equation}
and the new functions $F,\phi:[0,T]\times\mathbb{T}\times\mathbb{R}_+\to\mathbb{R}$ and $V',V'',\dot{V},\varrho,V_\ast,\varrho_{\ast},W_\ast:[0,T]\times\mathbb{R}_+\to\mathbb{R}$
\begin{equation}\label{rea21}
\begin{split}
&F(t,z,v):=\omega'(t,\theta,r),\qquad \phi(t,z,v):=\psi'(t,\theta,r),
 \end{split}
\end{equation}
\begin{equation}\label{rea22}
V'(t,v):=\partial_rv(t,r),\qquad V''(t,v):=\partial_{r}^2v(t,r),\qquad \dot{V}(t,v):=\partial_tv(t,r),
\end{equation}
\begin{equation}\label{rea22'}
\varrho(t,v):=1/r,\qquad \varrho_{\ast}(t,v):=\varrho(t,v)-\sqrt{\frac{2\pi v}{\kappa }},
\end{equation}
\begin{equation}\label{rea23}
V_\ast(t,v):=V'(t,v)+2\varrho(t,v) v, \qquad W_\ast(t,v):=-V_\ast(t,v)+\langle F\rangle(t,v)\varrho(t,v).
\end{equation}

Set
\begin{equation}\label{vbar}
\overline{\upsilon}:= \frac{4\kappa}{\pi \vartheta_0^2}\qquad {\rm and}\qquad  \underline{\upsilon}:=\frac{(\kappa+c_0)\vartheta_0^2}{16\pi}.
\end{equation}
Then, for any $t\in[0,T]$,
\begin{equation}\label{Psupp}
{\rm supp}\,F(t)\subseteq \mathbb{T}\times[\ubv,\obv],\qquad {\rm supp}\,\dot{V}(t) \cup {\rm supp}\, V_\ast(t)\cup {\rm supp}\, W_\ast(t)\subseteq [\ubv,\obv],
\end{equation}
\begin{equation}\label{Prho}
\varrho(t,v)=\sqrt{\frac{2\pi v}{\kappa}}\,\,{\rm for}\,\,v\in(\obv,\infty)\qquad{\rm and}\qquad \varrho(t,v)=\sqrt{\frac{2\pi v}{\kappa+c_0}}\,\,{\rm for}\,\,v\in(0,\ubv).
\end{equation}

The new variables $F,V_\ast,\varrho_\ast,W_\ast$ satisfy the evolution equations
\begin{equation}\label{Pef}
\partial_tF+\dot{V}\partial_vF=\varrho V'\partial_z\phi\partial_vF-\varrho V'\partial_v\mathbb{P}_{\neq0}\phi\partial_zF+P'_v\,V'(\partial_v-t\partial_z)F+\varrho P'_z\,\partial_zF,
\end{equation}
\begin{equation}\label{PeV'}
\partial_tV_\ast+\dot{V}\partial_vV_\ast=W_\ast/t,
\end{equation}
\begin{equation}\label{Perho}
\partial_t\varrho_\ast+\dot{V}\partial_v\varrho_\ast=-\sqrt{\frac{\pi}{2\kappa v}}\dot{V},
\end{equation}
\begin{equation}\label{PeH}
\begin{split}
\partial_tW_\ast+\dot{V}\partial_vW_\ast&=-W_\ast/t+\varrho^2 V'\langle \partial_z\phi\partial_vF\rangle-\varrho^2 V'\langle\partial_v\mathbb{P}_{\neq0}\phi\partial_zF\rangle\\
&+\varrho V'\langle P'_v(\partial_v-t\partial_z)F\rangle+\varrho^2 \langle P'_z\,\partial_zF\rangle,
\end{split}
\end{equation}
where, with $P'$ as in \eqref{PtS2},
\begin{equation}\label{PrS}
\begin{split}
&P'_v(t,z,v)=(P'(t),e_r)=P'_1(t)\cos(z+tv)+P'_2(t)\sin(z+tv),\\
& P'_{z}(t,z,v)=(P'(t),e_{\theta})=-P_1'(t)\sin(z+tv)+P'_2(t)\cos(z+tv).
\end{split}
\end{equation}

The variables $\phi,V'',$ and $\dot{V}$ satisfy the elliptic-type identities
\begin{equation}\label{Pephi}
\varrho^2\partial_z^2\phi+(V')^2(\partial_v-t\partial_z)^2\phi+V''(\partial_v-t\partial_z)\phi+\varrho V'(\partial_v-t\partial_z)\phi=F,
\end{equation}
\begin{equation}\label{PV''}
V''=V'\partial_vV',\qquad V'\partial_v\dot{V}+2\varrho\dot{V}=W_\ast/t.
\end{equation}
\end{proposition}

\subsection{Weights, energy functionals, and the bootstrap proposition}\label{weightsdef} In this subsection we construct our main energy functionals and state our main bootstrap proposition.

We will use three main weights $A_{NR}$, $A_R$, and $A_k$. These special weights are ``imbalanced" and can distinguish ``resonant" and ``non-resonant" times. For simplicity, the weights we employ here are identical to the weights defined in our paper \cite{IOJI}; this allows us to use some of the estimates already proved there.

Fix $\delta_0>0$, we define the function $\lambda:[0,\infty)\to[\delta_0,3\delta_0/2]$ by
\begin{equation}\label{reb10.5}
\lambda(0)=\frac{3}{2}\delta_0,\,\,\,\,\lambda'(t)=-\frac{\delta_0\sigma_0^2}{\langle t\rangle^{1+\sigma_0}},
\end{equation}
for $\sigma_0>0$ is small (say $\sigma_0:=0.01$). In particular, $\lambda$ is decreasing on $[0,\infty)$. Define
\begin{equation}\label{reb11}
A_R(t,\xi):=\frac{e^{\lambda(t)\langle\xi\rangle^{1/2}}}{b_R(t,\xi)}e^{\sqrt{\delta}\langle\xi\rangle^{1/2}},\qquad A_{NR}(t,\xi):=\frac{e^{\lambda(t)\langle\xi\rangle^{1/2}}}{b_{NR}(t,\xi)}e^{\sqrt{\delta}\langle\xi\rangle^{1/2}},
\end{equation}
where $\delta>0$ is a small parameter that may depend only on $\delta_0$, $\vartheta_0$, and $\kappa$. Then we define
\begin{equation}\label{reb12}
A_k(t,\xi):=e^{\lambda(t)\langle k,\xi\rangle^{1/2}}\Big(\frac{e^{\sqrt{\delta}\langle\xi\rangle^{1/2}}}{b_k(t,\xi)}+e^{\sqrt{\delta}|k|^{1/2}}\Big),\qquad k\in\mathbb{Z}.
\end{equation}

The precise definitions of the weights $b_{NR},\,b_R,\,b_k$ are very important; all the details are provided in section \ref{weights}. For now we simply note that, for any $t,\xi,k$,
\begin{equation}\label{reb13}
e^{-\delta\sqrt{|\xi|}}\leq b_R(t,\xi)\leq b_k(t,\xi)\leq b_{NR}(t,\xi)\leq 1,
\end{equation}
 In other words, the weights $1/b_{NR},\,1/b_R,\,1/b_k$ are small when compared to the main factors $e^{\lambda(t)\langle\xi\rangle^{1/2}}$ and $e^{\lambda(t)\langle k,\xi\rangle^{1/2}}$ in the weights $A_{NR},\,A_R,\,A_k$. However, their relative contributions are important as they are used to distinguish between ``resonant" and ``non-resonant" times. 

\subsubsection{The main bootstrap proposition} Assume that $\omega',\psi'$ are as in Proposition \ref{ChangedEquations} and define the functions $F,\phi,V',V'',\dot{V},\varrho,V_\ast,\varrho_{\ast},W_\ast$ as in \eqref{rea21}--\eqref{rea23}. We fix also Gevrey cutoff functions $\Psi,\Psi^\dagger:\mathbb{R}\to[0,1]$ satisfying
\begin{equation}\label{rec0}
\begin{split}
&\big\|e^{\langle\xi\rangle^{3/4}}\widetilde{\Psi}(\xi)\big\|_{L^\infty}\lesssim 1,\qquad {\rm supp}\,\Psi\subseteq \big[\ubv/3,3\obv\big], \qquad\Psi\equiv 1\text{ in }\big[\ubv/2,2\obv\big],\\
&\big\|e^{\langle\xi\rangle^{3/4}}\widetilde{\Psi^\dagger}(\xi)\big\|_{L^\infty}\lesssim 1,\qquad {\rm supp}\,\Psi^\dagger\subseteq \big[\ubv/9,9\obv\big], \qquad\Psi^\dagger\equiv 1\text{ in }\big[\ubv/8,8\obv\big],
 \end{split}
\end{equation}
where $\,\,\widetilde{}\,\,$ denotes the Fourier transform either on $\mathbb{R}$ or on $\mathbb{T}\times\mathbb{R}$. See subsection \ref{GevCut} for an explicit construction of such Gevrey cutoff functions.

We are now ready to define the main energy functionals. Let
\begin{equation}\label{rec1}
\mathcal{E}_F(t):=\sum_{k\in \mathbb{Z}}\int_{\R}A_k^2(t,\xi)\big|\widetilde{F}(t,k,\xi)\big|^2\,d\xi,
\end{equation}
\begin{equation}\label{rec3.5}
\mathcal{E}_{\phi}(t):=\sum_{k\in \mathbb{Z}\setminus\{0\}}\int_{\R}A_k^2(t,\xi)\frac{|k|^4\langle t\rangle^2\langle t-\xi/k\rangle^4}{|\xi/k|^2+\langle t\rangle^2}\big|\widetilde{\mathbb{P}_{\neq 0}(\Psi\phi)}(t,k,\xi)\big|^2\,d\xi,
\end{equation}
\begin{equation}\label{rec2}
\mathcal{E}_{V_\ast}(t):=\int_{\R}A_R^2(t,\xi)\big|\widetilde{V_\ast}(t,\xi)\big|^2\,d\xi,
\end{equation}
\begin{equation}\label{rec2.1}
\mathcal{E}_{\varrho_{\ast}}(t):=\int_{\R}A_R^2(t,\xi)\big|\widetilde{(\Psi^\dagger\varrho_{\ast})}(t,\xi)\big|^2\,d\xi,
\end{equation}
\begin{equation}\label{rec3}
\mathcal{E}_{W_\ast}(t):=\mathcal{K}^{2}\int_{\R}A_{NR}^2(t,\xi)\big(\langle t\rangle^{3/2}\langle\xi\rangle^{-3/2}\big)\big|\widetilde{W_\ast}(t,\xi)\big|^2\,d\xi,
\end{equation}
where $\mathcal{K}\geq 1$ is a large constant to be fixed (depending only on $\delta$).

We define also $\dot{A}_Y(t,\xi):=(\partial_t A_Y)(t,\xi)$, $Y\in\{NR,R,k\}$, and the space-time integrals
\begin{equation}\label{rec4}
\mathcal{B}_F(t):=\int_1^t\sum_{k\in \mathbb{Z}}\int_{\R}|\dot{A}_k(s,\xi)|A_k(s,\xi)\big|\widetilde{F}(s,k,\xi)\big|^2\,d\xi ds,
\end{equation}
\begin{equation}\label{rec6.5}
\mathcal{B}_{\phi}(t):=\int_1^t\sum_{k\in \mathbb{Z}\setminus\{0\}}\int_{\R}|\dot{A}_k(s,\xi)|A_k(s,\xi)\frac{|k|^4\langle s\rangle^2\langle s-\xi/k\rangle^4}{|\xi/k|^2+\langle s\rangle^2}\big|\widetilde{\mathbb{P}_{\neq 0}(\Psi\phi)}(t,k,\xi)\big|^2\,d\xi ds,
\end{equation}
\begin{equation}\label{rec5}
\mathcal{B}_{V_\ast}(t):=\int_1^t\int_{\R}|\dot{A}_R(s,\xi)|A_R(s,\xi)\big|\widetilde{V_\ast}(s,\xi)\big|^2\,d\xi ds,
\end{equation}
\begin{equation}\label{rec5.1}
\mathcal{B}_{\varrho_{\ast}}(t):=\int_1^t\int_{\R}|\dot{A}_R(s,\xi)|A_R(s,\xi)\big|\widetilde{(\Psi^\dagger\varrho_{\ast})}(s,\xi)\big|^2\,d\xi ds,
\end{equation}
\begin{equation}\label{rec6}
\mathcal{B}_{W_\ast}(t):=\mathcal{K}^{2}\int_1^t\int_{\R}|\dot{A}_{NR}(s,\xi)|A_{NR}(s,\xi)\big(\langle t\rangle^{3/2}\langle\xi\rangle^{-3/2}\big)\big|\widetilde{W_\ast}(s,\xi)\big|^2\,d\xi ds.
\end{equation}

Our main proposition is the following: 

\begin{proposition}\label{MainBootstrap}
Assume $T\geq 2$ and $\omega'\in C([0,T]:\G^{2\delta_0,1/2})$ is a solution of the system \eqref{PEu}--\eqref{PtS2} such that $\omega'(t)$ is supported in $\T\times[\vartheta_0/2,2/\vartheta_0]$ and satisfies the smallness assumption $\|\langle\omega'\rangle(t)\|_{H^{10}}\ll 1$ for any $t\in [0,T]$. Define $F,\phi,V',V'',\dot{V},\varrho,V_\ast,\varrho_{\ast},W_\ast$ as in \eqref{rea21}--\eqref{rea23}. Assume that $\eps_1$ is sufficiently small depending on $\delta$, and
\begin{equation}\label{boot1}
\sum_{g\in\{F,\phi,V_\ast,\varrho_\ast,W_\ast\}}\mathcal{E}_g(t)\leq\eps_1^3\qquad\text{ for any }t\in[0,2],
\end{equation}
\begin{equation}\label{boot2}
\sum_{g\in\{F,\phi,V_\ast,\varrho_\ast,W_\ast\}}\big[\mathcal{E}_g(t)+\mathcal{B}_g(t)\big]\leq\eps_1^2 \qquad\text{ for any }t\in[1,T].
\end{equation}

(i) Then for any $t\in[1,T]$ we have the improved bounds
\begin{equation}\label{boot3}
\sum_{g\in\{F,\phi,V_\ast,\varrho_\ast,W_\ast\}}\big[\mathcal{E}_g(t)+\mathcal{B}_g(t)\big]\leq\eps_1^2/2.
\end{equation}
Moreover, for the functions $F$ and $\Theta$, we have the stronger bounds for $t\in[1,T]$
\begin{equation}\label{boot3'}
\sum_{g\in\{F,\phi\}}[\mathcal{E}_g(t)+\mathcal{B}_g(t)]\lesssim_{\delta}\eps_1^3.
\end{equation}

(ii) In addition, for any $t\in[1,T]$ the following estimates hold:
\begin{equation}\label{boot4}
\begin{split}
t^2\|\dot{V}(t)\|_{\mathcal{G}^{\delta_0,1/2}}+t^3\|\langle(\partial_t+\dot{V}\partial_v)F\rangle(t)\|_{\mathcal{G}^{\delta_0,1/2}}&\lesssim_\delta\eps_1^{3/2},\\
e^{0.1\delta_0t^{1/2}}\big\{|P'(t)|+\|\Psi^\dagger\cdot P_z'(t)\|_{\mathcal{G}^{\delta_0,1/2}}+\|\Psi^\dagger\cdot P_v'(t)\|_{\mathcal{G}^{\delta_0,1/2}}\big\}&\lesssim_\delta\eps_1^{3/2},
\end{split}
\end{equation}
for any Gevrey function $\Psi_1$ supported in $\big[\ubv/7,7\obv\big]$ and satisfying $\big\|e^{\langle\xi\rangle^{2/3}}\widetilde{\Psi_1}(\xi)\big\|_{L^\infty}\lesssim 1$.

\end{proposition}

The proof of Proposition \ref{MainBootstrap} is the main part of this paper, and covers sections \ref{firstbounds}--\ref{coimprov1}. The basic idea is to differentiate the energy functionals $\mathcal{E}_F(t)$, $\mathcal{E}_{V_\ast}(t)$, $\mathcal{E}_{\varrho_\ast}(t)$, and $\mathcal{E}_{W_\ast}(t)$,  use the evolution equations \eqref{Pef}--\eqref{PeH}, and estimate the increments. In proving these estimates it is important that the weights $A_k,A_{NR},A_R$ are decreasing in time, therefore generating space-time integrals with the right sign. 

On the other hand, to estimate the energy functional $\mathcal{E}_\phi(t)$ and the corresponding space-time integral $\mathcal{B}_\phi(t)$ we use elliptic estimates and the identity \eqref{Pephi}.

To apply this scheme we often need to estimate weighted products and paraproducts of functions. In most cases we use the following general lemma, see \cite[Lemma 8.1]{IOJI} for the proof.

\begin{lemma}\label{Multi0}
(i) Assume that $m,m_1,m_2:\R\to\C$ are symbols satisfying
\begin{equation}\label{TLX5}
|m(\xi)|\leq |m_1(\xi-\eta)|\,|m_2(\eta)|\{\langle\xi-\eta\rangle^{-2}+\langle\eta\rangle^{-2}\}
\end{equation}
for any $\xi,\eta\in\R$. If $M, M_1, M_2$ are the operators defined by these symbols then
\begin{equation}\label{TLX6}
\|M(gh)\|_{L^2(\R)}\lesssim \|M_1g\|_{L^2(\R)}\|M_2h\|_{L^2(\R)}.
\end{equation}

(ii) Similarly, if $m,m_2:\Z\times\R\to\C$ and $m_1:\R\to\C$ are symbols satisfying
\begin{equation}\label{TLX5.1}
|m(k,\xi)|\leq |m_1(\xi-\eta)|\,|m_2(k,\eta)|\{\langle\xi-\eta\rangle^{-2}+\langle k,\eta\rangle^{-2}\}
\end{equation}
for any $\xi,\eta\in\R$, $k\in\Z$, and $M, M_1, M_2$ are the operators defined by these symbols, then
\begin{equation}\label{TLX6.1}
\|M(gh)\|_{L^2(\mathbb{T}\times\R)}\lesssim \|M_1g\|_{L^2(\R)}\|M_2h\|_{L^2(\mathbb{T}\times\R)}.
\end{equation}

(iii) Finally, assume that $m,m_1,m_2:\Z\times\R\to\C$ are symbols satisfying
\begin{equation}\label{TLX5.2}
|m(k,\xi)|\leq |m_1(k-\ell,\xi-\eta)|\,|m_2(\ell,\eta)|\{\langle k-\ell,\xi-\eta\rangle^{-2}+\langle \ell,\eta\rangle^{-2}\}
\end{equation}
for any $\xi,\eta\in\R$, $k,\ell\in\Z$. If $M, M_1, M_2$ are the operators defined by these symbols, then
\begin{equation}\label{TLX6.2}
\|M(gh)\|_{L^2(\mathbb{T}\times\R)}\lesssim \|M_1g\|_{L^2(\mathbb{T}\times\R)}\|M_2h\|_{L^2(\mathbb{T}\times\R)}.
\end{equation}
\end{lemma}

To apply Lemma \ref{Multi0} we need good bounds on products of weights. In our situation, some of the required bilinear weighted bounds have already been proved in \cite[Sections 7,8] {IOJI}, which is the main reason we use exactly the same main weights as in \cite{IOJI}.

%%%%%%%%%%%%%%%%%%%%%%%%%%%%%%%%%
%
%Section: Proof of the main theorem
%
%%%%%%%%%%%%%%%%%%%%%%%%%%%%%%%%%

\section{Proofs of the main theorems}\label{mainProof}

In this section we show how to use Proposition \ref{MainBootstrap} to prove our main theorems in section \ref{Introd}.

\subsection{Gevrey spaces}\label{appendix}

We review first some general properties of the Gevrey spaces of functions. The two lemmas we state here are useful to show that Theorem \ref{Thm} can be deduced from the main bootstrap estimates in Proposition \ref{MainBootstrap}. 

We start with a characterization of the Gevrey spaces on the physical side. See Lemma A2 in \cite{IOJI} for the elementary proof.

\begin{lemma}\label{lm:Gevrey}
(i) Suppose that $0<s<1$, $K>1$, and $f\in C^{\infty}(\mathbb{T}\times \mathbb{R})$ with ${\rm supp}\,f\subseteq \mathbb{T}\times[-L,L]$ satisfies the bounds
\begin{equation}\label{growth}
\big|D^{\alpha}f(x)\big|\leq K^{m}(m+1)^{m/s},
\end{equation}
for all integers $m\ge 0$ and multi-indeces $\alpha$ with $|\alpha|=m$. Then
\begin{equation}\label{gevreyP}
\big|\widetilde{f}(k,\xi)\big|\lesssim_{K,s} Le^{-\mu|k,\xi|^s},
\end{equation}
for all $k\in\mathbb{Z}, \xi\in \R$ and some $\mu=\mu(K,s)>0$.

(ii) Conversely, assume that $\mu>0$, $s\in(0,1)$, and $f:\T\times\R\to\mathbb{C}$ satisfies
\begin{equation}\label{eq:fouP}
\big\|f\big\|_{\mathcal{G}^{\mu,s}(\mathbb{T}\times \mathbb{R})}\leq 1.
\end{equation}
Then there is $K=K(s,\mu)>1$ such that, for any $m\geq 0$ and all multi-indices $\alpha$ with $|\alpha|\leq m$,
\begin{equation}\label{eq:four}
\left|D^{\alpha}f(x)\right|\lesssim_{\mu,s} K^m(m+1)^{m/s}.
\end{equation}
\end{lemma}

The physical space characterization of Gevrey functions is useful when studying compositions and algebraic operations of functions. For any domain $D\subseteq\T\times\R$ (or $D\subseteq\R$) and parameters $s\in(0,1)$ and $M\geq 1$ we define the spaces
\begin{equation}\label{Gevr2}
\widetilde{\mathcal{G}}^{s}_M(D):=\big\{f:D\to\mathbb{C}:\,\|f\|_{\widetilde{\mathcal{G}}^{s}_M(D)}:=\sup_{x\in D,\,m\geq 0,\,|\alpha|\leq m}|D^\alpha f(x)|M^{-m}(m+1)^{-m/s}<\infty\big\}.
\end{equation}

\begin{lemma}\label{GPF} (i) Assume  $s\in(0,1)$, $M\geq 1$, and $f_1,f_2\in \widetilde{\mathcal{G}}^{s}_M(D)$. Then $f_1f_2\in\widetilde{\mathcal{G}}^{s}_{M'}(D)$ and
\begin{equation*}
\|f_1f_2\|_{\widetilde{\mathcal{G}}^{s}_{M'}(D)}\lesssim \|f_1\|_{\widetilde{\mathcal{G}}^{s}_{M}(D)}\|f_2\|_{\widetilde{\mathcal{G}}^{s}_{M}(D)}
\end{equation*}
for some $M'=M'(s,M)\geq M$. Similarly, if $f_1\geq 1$ in $D$ then $\|(1/f_1)\|_{\widetilde{\mathcal{G}}^{s}_{M'}(D)}\lesssim 1$.

(ii) Suppose $s\in(0,1)$, $M\geq 1$, $I_1\subseteq \R$ is an interval, and $g:\mathbb{T}\times I_1\to \mathbb{T}\times I_2$ satisfies
\begin{equation}\label{gbo1}
|D^\alpha g(x)|\leq M^m(m+1)^{m/s}\qquad \text{ for any }x\in\T\times I_1,\,m\geq 1,\text{ and }|\alpha|\in [1,m].
\end{equation}
If $K\geq 1$ and $f\in \widetilde{G}^s_K(\T\times I_2)$ then $f\circ g\in \widetilde{G}^s_L(\T\times I_1)$ for some $L=L(s,K,M)\geq 1$ and
\begin{equation}\label{Ffgc}
\left\|f\circ g\right\|_{\widetilde{G}^s_L(\T\times I_1)}\lesssim_{s,K,M} \left\|f\right\|_{\widetilde{G}^s_K(\T\times I_2)}.
\end{equation}

(iii) Assume $s\in(0,1)$, $L\in[1,\infty)$, $I,J\subseteq\mathbb{R}$ are open intervals, and $g:I\to J$ is a smooth bijective map satisfying, for any $m\geq 1$,
\begin{equation}\label{gbo2}
|D^\alpha g(x)|\leq L^m(m+1)^{m/s}\qquad \text{ for any }x\in I\text{ and }|\alpha|\in [1,m].
\end{equation}
If $|g'(x)|\geq \rho>0$ for any $x\in I$ then the inverse function $g^{-1}:J\to I$ satisfies the bounds
\begin{equation}\label{gbo2.1}
|D^\alpha (g^{-1})(x)|\leq M^m(m+1)^{m/s}\qquad \text{ for any }x\in J\text{ and }|\alpha|\in [1,m],
\end{equation}
for some constant $M=M(s,L,\rho)\geq L$.
\end{lemma}

Lemma \ref{GPF}, which is used only in this section to pass from Proposition \ref{MainBootstrap} to Theorem \ref{Thm}, can be proved by elementary means using just the definition \eqref{Gevr2}. See also Theorem 6.1 and Theorem 3.2 of \cite{Yamanaka} for more general estimates on functions in Gevrey spaces.

\subsubsection{Gevrey cutoff functions}\label{GevCut} Using Lemma \ref{lm:Gevrey}, one can construct explicit cutoff functions in Gevrey spaces. For $a>0$ let
\begin{equation}\label{gev1}
\psi_a(x):=\begin{cases}
e^{-[1/x^a+1/(1-x)^a]}&\quad\text{ if }x\in[0,1],\\
0&\quad\text{ if }x\notin[0,1].
\end{cases}
\end{equation}
Clearly $\psi_a$ are smooth functions on $\R$, supported in the interval $[0,1]$ and independent of the periodic variable. It is easy to verify that $\psi_a$ satisfies the bounds \eqref{growth} for $s:=a/(a+1)$. Thus
\begin{equation}\label{gev2}
|\widetilde{\psi_a}(\xi)|\lesssim e^{-\mu|\xi|^{a/(a+1)}}\qquad\text{ for some }\mu=\mu(a)>0.
\end{equation} 

One can also construct compactly supported Gevrey cutoff functions which are equal to $1$ in a given interval. Indeed, for any $\rho\in[9/10,1)$, the function
\begin{equation}\label{gev3}
\psi'_{a,\rho}(x):=\frac{\psi_a(x)}{\psi_a(x)+\psi_a(x-\rho)+\psi_a(x+\rho)}
\end{equation}
is smooth, non-negative, supported in $[0,1]$, and equal to $1$ in $[1-\rho,\rho]$. Moreover, it follows from Lemma \ref{lm:Gevrey} (i) that $|\widetilde{\psi'_{a,\rho}}(\xi)|\lesssim e^{-\mu|\xi|^{a/(a+1)}}$ for some $\mu=\mu(a,\rho)>0$.

\subsection{Local regularity} As a starting point for the proofs of the main theorems we will also need the following local regularity lemma.

\begin{lemma}\label{lm:persistenceofhigherregularity}
(i) Assume that $\vartheta\in(0,1/4]$, and $\omega'_0\in H^6(\T\times\R)$ has the property that $\mathrm{supp}\,\omega'_0\subseteq \T\times [\vartheta,1/\vartheta]$. Then there is $T_0=T_0(\vartheta,\|\omega'_0\|_{H^6},\kappa)>0$ and a unique solution $\omega'\in C([0,T_0]:H^6)$ of the system \eqref{PEu}--\eqref{PtS2} satisfying, for any $t\in[0,T_0]$,
 \begin{equation}\label{SupA1}
 \mathrm{supp}\,\omega'(t)\subseteq \T\times [\vartheta/2,1-\vartheta/2].
 \end{equation}

(ii) Assume $T\geq 0$ and $\omega'\in C([0,T]:H^6)$ is a solution of the system \eqref{PEu}--\eqref{PtS2} satisfying the support assumption \eqref{SupA1} for any $t\in[0,T]$.  Assume that $s\in[1/4,3/4]$, $\lambda_0\in(0,1)$, and
\begin{equation}\label{ini1}
A:=\left\|\langle \nabla\rangle^3\,\omega'_0\right\|_{\mathcal{G}^{\lambda_0,s}}<\infty,\qquad B:=\int_0^t(\|\omega'(s)\|_{H^6}+1)\,ds<\infty.
\end{equation}
Then, for any cutoff function $\Upsilon\in\mathcal{G}^{1,3/4}$ with ${\rm supp}\,\Upsilon\subseteq [\vartheta/20, 20/\vartheta]$ we have
\begin{equation}\label{ini2}
\begin{split}
\left\|\langle\nabla\rangle^5(\Upsilon\psi')(t)\right\|_{\mathcal{G}^{\lambda(t),s}}+\left\|\langle\nabla\rangle^3\,\omega'(t)\right\|_{\mathcal{G}^{\lambda(t),s}}&\lesssim \exp(C_{\ast}B)\|\langle\nabla\rangle^3\omega'_0\|_{\mathcal{G}^{\lambda_0,s}},\\
\end{split}
\end{equation}
for any $t\in[0,T]$, where $C_\ast=C_\ast(\vartheta,\kappa)$ is a suitable large constant and
\begin{equation}\label{fdLam1}
\lambda(t):=\lambda_0\exp{\big\{-C_{\ast}A\,t\exp(C_\ast B)-C_{\ast}t\big\}}.
\end{equation}
\end{lemma}

In other words, the solution extends smoothly as long as its support in $r$ is bounded away from $0$. Lemma \ref{lm:persistenceofhigherregularity} can be proved by following the argument in the appendix of \cite{IOJI}, and we omit its proof.  

We note that an important aspect of the regularity theory for Euler equations in Gevrey spaces is the shrinking in time, at a fast rate, of the radius of convergence (the function $\lambda(t)$ in Lemma \ref{lm:persistenceofhigherregularity}). See, for example the work of Kukavica--Vicol \cite{Vicol}, Theorem 6.1 and Remark 6.2, for more general well-posedness results of this type in Gevrey spaces.

\subsection{Proofs of the main theorems} We are now ready to  proceed to the proofs.

\begin{proof}[Proof of Theorem \ref{Thm}] For the purpose of proving continuity in time of the energy functionals $\mathcal{E}_g$ and $\mathcal{B}_g$, we make the {\it{a priori} } assumption that $\omega_0\in\mathcal{G}^{1,2/3}$. Indeed, we may replace $\omega_0'$ with $(\omega_0')^n:=\omega_0'\ast K_n$, where $K_n\in\mathcal{G}^{1,2/3}$ is an approximation of the identity sequence and $\mathrm{supp}\,K_n\subseteq [-2^{-n},2^{-n}]$ (see  subsection \ref{GevCut} for an axplicit construction of such kernels). Then we prove uniform bounds in $n$ on the solutions generated by the mollified data $(\omega_0')^n$ , and finally pass to the limit $n\to\infty$ on any finite time interval $[0,T]$.

(i) Given small data $\omega_0$ satisfying \eqref{Eur0} we apply first Lemma \ref{lm:persistenceofhigherregularity}. Therefore $\omega',\psi'\in C([0,3]:\G^{\lambda_1,2/3})$, $\lambda_1>0$, satisfy the quantitative estimates
\begin{equation}\label{smallnessofomega}
\sup_{t\in[0,3]}\big\|\langle k,\xi\rangle^2e^{\beta'_0\langle k,\xi\rangle^{1/2}}\widetilde{(\Upsilon\psi')}(t,k,\xi)\big\|_{L^2_{k,\xi}}+\sup_{t\in[0,3]}\big\|e^{\beta'_0\langle k,\xi\rangle^{1/2}}\widetilde{\omega'}(t,k,\xi)\big\|_{L^2_{k,\xi}}\lesssim \eps,
\end{equation}
for some $\beta'_0=\beta'_0(\beta_0,\vartheta_0,\kappa)>0$, where $\Upsilon\in\mathcal{G}^{1,3/4}$ is a fixed Gevrey cutoff function supported in $[\vartheta_0/20, 20/\vartheta_0]$ and equal to $1$ in $[\vartheta_0/10, 10/\vartheta_0]$. Using the formula \eqref{prv} and Lemma \ref{lm:Gevrey}, it follows that, for some constant $K_1=K_1(\beta_0,\vartheta_0,\kappa)$,
\begin{equation}\label{proo1}
|D^\alpha_r [v(t,r)-\kappa/(2\pi r^2)]|\leq \eps K_1^m(m+1)^{2m},
\end{equation}
for any $(t,r)\in[0,3]\times[\vartheta_0/10, 10/\vartheta_0]$, $m\geq 1$, and $|\alpha|\in[1,m]$. Using now Lemma \ref{GPF} (iii) and letting 
$\mathcal{Y}(t,v)$ denote the inverse of the function $r\to v(t,r)$, we have
\begin{equation}\label{proo2}
|D^\alpha_v \mathcal{Y}(t,v)|\leq K_2^m(m+1)^{2m},
\end{equation}
for any $(t,v)\in[0,3]\times[\ubv/12,12\obv]$, $m\geq 1$, and $|\alpha|\in[1,m]$. 

Recall the formulas (see Proposition \ref{ChangedEquations})
\begin{equation}\label{prok1}
\begin{split}
&F(t,z,v)=\omega'(t,z+tv,\mathcal{Y}(t,v)),\qquad\phi(t,z,v)=\psi'(t,z+tv,\mathcal{Y}(t,v)),\\
&V'(t,v)=\partial_rv(t,\mathcal{Y}(t,v)),\qquad \varrho(t,v)=1/\mathcal{Y}(t,v).
\end{split}
\end{equation}
Using these identities, the bounds \eqref{smallnessofomega}--\eqref{proo2}, Lemma \ref{lm:Gevrey} and Lemma \ref{GPF} (ii), we have
\begin{equation}\label{prok2}
\sup_{t\in[0,3]}\big\|e^{\beta''_0\langle k,\xi\rangle^{1/2}}\widetilde{F}(t,k,\xi)\big\|_{L^2_{k,\xi}}+\sup_{t\in[0,3]}\big\|e^{\beta''_0\langle k,\xi\rangle^{1/2}}\widetilde{\Psi\phi}(t,k,\xi)\big\|_{L^2_{k,\xi}}\lesssim \eps,
\end{equation}
for some constant $\beta''_0=\beta''_0(\beta_0,\vartheta_0,\kappa)>0$. Moreover, using \eqref{proo1} and Lemma \ref{GPF} we see that
\begin{equation}\label{proo2.1}
\big|D^\alpha_v [v-\kappa/(2\pi \mathcal{Y}(t,v)^2)]\big|\leq \eps K_3^m(m+1)^{2m},
\end{equation}
for some constant $K_3=K_3(\beta_0,\vartheta_0,\kappa)\geq 1$ and for any $(t,v)\in[0,3]\times[\ubv/12,12\obv]$, $m\geq 1$, and $|\alpha|\in[1,m]$. Since $\varrho=1/\mathcal{Y}$, it follows from \eqref{proo2}, \eqref{proo2.1}, and Lemma \ref{GPF} (i) that $\|\varrho_\ast(t)\|_{\widetilde{G}^s_{K_4}(\T\times[\ubv/12,12\obv])}\lesssim \eps$, for some $K_4=K_4(\beta_0,\vartheta_0,\kappa)\geq 1$. One can bound $V_\ast$ in a similar way, using the \eqref{proo1}, \eqref{proo2}, and Lemma \ref{GPF}, and then one can also bound $W_\ast$ using the formula $W_\ast=-V_\ast+\langle F\rangle\varrho$. Finally, using again Lemma \ref{lm:Gevrey}, we have
\begin{equation}\label{prok2.1}
\big\|e^{\beta'''_0\langle \xi\rangle^{1/2}}\widetilde{V_\ast}(t,\xi)\big\|_{L^2_{\xi}}+\big\|e^{\beta'''_0\langle \xi\rangle^{1/2}}\widetilde{(\Psi^\dagger\varrho)}(t,\xi)\big\|_{L^2_{\xi}}+\big\|e^{\beta'''_0\langle \xi\rangle^{1/2}}\widetilde{W_\ast}(t,\xi)\big\|_{L^2_{\xi}}\lesssim \eps,
\end{equation}
for some constant $\beta''_0=\beta''_0(\beta_0,\vartheta_0,\kappa)>0$, for any $t\in[0,3]$. The desired bounds \eqref{boot1} follow from \eqref{prok2} and \eqref{prok2.1} provided that $\eps_1\approx \eps^{2/3}$, see \eqref{reb11}-\eqref{reb13}. 

Assume now that the solution $\omega'$ satisfies the bounds in the hypothesis of Proposition \ref{MainBootstrap} on a given interval $[0,T]$, $T\geq 3$. We would like to show that the support of $\omega'(t)$ is contained in $\T\times\big[2\vartheta_0/3,3/(2\vartheta_0)\big]$ and that $\|\langle\omega'\rangle(t)\|_{H^{10}}\lesssim\eps_1$ for any $t\in[0,T]$. Indeed, we have
\begin{equation*}
\langle\omega'\rangle(t,r)=F(t,v(t,r)),
\end{equation*}
and the bound $\|\langle\omega'\rangle(t)\|_{H^{10}}\lesssim\eps_1$ follows from \eqref{boot2}. For the support conclusion, we notice that only transportation in the $r$ direction, given by the terms $(P'(t),e_r)\,\partial_r\omega'$ and $\partial_{\theta}\psi'\partial_r\omega'/r$, could enlarge the support of $\omega'$ in $r$. 
Notice that, on the support of $\omega'$, 
\begin{equation}\label{uyM}
(\partial_{\theta}\psi')(t,\theta,r)=\partial_zP_{\neq0}\big(\Psi\phi\big)(t,\theta-tv(t,r),v(t,r)).
\end{equation}
Using the bound on $\mathcal{E}_{\phi}$ from \eqref{boot3}, we can bound, for all $t\in[0,T]$,
\begin{equation}\label{integuy}
\sup_{(\theta,r)\in \mathbb{T}\times [\vartheta_0/2,2/\vartheta_0]}\big|\partial_{\theta}\psi'(t,\theta,r)\big|\lesssim \epsilon_1\langle t\rangle^{-2}.
\end{equation}
Moreover, using the bounds \eqref{boot4} on $P'_v(t)$, we see the total transportation in $r$ due to the term $(P'(t),e_r)\,\partial_r\omega'$ is of the size $O(\epsilon_1)$.
Since ${\rm supp}\,\omega'(0)\subseteq\mathbb{T}\times[\vartheta_0,1/\vartheta_0]$, we conclude that ${\rm supp}\,\omega'(t)\subseteq \mathbb{T}\times\big[2\vartheta_0/3,3/(2\vartheta_0)\big]$ for any $t\in[0,T]$, 
as long as $\epsilon_1$ is sufficiently small.  

To summarize, we can now use a simple continuity argument to show that if $\omega_0'\in\G^{1,2/3}$ has compact support in $\T\times [\vartheta_0,1/\vartheta_0]$ and satisfies the assumptions \eqref{Eur0}, then there is a global solution $\omega'\in C([0,\infty):\G^{1,3/5})$ of the system \eqref{PEu}--\eqref{PtS2}, which has compact support in $[\vartheta_0/2,2/\vartheta_0]$ and satisfies $\|\langle\omega'\rangle(t)\|_{H^{10}}\lesssim\eps_1$  for all $t\in[0,\infty)$. Moreover
\begin{equation}\label{proo6}
\sum_{g\in\{V_\ast,\varrho_\ast,W_\ast\}}\mathcal{E}_g(t)\leq\eps_1^2\quad\text{ and }\quad\sum_{g\in\{F,\phi\}}\mathcal{E}_g(t)\lesssim_{\delta}\eps_1^3\quad \text{ for any }t\in[0,\infty).
\end{equation}

(ii) Since $\eps_1^{3/2}\approx \eps$, $A_k(t,\xi)\ge e^{1.1\delta_0\langle k,\xi\rangle^{1/2}}$ and $A_R(t,\xi)\ge A_{NR} (t,\xi)\geq e^{1.1\delta_0\langle\xi\rangle^{1/2}}$ for any $(t,k,\xi)\in[0,\infty)\times\Z\times\R$, it follows from \eqref{proo6} that 
\begin{equation}\label{uniformf}
\big\|F\big\|_{\mathcal{G}^{\delta_0,1/2}}+\langle t\rangle^2\big\|\mathbb{P}_{\neq 0}(\Psi\phi)\big\|_{\mathcal{G}^{\delta_0,1/2}}\lesssim_\delta\epsilon.
\end{equation}
Since $\mathcal{Y}=1/\varrho$ it follows from \eqref{proo6} that, for any $a,b\in[-2,2]\cap\Z$,
\begin{equation}\label{wsx1}
\|\Psi_1\cdot (\varrho)^a(V')^b(t)\|_{\mathcal{G}^{\delta_0,1/2}}+\|\Psi_1\cdot \mathcal{Y}(t)\|_{\mathcal{G}^{\delta_0,1/2}}\lesssim 1
\end{equation}
for any Gevrey function $\Psi_1$ supported in $\big[\ubv/7,7\obv\big]$ and satisfying $\big\|e^{\langle\xi\rangle^{2/3}}\widetilde{\Psi_1}(\xi)\big\|_{L^\infty}\lesssim_\delta 1$.

Recall that $\partial_rv(t,r)=V'(t,v(t,r))$, $\partial_tv(t,r)=\dot{V}(t,v(t,r))$, and $v=\mathcal{Y}^{-1}$. Therefore, using Lemmas \ref{lm:Gevrey}--\eqref{GPF} and the estimates \eqref{boot4} and \eqref{wsx1} we have
\begin{equation}\label{DdtvM}
\begin{split}
&\big\|v(t)\big\|_{\widetilde{\mathcal{G}}^{1/2}_{M_1}([\vartheta_0/6,6\vartheta_0])}+\big\|(\partial_rv)(t)\big\|_{\widetilde{\mathcal{G}}^{1/2}_{M_1}([\vartheta_0/6,6\vartheta_0])}\lesssim 1,\\
&\big\|(\partial_tv)(t)\big\|_{\widetilde{\mathcal{G}}^{1/2}_{M_1}([\vartheta_0/6,6\vartheta_0])}\lesssim_\delta \epsilon\langle t\rangle^{-2}.
\end{split}
\end{equation}
for any $t\geq 1$ and some $M_1=M_1(\beta_0,\vartheta_0,\kappa)\geq 1$. Moreover, recalling the equation \eqref{Pef} for $F$, and using the bounds \eqref{boot4} and \eqref{uniformf}--\eqref{wsx1} we have, for any $t\geq 1$,
\begin{equation}\label{QDtf}
\big\|\partial_tF(t)\big\|_{\mathcal{G}^{\delta_0/2,1/2}}\lesssim_\delta \epsilon\langle t\rangle^{-2}.
\end{equation}

Using the definitions, we notice that
$$\omega'(t,\theta+\kappa t/(2\pi r^2)+\Phi(t,r),r)=\omega'(t,\theta+tv(t,r),r)=F(t,\theta,v(t,r)).$$
Therefore
\begin{equation*}
\frac{d}{dt}\{\omega'(t,\theta+\kappa t/(2\pi r^2)+\Phi(t,r),r)\}=(\partial_tF)(t,\theta,v(t,r))+(\partial_tv)(t,r)\cdot (\partial_rF)(t,\theta,v(t,r)),
\end{equation*}
and using \eqref{DdtvM}--\eqref{QDtf} and Lemmas \ref{lm:Gevrey}--\ref{GPF} we have
\begin{equation*}
\Big\|\frac{d}{dt}\{\omega'(t,\theta+\kappa t/(2\pi r^2)+\Phi(t,r),r)\}\Big\|_{\mathcal{G}^{\delta_1,1/2}}\lesssim_\delta \epsilon\langle t\rangle^{-2},
\end{equation*}
for any $t\geq 1$ and some $\delta_1=\delta_1(\beta_0,\vartheta_0,\kappa)>0$. The existence of the function $\Omega_\infty$ and the bounds \eqref{convergence1} follow. The existence of the limit point $P_\infty$ and the bounds \eqref{convergence2} follow from the bound $|P'(t)|\lesssim_\delta\eps e^{-0.1\delta_0t^{1/2}}$ in \eqref{boot4}.

(iii) We prove first that, for any $t\geq 1$,
\begin{equation}\label{wsx2}
\big\|\partial_t\langle\partial_r\psi'\rangle(t)\big\|_{\mathcal{G}^{\delta_2,1/2}}\lesssim_\delta \epsilon\langle t\rangle^{-3},
\end{equation}
for some $\delta_2=\delta_2(\beta_0,\vartheta_0,\kappa)>0$. Indeed, starting from \eqref{psiInf2}, we have 
\begin{equation}\label{wsx3}
\partial_r(r\langle\partial_t\partial_r\psi'\rangle)=r\langle\partial_t\omega'\rangle.
\end{equation}
Using now \eqref{DNC} we have
\begin{equation*}
\langle\partial_t\omega'\rangle(t,r)=\langle\partial_tF\rangle(t,v(t,r))+\dot{V}(t,v(t,r))\langle F\rangle(t,v(t,r)).
\end{equation*}
Using now Lemma \ref{GPF} (ii), the bounds on $\langle\partial_tF+\dot{V}F\rangle$ in \eqref{boot4}, and the bounds \eqref{DdtvM} on $v$, we have $\|\langle\partial_t\omega'\rangle(t)\|_{\mathcal{G}^{\delta_2,1/2}}\lesssim_\delta \epsilon\langle t\rangle^{-3}$. Moreover, the function $\partial_t\langle\partial_r\psi'\rangle$ is supported in $[\vartheta_0/2,2/\vartheta_0]$, due to \eqref{psiInf4}, and the bounds \eqref{wsx2} follow from \eqref{wsx3} and the uncertainty principle.

The existence of the limit function $u'_\infty$ and the estimates \eqref{convergenceofux} follow from \eqref{wsx2}, provided that $\beta_1$ is sufficiently small. The identities \eqref{AsymPhi2} follow from \eqref{psiInf2} and \eqref{psiInf4}. 

To prove the decay estimates \eqref{convergencetomean} and \eqref{convergenceuy} for $ u'_{\theta}-\langle u'_{\theta}\rangle$ and $u'_r$ we use properties of the stream function $\psi'$. The starting point is the equation \eqref{St}, written in the form
\begin{equation*}
\partial_r^2\psi'+\partial_r\psi'/r+\partial_{\theta}^2\psi'/r^2=\omega'(t,\theta,r)=F(t,\theta-tv(t,r),v(t,r)),
\end{equation*}
for $(\theta,r)\in\T\times(0,\infty)$. Let $\psi'_k$ and $F_k$ denote the $k$--th Fourier modes of $\psi'$ and $F$ in $\theta$. Thus
\begin{equation}\label{formulapsikMain}
\psi'_k(t,r)=\int_0^{\infty}G_k(r,\rho)\,F_k(t,v(t,\rho))\,e^{-iktv(t,\rho)}\,d\rho,
\end{equation}
where
\begin{equation}\label{G.krepeat}
G_k(r,\rho):=\left\{\begin{array}{ll}
                         -\frac{\rho}{2|k|}\big(\frac{r}{\rho}\big)^{|k|}&{\rm for}\,\,r<\rho;\\
                          -\frac{\rho}{2|k|}\big(\frac{\rho}{r}\big)^{|k|}&{\rm for}\,\,r>\rho
                        \end{array}\right.
\end{equation}
is the associated Green function for the operator $\partial_r^2+\partial_r/r-k^2/r^2$. See \eqref{phik}--\eqref{phi.k1} for more details. Moreover
\begin{equation}\label{wsx7}
\begin{split}
r\big|u'_{r}(t,\theta,r)\big|&\lesssim\sum_{k\neq 0}|k|\big|\psi'_k(t,r)\big|\lesssim \sup_{k\neq0}|k|^3\big|\psi'_k(t,r)\big|,\\
\big|u'_{\theta}(t,\theta,r)-\langle u'_{\theta}\rangle(t,r)\big|&\lesssim\sum_{k\neq0}\big|\partial_r\psi'_k(t,r)\big|\lesssim \sup_{k\neq0}|k|^2\big|(\partial_r\psi'_k)(t,r)\big|.
\end{split}
\end{equation}

We can now integrate by parts in $\rho$ twice in the identities (\ref{formulapsikMain}), and use the formulas \eqref{G.krepeat} and the smoothness of the functions $F$ and $v$, see \eqref{uniformf} and \eqref{DdtvM}. Thus
\begin{equation}\label{pduy}
\big|\psi'_k(t,r)\big|\lesssim_{\delta}\epsilon\langle t\rangle^{-2}|k|^{-4}r.
\end{equation}
Similarly, we can differentiate in $r$ and integrate by parts in $\rho$ once in (\ref{formulapsikMain}) to see that $\big|(\partial_r\psi'_k)(t,r)\big|\lesssim_{\delta}\epsilon\langle t\rangle^{-1}|k|^{-4}$. The desired bounds (\ref{convergencetomean}) and (\ref{convergenceuy}) follow from \eqref{wsx7}, which completes the proof of Theorem \ref{Thm}.
\end{proof}

\section{Bounds on the functions $V'$, $V''$, $\varrho$, $\dot{V}$, $\mathbb{P}_{\neq0}\phi$, $P'_z$, and $P'_v$}\label{firstbounds}

In this section we prove suitable bounds on many of the functions defined in Proposition \ref{ChangedEquations}. In most cases we apply the definitions, the bootstrap assumptions \eqref{boot2}, and Lemma \ref{Multi0}. To apply this lemma we need suitable bilinear weighted bounds, which we collect in Lemma \ref{lm:Multi} below; see Lemmas 8.2 and 8.3 in \cite{IOJI} for the complete proofs.

\begin{lemma}\label{lm:Multi}
(i) For any $t\in[1,\infty)$, $\alpha\in[0,4]$, $\xi,\eta\in\R$, and $Y\in\{NR,R\}$ we have
\begin{equation}\label{TLX4}
\langle\xi\rangle^{-\alpha}A_Y(t,\xi)\lesssim_\delta \langle\xi-\eta\rangle^{-\alpha}A_Y(t,\xi-\eta)\langle\eta\rangle^{-\alpha}A_Y(t,\eta)e^{-(\lambda(t)/20)\min(\langle\xi-\eta\rangle,\langle \eta\rangle)^{1/2}}
\end{equation}
and
\begin{equation}\label{DtVMulti}
\big|(\dot{A}_Y/A_Y)(t,\xi)\big|\lesssim_\delta \big\{\big|(\dot{A}_Y/A_Y)(t,\xi-\eta)\big|+\big|(\dot{A}_Y/A_Y)(t,\eta)\big|\big\}e^{4\sqrt\delta\min(\langle\xi-\eta\rangle,\langle \eta\rangle)^{1/2}}.
\end{equation}

(ii) For any $t\in[1,\infty)$, $\xi,\eta\in\R$, and $k\in\mathbb{Z}$ we have
\begin{equation}\label{TLX7}
A_k(t,\xi)\lesssim_\delta A_R(t,\xi-\eta)A_k(t,\eta)e^{-(\lambda(t)/20)\min(\langle\xi-\eta\rangle,\langle k,\eta\rangle)^{1/2}}
\end{equation}
and
\begin{equation}\label{vfc30.7}
\big|(\dot{A}_k/A_k)(t,\xi)\big|\lesssim_\delta \big\{\big|(\dot{A}_R/A_R)(t,\xi-\eta)\big|+\big|(\dot{A}_k/A_k)(t,\eta)\big|\big\}e^{12\sqrt\delta\min(\langle\xi-\eta\rangle,\langle k,\eta\rangle)^{1/2}}.
\end{equation}
\end{lemma}

For simplicity of notation, for any $f\in C([1,T]:H^{4}(\mathbb{R}))$ we define
\begin{equation}\label{rew0}
\|f\|^2_{R}:=\sup_{t\in [1,T]}\Big\{\int_{\R}A_R^2(t,\xi)\big|\widetilde{f}(t,\xi)\big|^2\,d\xi+\int_1^t\int_{\R}|\dot{A}_R(s,\xi)|A_R(s,\xi)\big|\widetilde{f}(s,\xi)\big|^2\,d\xi ds\Big\}.
\end{equation}

We prove first bounds on the functions $V'$, $V''$, and $\varrho$.

\begin{lemma}\label{nar8}
(i) Assume that $f,g\in C([1,T]:H^{4}(\mathbb{R}))$ then
\begin{equation}\label{rew1}
\|fg\|_R\lesssim_\delta \|f\|_R\|g\|_R.
\end{equation}

(ii) As a consequence, if $\Psi_1,\Psi_2$ are Gevrey cutoff functions supported in $\big[\ubv/7,7\obv\big]$ and satisfying $\big\|e^{\langle\xi\rangle^{2/3}}\widetilde{\Psi_a}(\xi)\big\|_{L^\infty}\lesssim 1$, $a\in\{1,2\}$, and 
\begin{equation}\label{rew3}
f\in\{\Psi_1\cdot(\varrho)^a(V')^b(\langle\partial_v\rangle^{-1}(\Psi_2V''))^c:\,a,b\in[-2,2]\cap\mathbb{Z},\,c\in\{0,1\}\},
\end{equation}
then for any $t\in[1,T]$ we have
\begin{equation}\label{nar4}
\begin{split}
&\int_{\R}A_R^2(t,\xi)\big|\widetilde{f}(t,\xi)\big|^2\,d\xi+\int_1^t\int_{\R}|\dot{A}_R(s,\xi)|A_R(s,\xi)\big|\widetilde{f}(s,\xi)\big|^2\,d\xi ds\lesssim 1.
\end{split}
\end{equation} 
\end{lemma}

\begin{proof} (i) The bound on the first term in the definition of the $R$-norm follows from Lemma \ref{Multi0} (i) and \eqref{TLX4} (with $Y=R$ and $\alpha=0$). To bound the second term we notice that
\begin{equation*}
\begin{split}
|\dot{A}_R(s,\xi)A_R(s,\xi)|^{1/2}\lesssim_\delta \{&|\dot{A}_R(s,\xi-\eta)A_R(s,\xi-\eta)|^{1/2}A_R(s,\eta)\\
&+A_R(s,\xi-\eta)|\dot{A}_R(s,\eta)A_R(s,\eta)|^{1/2}\}e^{-(\lambda(t)/30)\min(\langle\xi-\eta\rangle,\langle \eta\rangle)^{1/2}},
\end{split}
\end{equation*}
as a consequence of \eqref{TLX4}--\eqref{DtVMulti}. Therefore, using Lemma \ref{Multi0} (i) we estimate
\begin{equation}\label{rew3.2}
\begin{split}
\big\|\sqrt{|\dot{A}_RA_R|(s,\xi)}\widetilde{(fg)}(s,\xi)\big\|_{L^2_sL^2_{\xi}}&\lesssim_\delta \big\|\sqrt{|\dot{A}_RA_R|(s,\rho)}\widetilde{f}(s,\rho)\big\|_{L^2_sL^2_{\rho}}\big\|A_R(s,\eta)\widetilde{g}(s,\eta)\big\|_{L^\infty_sL^2_{\eta}}\\
&+\big\|A_R(s,\rho)\widetilde{f}(s,\rho)\big\|_{L^\infty_sL^2_{\rho}}\big\|\sqrt{|\dot{A}_RA_R|(s,\eta)}\widetilde{g}(s,\eta)\big\|_{L^2_sL^2_{\eta}},
\end{split}
\end{equation}
and the desired estimate follows.

(ii) Notice that the constants in the bounds \eqref{nar4} do not on $\delta$, as long as $\eps_1$ is sufficiently small depending on $\delta$. This is useful in the commutator estimates in section \ref{coimprov1}.

In view of the definitions, with $\Psi^\dagger$ as in \eqref{rec0} we have the identities
\begin{equation}\label{rew6}
\begin{split}
\Psi^\dagger(v)\varrho(t,v)&=\Psi^\dagger(v)\varrho_\ast(t,v)+\Psi^\dagger(v)\sqrt{2\pi v/\kappa},\\
\Psi^\dagger(v)V'(t,v)&=\Psi^\dagger(v)[V_{\ast}(t,v)-2v\varrho_\ast(t,v)]-\Psi^\dagger(v)2v\sqrt{2\pi v/\kappa}.
\end{split}
\end{equation}
We fix a Gevrey cutoff function $\Psi^\flat$ satisfying $\big\|e^{\langle\xi\rangle^{3/4}}\widetilde{\Psi^\flat}(\xi)\big\|_{L^\infty}\lesssim 1$, supported in $\big[\ubv/8,8\obv\big]$ and equal to $1$ in $\big[\ubv/7,7\obv\big]$. We have the identities
\begin{equation}\label{rew6.1}
\frac{\Psi^\flat(v)}{\varrho(t,v)}=\frac{\Psi^\flat(v)}{\varrho_\ast(v,t)+\sqrt{2\pi v/\kappa}}=\frac{\Psi^\flat(v)}{\sqrt{2\pi v/\kappa}}\sum_{m\geq 0}(-1)^m\frac{(\Psi^\dagger(v)\rho_\ast(t,v))^m}{(2\pi v/\kappa)^{m/2}}
\end{equation}
and
\begin{equation}\label{rew6.2}
\frac{\Psi^\flat(v)}{V'(t,v)}=\frac{\Psi^\flat(v)}{V_\ast(t,v)-2v\varrho_\ast(v,t)-2v\sqrt{2\pi v/\kappa}}=\frac{-\Psi^\flat(v)}{2v\sqrt{2\pi v/\kappa}}\sum_{m\geq 0}\frac{[V_\ast(t,v)-\Psi^\dagger(v)2v\rho_\ast(t,v)]^m}{(2v\sqrt{2\pi v/\kappa})^m}.
\end{equation}
Moreover, since $V''=\partial_v(V')^2/2$, see \eqref{PV''}, we also have the identity
\begin{equation}\label{rew6.25}
\begin{split}
\Psi^\flat(v)&V''(t,v)=\Psi^\flat(v)\cdot 6v^2(2\pi/\kappa)\\
&+(1/2)\Psi^\flat(v)\partial_v[(V_\ast(t,v)-2v\rho_\ast(t,v))(V_\ast(t,v)-2v\rho_\ast(t,v)-4v\sqrt{2\pi v/\kappa})].
\end{split}
\end{equation}

Notice that
\begin{equation}\label{rew6.5}
\int_{\R}A_R^2(t,\xi)e^{-8\delta_0\langle\xi\rangle^{1/2}}\,d\xi+\int_1^t\int_{\R}|\dot{A}_R(s,\xi)|A_R(s,\xi)e^{-8\delta_0\langle\xi\rangle^{1/2}}\,d\xi ds\lesssim 1
\end{equation}
for any $t\geq 1$. This is because the weights $A_R$ are decreasing in time and satisfy the bounds $|A_R(t,\xi)|^2\lesssim e^{4\delta_0\langle\xi\rangle^{1/2}}$. In particular, functions of the form $v^\alpha\Psi_1(v)$, $v^\alpha\Psi_2(v)$, or $v^\alpha\Psi_1(v)\Psi_2(v)$, $\alpha\in[-20,20]$, satisfy the bounds \eqref{nar4}, due to the assumptions on $\Psi_a$ and Lemma \ref{lm:Gevrey}. 

The desired bounds \eqref{nar4} follow using the algebra property proved in part (i), the bootstrap assumptions \eqref{boot2} on $V_\ast$ and $\varrho_\ast$, and the identities \eqref{rew6}--\eqref{rew6.25}, as long as $\eps_1$ is sufficiently small depending on $\delta$.
\end{proof}

We prove now bounds on the functions $\partial_z\phi$ and $\partial_v\mathbb{P}_{\neq 0}\phi$.

\begin{lemma}\label{nar13}
(i) For any $t\in[1,T]$ and $h_1\in\{(\varrho)^a(V')^b\partial_z(\Psi\phi):\,a,b\in[-2,2]\}$ we have
\begin{equation}\label{nar34}
\begin{split}
&\sum_{k\in \mathbb{Z}\setminus\{0\}}\int_{\R}A_k^2(t,\xi)\frac{k^2\langle t\rangle^4\langle t-\xi/k\rangle^4}{(|\xi/k|^2+\langle t\rangle^2)^2}\big|\widetilde{h_1}(t,k,\xi)\big|^2\,d\xi\lesssim_\delta\eps_1^2\\
&\int_1^t\sum_{k\in \mathbb{Z}\setminus\{0\}}\int_{\R}|\dot{A}_k(s,\xi)|A_k(s,\xi)\frac{k^2\langle s\rangle^4\langle s-\xi/k\rangle^4}{(|\xi/k|^2+\langle s\rangle^2)^2}\big|\widetilde{h_1}(s,k,\xi)\big|^2\,d\xi ds\lesssim_\delta\eps_1^2.
\end{split}
\end{equation}

(ii) For any $t\in[1,T]$ and $h_2\in\{(\varrho)^a(V')^b\partial_v\mathbb{P}_{\neq 0}(\Psi\phi):\,a,b\in[-2,2]\}$ we have
\begin{equation}\label{nar14} 
\begin{split}
&\sum_{k\in \mathbb{Z}\setminus\{0\}}\int_{\R}A_k^2(t,\xi)\frac{k^4\langle t\rangle^2\langle t-\xi/k\rangle^4}{(|\xi/k|^2+\langle t\rangle^2)\langle\xi\rangle^2}\big|\widetilde{h_2}(t,k,\xi)\big|^2\,d\xi\lesssim_\delta\eps_1^2\\
&\int_1^t\sum_{k\in \mathbb{Z}\setminus\{0\}}\int_{\R}|\dot{A}_k(s,\xi)|A_k(s,\xi)\frac{k^4\langle s\rangle^2\langle s-\xi/k\rangle^4}{(|\xi/k|^2+\langle s\rangle^2)\langle\xi\rangle^2}\big|\widetilde{h_2}(s,k,\xi)\big|^2\,d\xi ds\lesssim_\delta\eps_1^2.
\end{split}
\end{equation}
\end{lemma}

\begin{proof} (i) The bounds when $h_1=\partial_z(\Psi\phi)$ follow directly from the bootstrap assumptions \eqref{boot2} on $\mathcal{E}_\phi$ and $\mathcal{B}_\phi$. One could in fact get slightly stronger bounds on $\partial_z(\Psi\phi)$, but these bounds are not compatible with multiplication by $\varrho$ and $V'$.

To prove the bounds \eqref{nar34} for all $h_1$ as claimed, we prove the multiplier bounds
\begin{equation}\label{nar34.01}
\begin{split}
A_k(t,\xi)&\frac{|k|\langle t\rangle^2\langle t-\xi/k\rangle^2}{|\xi/k|^2+\langle t\rangle^2}\lesssim_\delta A_R(t,\xi-\eta)\cdot A_k(t,\eta)\frac{|k|\langle t\rangle^2\langle t-\eta/k\rangle^2}{|\eta/k|^2+\langle t\rangle^2}\cdot \{\langle\xi-\eta\rangle^{-2}+\langle k,\eta\rangle^{-2}\}
\end{split}
\end{equation}
and
\begin{equation}\label{nar34.02}
\begin{split}
\big|\dot{A}_k(t,\xi)&A_k(t,\xi)\big|^{1/2}\frac{|k|\langle t\rangle^2\langle t-\xi/k\rangle^2}{|\xi/k|^2+\langle t\rangle^2}\lesssim_\delta \left[\big|(\dot{A}_R/A_R)(t,\xi-\eta)\big|^{1/2}+\big|(\dot{A}_k/A_k)(t,\eta)\big|^{1/2}\right]\\
&\times A_R(t,\xi-\eta)\cdot A_k(t,\eta)\frac{|k|\langle t\rangle^2\langle t-\eta/k\rangle^2}{|\eta/k|^2+\langle t\rangle^2}\cdot \{\langle\xi-\eta\rangle^{-2}+\langle k,\eta\rangle^{-2}\},
\end{split}
\end{equation}
for any $t\in[1,T]$, $\xi,\eta\in\R$, and $k\in\mathbb{Z}\setminus\{0\}$. Indeed, we use Lemma \ref{lm:Multi} (ii). In addition, by considering the cases $|\xi-\eta|\leq 10|k,\eta|$ and $|\xi-\eta|\geq 10|k,\eta|$, it is easy to see that
\begin{equation}\label{nar34.03}
\frac{|k|\langle t\rangle^2\langle t-\xi/k\rangle^2}{|\xi/k|^2+\langle t\rangle^2}\lesssim_\delta \frac{|k|\langle t\rangle^2\langle t-\eta/k\rangle^2}{|\eta/k|^2+\langle t\rangle^2}\cdot e^{\delta\min(\langle\xi-\eta\rangle,\langle k,\eta\rangle)^{1/2}}
\end{equation}
for any $t\in[1,T]$, $\xi,\eta\in\R$, and $k\in\mathbb{Z}\setminus\{0\}$. The bounds \eqref{nar34.01} follow from \eqref{TLX7} and \eqref{nar34.03}, while the bounds \eqref{nar34.02} follow by multiplication from \eqref{DtVMulti}, \eqref{TLX7}, and \eqref{nar34.03}.

The bounds \eqref{nar34} now follow from \eqref{nar34.01}--\eqref{nar34.02} and Lemma \ref{Multi0} (ii). Indeed,
\begin{equation*}
\begin{split}
\Big\|A_k(t,\xi)\frac{|k|\langle t\rangle^2\langle t-\xi/k\rangle^2}{|\xi/k|^2+\langle t\rangle^2}&\mathcal{F}\{(\varrho)^a(V')^b\partial_z(\Psi\phi)\}(t,k,\xi)\Big\|_{L^2_{k,\xi}}\lesssim_\delta\Big\|A_R(t,\rho)\mathcal{F}\{(\varrho)^a(V')^b\}(t,\rho)\Big\|_{L^2_{\rho}}\\
&\times \Big\|A_k(t,\eta)\frac{|k|\langle t\rangle^2\langle t-\eta/k\rangle^2}{|\eta/k|^2+\langle t\rangle^2}\mathcal{F}\{\partial_z(\Psi\phi)\}(t,k,\eta)\Big\|_{L^2_{k,\eta}},
\end{split}
\end{equation*}
for any $t\in[1,T]$. The estimates in the first line of \eqref{nar34} follows using \eqref{nar4}. The estimates in the second line follow by a similar argument (compare with \eqref{rew3.2}).

(ii) The bounds when $h_2=\partial_v\mathbb{P}_{\neq 0}(\Psi\phi)$ follow directly from the bootstrap assumptions on $\mathcal{E}_\phi$ and $\mathcal{B}_\phi$. For the general case, we use the multiplier bounds 
\begin{equation}\label{rew19}
\begin{split}
A_k(t,\xi)&\frac{k^2\langle t\rangle\langle t-\xi/k\rangle^2}{(|\xi/k|+\langle t\rangle)\langle\xi\rangle}\lesssim_\delta A_R(t,\xi-\eta) A_k(t,\eta)\frac{k^2\langle t\rangle\langle t-\eta/k\rangle^2}{(|\eta/k|+\langle t\rangle)\langle\eta\rangle}\{\langle\xi-\eta\rangle^{-2}+\langle k,\eta\rangle^{-2}\}
\end{split}
\end{equation}
and
\begin{equation}\label{rew19.1}
\begin{split}
\big|\dot{A}_k(t,\xi)&A_k(t,\xi)\big|^{1/2}\frac{k^2\langle t\rangle\langle t-\xi/k\rangle^2}{(|\xi/k|+\langle t\rangle)\langle\xi\rangle}\lesssim_\delta \left[\big|(\dot{A}_R/A_R)(t,\xi-\eta)\big|^{1/2}+\big|(\dot{A}_k/A_k)(t,\eta)\big|^{1/2}\right]\\
&\times A_R(t,\xi-\eta) A_k(t,\eta)\frac{k^2\langle t\rangle\langle t-\eta/k\rangle^2}{(|\eta/k|+\langle t\rangle)\langle\eta\rangle}\{\langle\xi-\eta\rangle^{-2}+\langle k,\eta\rangle^{-2}\},
\end{split}
\end{equation}
for any $t\in[1,T]$, $\xi,\eta\in\R$, and $k\in\mathbb{Z}\setminus\{0\}$, which are similar to \eqref{nar34.01}--\eqref{nar34.02}. The desired bounds \eqref{nar14} follow using Lemma \ref{Multi0} (ii) and \eqref{nar4} as before.
\end{proof}

We estimate now the functions $P_z', P_v'$, which are generated by the global shift of variables due to the movement of the Dirac mass. Our main estimates are the following:

\begin{lemma}\label{P'bounnds}

For any $t\in[1,T]$ we have
\begin{equation}\label{Pboot4}
|P'(t)|\lesssim_\delta \epsilon_1 e^{-\delta_0\langle t\rangle^{1/2}}.
\end{equation}
and
\begin{equation}\label{Par1.7}
|\widetilde{\Psi^\dagger P_z'}(t,k,\xi)|+|\widetilde{\Psi^\dagger P_v'}(t,k,\xi)|\lesssim_\delta \eps_1A_1(t,t)^{-1}\mathbf{1}_{\{-1,1\}}(k)e^{-\langle\xi-kt\rangle^{3/4}}.
\end{equation}
As a consequence, if $f\in\{t\Psi_1(\varrho)^a(V')^bP_z', t\Psi_1(\varrho)^a(V')^bP_v':\,a,b\in[-2,2]\cap\mathbb{Z}\}$, where $\Psi_1$ is a Gevrey cutoff function as in Lemma \ref{nar8} (ii), then
\begin{equation}\label{Par1}
\begin{split}
&\sum_{k\in \mathbb{Z}\setminus\{0\}}\int_{\R}A_k^2(t,\xi)\frac{\langle t\rangle^2k^4\langle t-\xi/k\rangle^4}{(|\xi/k|^2+\langle t\rangle^2)\langle\xi\rangle^2}\big|\widetilde{f}(t,k,\xi)\big|^2\,d\xi\lesssim_\delta\eps_1^2,\\
&\int_1^t\sum_{k\in \mathbb{Z}\setminus\{0\}}\int_{\R}|\dot{A}_k(s,\xi)|A_k(s,\xi)\frac{\langle s\rangle^2k^4\langle s-\xi/k\rangle^4}{(|\xi/k|^2+\langle s\rangle^2)\langle\xi\rangle^2}\big|\widetilde{f}(s,k,\xi)\big|^2\,d\xi ds\lesssim_\delta\eps_1^2.
\end{split}
\end{equation}
\end{lemma}

\begin{proof} The identities \eqref{PtS2} and \eqref{fphi} show that
\begin{equation}\label{Pop2}
P'(t)=\frac{1}{2\pi}\int_0^{\infty}\int_0^{2\pi}\big(\sin{\theta},-\cos{\theta}\big)F(t,\theta-tv(t,r),v(t,r)) d\theta dr.
\end{equation}
Therefore, recalling the support property (\ref{Psupp}) of $F$, for any $t\in[1,T]$
\begin{equation}\label{Pop3}
\begin{split}
\big|P'(t)\big|&\lesssim \sum_{l\in\{1,-1\}}\Big|\int_0^{\infty}\int_0^{2\pi}e^{-il\theta}F(t,\theta-tv(t,r),v(t,r))\Psi(v(t,r)) d\theta dr\Big|\\
&\lesssim  \sum_{l\in\{1,-1\}}\Big|\int_0^{\infty}\int_{\mathbb{R}}\widetilde{F}(t,l,\xi)e^{-iltv(t,r)}e^{i\xi v(t,r)}\Psi(v(t,r))\, d\xi dr\Big|\\
&\lesssim  \sum_{l\in\{1,-1\}}\Big|\int_{\R}\widetilde{F}(t,l,\xi)\widetilde{(\Psi/V')}(t,lt-\xi)\,d\xi\Big|.
    \end{split}
\end{equation}
It follows from \eqref{boot2} and \eqref{nar4} that
\begin{equation*}
\|A_l(t,\xi)\widetilde{F}(t,l,\xi)\|_{L^2_\xi}\lesssim_\delta\eps_1\,\,\,\,\text{ and }\,\,\,\,\|A_R(t,\eta)\widetilde{(\Psi/V')}(t,\eta)\|_{L^2_\eta}\lesssim_\delta 1.
\end{equation*}
Moreover $A_l(t,\xi)A_R(t,lt-\xi)\gtrsim_\delta A_l(t,lt)$ (see \eqref{TLX7}) and $A_1(t,t)=A_{-1}(t,-t)$ (see \eqref{reb10}, \eqref{dor1}, and \eqref{dor4}). Using now \eqref{Pop3} it follows that
\begin{equation}\label{Pop3.1}
\big|P'(t)\big|\lesssim_\delta \eps_1A_1(t,t)^{-1},\qquad\text{ for any }t\in[1,T].
\end{equation}
The bounds \eqref{Pboot4}--\eqref{Par1.7} follow using also \eqref{reb12} and the assumption $|\widetilde{\Psi^\dagger}(\xi)|\lesssim e^{-\langle\xi\rangle^{3/4}}$.

We now turn to the proof of \eqref{Par1}. Notice that the weights used in the left-hand side are identical to the weights used in \eqref{nar14} to bound the functions $(\varrho)^a(V')^b\partial_v\mathbb{P}_{\neq 0}(\Psi\phi)$. The desired bounds follow directly from \eqref{Par1.7} if $f\in\{t\Psi^\dagger P_z',t\Psi^\dagger P_v'\}$, once we notice that
\begin{equation*}
A_1(t,\xi)=A_{-1}(t,-\xi),\qquad \dot{A}_1(t,\xi)=\dot{A}_{-1}(t,-\xi),\qquad \frac{A_1(t,\xi)}{A_1(t,t)}\lesssim_\delta e^{2\delta_0\langle\xi-t\rangle^{1/2},}
\end{equation*}
which follow from the definitions in subsection \ref{weightsdefin} and the bounds \eqref{vfc25.1}. Moreover, as we have seen in the proof of Lemma \ref{nar13}, these bounds can be extended to the full set of functions $f$, using the weighted bounds \eqref{rew19}--\eqref{rew19.1} and the estimates $\|\Psi_1(\varrho)^a(V')^b\|_R\lesssim_\delta 1$, see \eqref{nar4}.
\end{proof}

We prove now bounds on the function $\dot{V}$.

\begin{lemma}\label{nar8ext} With $\mathcal{K}$ as in \eqref{rec3} and \eqref{rec6}, for any $t\in[1,T]$ we have
\begin{equation}\label{nar7}
\begin{split}
&\int_{\R}A_{NR}^2(t,\xi)\big(\langle\xi\rangle^2\langle t\rangle^2+\mathcal{K}^2\langle\xi\rangle^{1/2}\langle t\rangle^{7/2}\big)\big|\widetilde{\dot{V}}(t,\xi)\big|^2\,d\xi\lesssim_\delta\eps_1^2,\\
&\int_1^t\int_{\R}|\dot{A}_{NR}(s,\xi)|A_{NR}(s,\xi)\big(\langle\xi\rangle^2\langle s\rangle^2+\mathcal{K}^2\langle\xi\rangle^{1/2}\langle s\rangle^{7/2}\big)\big|\widetilde{\dot{V}}(s,\xi)\big|^2\,d\xi ds\lesssim_\delta\eps_1^2.
\end{split}
\end{equation}
\end{lemma}

\begin{proof}
We use the formula $V'\partial_v\dot{V}+2\varrho\dot{V}=W_\ast/t$ (see \eqref{PV''}) and the bootstrap assumptions \eqref{boot2}. Since $\dot{V}$ and $W_\ast$  are supported in $[\underline{v},\overline{v}]$ we have
\begin{equation*}
\partial_v\dot{V}+2\Psi\varrho(V')^{-1}\dot{V}=\Psi W_\ast (V')^{-1}/t,
\end{equation*}
where $\Psi$ is as in \eqref{rec0}. Thus, if $H$ is defined such that $\partial_vH=2\Psi\varrho(V')^{-1}$ then
\begin{equation}\label{rew8}
\partial_v(e^H\dot{V})=e^H\Psi W_\ast (V')^{-1}/t.
\end{equation}

{\bf{Step 1.}} With $\Psi^\flat$ as in the proof of Lemma \ref{nar8} above, we show first that
\begin{equation}\label{rew9}
\|\Psi^\flat e^H\|_R+\|\Psi^\flat e^{-H}\|_R+\|\Psi^\flat e^H\partial_vH\|_R+\|\Psi^\flat e^{-H}\partial_vH\|_R\lesssim_\delta 1.
\end{equation}
Indeed, let $G_1(t,v):=\Psi(v)\varrho(t,v)(V'(t,v))^{-1}$. Since $H(t,v)=\int_0^v G_1(t,x)\,dx$ we have
\begin{equation*}
\begin{split}
\widetilde{(\Psi^\dagger H)}(t,\xi)&=C\int_{\mathbb{R}}\Psi^\dagger(v)e^{-i\xi v}\Big(\int_\mathbb{R}\int_0^v\widetilde{G_1}(t,\eta)e^{i\eta x}\,dx d\eta\Big)\,dv\\
&=C\int_{\mathbb{R}}\widetilde{G_1}(t,\eta)\frac{\widetilde{\Psi^\dagger}(\xi-\eta)-\widetilde{\Psi^\dagger}(\xi)}{i\eta}\,d\eta.
\end{split}
\end{equation*}
Therefore, recalling the properties of $\Psi^\dagger$,
\begin{equation}\label{rew10}
|\widetilde{(\Psi^\dagger H)}(t,\xi)|\lesssim\int_{\mathbb{R}}|\widetilde{G_1}(t,\eta)|e^{-\mu\langle\xi-\eta\rangle^{3/4}}\,d\eta+e^{-\mu\langle\xi\rangle^{3/4}}\|\widetilde{G_1}(t,\eta)\|_{L^2_\eta}.
\end{equation}
for some $\mu\approx 1$. In view of Lemma \ref{nar8} we have $\|G_1\|_R\lesssim_\delta 1$, therefore $\|\Psi^\dagger H\|_R\lesssim_\delta 1$ as a consequence of \eqref{rew10} and Lemma \ref{nar8} (i). Moreover, $\|\Psi^\flat \partial_v H\|_R\lesssim_\delta 1$, since $\partial_vH=2\Psi\varrho(V')^{-1}$, and the bounds \eqref{rew9} follow using the algebra bounds \eqref{rew1}.

{\bf{Step 2.}} With $G_2\in\{W_\ast/t,e^H\Psi W_\ast (V')^{-1}/t\}$, for any $t\geq 1$ we show now that
\begin{equation}\label{rew11}
\begin{split}
&\int_{\R}A_{NR}^2(t,\xi)\big(\langle t\rangle^2+\mathcal{K}^2\langle\xi\rangle^{-3/2}\langle t\rangle^{7/2}\big)\big|\widetilde{G_2}(t,\xi)\big|^2\,d\xi\lesssim_\delta\eps_1^2,\\
&\int_1^t\int_{\R}|\dot{A}_{NR}(s,\xi)|A_{NR}(s,\xi)\big(\langle s\rangle^2+\mathcal{K}^2\langle\xi\rangle^{-3/2}\langle s\rangle^{7/2}\big)\big|\widetilde{G_2}(s,\xi)\big|^2\,d\xi ds\lesssim_\delta\eps_1^2.
\end{split}
\end{equation}
Indeed, if $G_2=W_\ast/t$ then the bounds \eqref{rew11} hold, due to the bootstrap assumptions \eqref{boot2} and the identity $W_\ast=-V_\ast+\varrho\langle F\rangle$. Moreover, since $\|e^H\Psi (V')^{-1}\|_R\lesssim_\delta 1$, due to \eqref{rew8} and Lemma \ref{nar8}, the bounds \eqref{rew11} for $G_2=e^H\Psi W_\ast (V')^{-1}/t$ follow by the same argument as in the proof of Lemma \ref{nar8} (i), using the bilinear estimates \eqref{TLX4}--\eqref{DtVMulti} with $Y=NR$ and $\alpha\in\{0,3/4\}$.

{\bf{Step 3.}} We are now ready to prove the bounds \eqref{nar7}. In view of \eqref{rew8} and the compact support of the function $\dot{V}$, for any $t\geq 1$ and $\xi\in\mathbb{R}$  we have
\begin{equation*}
|\widetilde{e^H\dot{V}}(t,\xi)|\lesssim \langle\xi\rangle^{-1}|\widetilde{G_2}(t,\xi)|+\|\widetilde{G_2}(t,\eta)\|_{L^2_\eta}\mathbf{1}_{[0,1]}(|\xi|),
\end{equation*}
where $G_2=e^H\Psi W_\ast (V')^{-1}/t$. Therefore, in view of \eqref{rew11},
\begin{equation}\label{rew12}
\begin{split}
&\int_{\R}A_{NR}^2(t,\xi)\big(\langle\xi\rangle^2\langle t\rangle^2+\mathcal{K}^2\langle\xi\rangle^{1/2}\langle t\rangle^{7/2}\big)\big|\widetilde{e^H\dot{V}}(t,\xi)\big|^2\,d\xi\lesssim_\delta\eps_1^2,\\
&\int_1^t\int_{\R}|\dot{A}_{NR}(s,\xi)|A_{NR}(s,\xi)\big(\langle\xi\rangle^2\langle s\rangle^2+\mathcal{K}^2\langle\xi\rangle^{1/2}\langle s\rangle^{7/2}\big)\big|\widetilde{e^H\dot{V}}(s,\xi)\big|^2\,d\xi ds\lesssim_\delta\eps_1^2.
\end{split}
\end{equation}
The desired bounds \eqref{nar7} now follow as in the proof of Lemma \ref{nar8} (i), using \eqref{rew9} and the bilinear estimates \eqref{TLX4}--\eqref{DtVMulti} with $Y=NR$ and $\alpha=0$.
\end{proof}

We conclude this section with estimates on some functions that appear in the right-hand side of \eqref{PeH}.

\begin{lemma}
Assume that 
\begin{equation}\label{rew50}
f\in\mathcal{S}:=\{\langle \partial_z\phi\partial_vF\rangle,\,\langle\partial_v\mathbb{P}_{\neq0}\phi\partial_zF\rangle,\,\langle P'_v\partial_vF\rangle,\,t\langle P'_z\,\partial_zF\rangle,\,t\langle P'_v\,\partial_zF\rangle\}
\end{equation}
and $g\in\{(\varrho)^a(V')^bf:\,a,b\in[-2,2]\cap\mathbb{Z},\,f\in\mathcal{S}\}$. Then, for any $t\in[1,T]$,
\begin{equation}\label{yar24}
\begin{split}
&\int_{\R}|\dot{A}_{NR}(t,\xi)|^{-2}A^4_{NR}(t,\xi)\big(\langle t\rangle^{3/2}\langle\xi\rangle^{-3/2}\big)|\widetilde{g}(t,\xi)|^2\,d\xi\lesssim_\delta \eps_1^4,\\
&\int_1^t\int_{\R}|\dot{A}_{NR}(s,\xi)|^{-1}A^3_{NR}(s,\xi)\big(\langle s\rangle^{3/2}\langle\xi\rangle^{-3/2}\big)|\widetilde{g}(s,\xi)|^2\,d\xi ds\lesssim_\delta \eps_1^4.
\end{split}
\end{equation}
\end{lemma}

\begin{proof}{\bf{Step 1.}} We prove first the bounds for the functions $f\in\mathcal{S}$. Take, for example, $f=\langle\partial_v\mathbb{P}_{\neq0}\phi\partial_zF\rangle$, and recall the support assumption on $F$. Thus
\begin{equation*}
\begin{split}
f(t,v)&=\frac{1}{2\pi}\int_{\mathbb{T}}\partial_v\mathbb{P}_{\neq0}\phi(t,z,v)\partial_zF(t,z,v)\,dz\\
&=C\sum_{k\in\mathbb{Z}}\int_{\R^2}e^{iv(\rho+\eta)}\widetilde{\partial_v\mathbb{P}_{\neq0}(\Psi\phi)}(t,k,\rho)\widetilde{\partial_zF}(t,-k,\eta)\,d\rho d\eta.
\end{split}
\end{equation*}
Therefore
\begin{equation}\label{yar25}
\widetilde{f}(t,\xi)=C\sum_{k\in\mathbb{Z}\setminus\{0\}}\int_{\R}k\eta\,\widetilde{\mathbb{P}_{\neq0}(\Psi\phi)}(t,k,\eta)\widetilde{F}(t,-k,\xi-\eta)\,d\eta.
\end{equation}

For any $t\in[0,T]$, $k\in\Z\setminus\{0\}$, and $\xi,\eta\in\R$, $\rho=\xi-\eta$, we have the multiplier bounds
\begin{equation}\label{yar29}
\begin{split}
\frac{A^2_{NR}(t,\xi)}{|\dot{A}_{NR}(t,\xi)|}&\frac{\langle t\rangle^{3/4}}{\langle\xi\rangle^{3/4}}(\langle\eta\rangle+\langle\rho\rangle)\lesssim_\delta A_k(t,\eta)\frac{\langle t\rangle\langle t-\eta/k\rangle^2}{\langle t\rangle+|\eta/k|}A_{-k}(t,\rho)\{\langle\rho\rangle^{-4}+\langle\eta\rangle^{-4}\}
\end{split}
\end{equation}
and
\begin{equation}\label{yar30}
\begin{split}
\frac{A^{3/2}_{NR}(t,\xi)}{|\dot{A}_{NR}(t,\xi)|^{1/2}}&\frac{\langle t\rangle^{3/4}}{\langle\xi\rangle^{3/4}}(\langle\eta\rangle+\langle\rho\rangle)\lesssim_\delta \big[|(\dot{A}_k/A_k)(t,\eta)|^{1/2}+|(\dot{A}_{-k}/A_{-k})(t,\rho)|^{1/2}\big]\\
&\times A_k(t,\eta)\frac{\langle t\rangle\langle t-\eta/k\rangle^2}{\langle t\rangle+|\eta/k|}A_{-k}(t,\rho)\{\langle\rho\rangle^{-4}+\langle\eta\rangle^{-4}\}.
\end{split}
\end{equation}
These estimates follow from Lemma 8.7 in \cite{IOJI}.

As before, the estimates \eqref{yar24} follow from the multiplier bounds \eqref{yar29}--\eqref{yar30}. Indeed, to prove the harder bounds in the second line we estimate first
\begin{equation*}
\begin{split}
&\Big\{\int_1^t\int_{\R}|\dot{A}_{NR}(s,\xi)|^{-1}A^3_{NR}(s,\xi)\big(\langle s\rangle^{3/2}\langle\xi\rangle^{-3/2}\big)|\widetilde{f}(s,\xi)|^2\,d\xi ds\Big\}^{1/2}\\
&\lesssim\sup_{\|P\|_{L^2([1,t]\times\R)}=1}\int_1^t\int_{\R}|P(s,\xi)||\dot{A}_{NR}(s,\xi)|^{-1/2}A^{3/2}_{NR}(s,\xi)\big(\langle s\rangle^{3/4}\langle\xi\rangle^{-3/4}\big)|\widetilde{f}(s,\xi)|\,d\xi ds.
\end{split}
\end{equation*}
Using now \eqref{yar25} and \eqref{yar30}, the right-hand side of the expression above is bounded by
\begin{equation}\label{yar31}
\begin{split}
&C_\delta\int_1^t\int_\R\int_\R \sum_{k\in\Z\setminus\{0\}}\big\{|P(s,\eta+\rho)|\big[|(\dot{A}_k/A_k)(s,\eta)|^{1/2}+|(\dot{A}_{-k}/A_{-k})(s,\rho)|^{1/2}\big]\\
&\times A_k(s,\eta)\frac{\langle s\rangle|k|\langle s-\eta/k\rangle^2}{\langle s\rangle+|\eta/k|}A_{-k}(s,\rho)\{\langle\rho\rangle^{-4}+\langle\eta\rangle^{-4}\}|\widetilde{\mathbb{P}_{\neq0}(\Psi\phi)}(s,k,\eta)||\widetilde{F}(s,-k,\rho)|\big\}\,d\eta d\rho ds.
\end{split}
\end{equation}
We integrate first the variables $\eta$ and $\rho$. For any $k\in\Z$ and $t\in[1,T]$ let
\begin{equation*}
\begin{split}
\widetilde{F}^\#(t,k)&:=\Big\{\int_{\R}A^2_k(t,\xi)|\widetilde{F}(t,k,\xi)|^2\,d\xi\Big\}^{1/2},\\
\widetilde{F}^{\#\#}(t,k)&:=\Big\{\int_{\R}|\dot{A}_k(t,\xi)|A_k(t,\xi)|\widetilde{F}(t,k,\xi)|^2\,d\xi\Big\}^{1/2}.
\end{split}
\end{equation*}
Similarly, for any $k\in\Z\setminus\{0\}$ and $t\in[1,T]$ let
\begin{equation*}
\begin{split}
\widetilde{\Theta}^\#(t,k)&:=\Big\{\int_{\R}A^2_k(t,\xi)\frac{|k|^4\langle t\rangle^2\langle t-\xi/k\rangle^4}{|\xi/k|^2+\langle t\rangle^2}|\widetilde{\mathbb{P}_{\neq0}(\Psi\phi)}(t,k,\xi)|^2\,d\xi\Big\}^{1/2},\\
\widetilde{\Theta}^{\#\#}(t,k)&:=\Big\{\int_{\R}|\dot{A}_k(t,\xi)|A_k(t,\xi)\frac{|k|^4\langle t\rangle^2\langle t-\xi/k\rangle^4}{|\xi/k|^2+\langle t\rangle^2}|\widetilde{\mathbb{P}_{\neq0}(\Psi\phi)}(t,k,\xi)|^2\,d\xi\Big\}^{1/2}.
\end{split}
\end{equation*}
Letting also $P^\#(s):=\|P(s,\xi)\|_{L^2_\xi}$, the expression in \eqref{yar31} is bounded by
\begin{equation*}
\begin{split}
C_\delta\int_1^t \sum_{k\in\Z\setminus\{0\}}&\big\{P^\#(s)\widetilde{F}^\#(s,-k)\widetilde{\Theta}^{\#\#}(s,k)+P^\#(s)\widetilde{F}^{\#\#}(s,-k)\widetilde{\Theta}^{\#}(s,k)\big\}ds\\
&\lesssim_\delta \|P^\#\|_{L^2_s}\|\widetilde{F}^\#\|_{L^\infty_sL^2_k}\|\widetilde{\Theta}^{\#\#}\|_{L^2_sL^2_k}+\|P^\#\|_{L^2_s}\|\widetilde{f}^{\#\#}\|_{L^2_sL^2_k}\|\widetilde{\Theta}^{\#}\|_{L^\infty_sL^2_k}.
\end{split}
\end{equation*}
The desired bounds in the second line of \eqref{yar24} follow since $\|\widetilde{F}^\#\|_{L^\infty_sL^2_k}+\|\widetilde{F}^{\#\#}\|_{L^2_sL^2_k}+\|\widetilde{\Theta}^{\#}\|_{L^\infty_sL^2_k}+\|\widetilde{\Theta}^{\#\#}\|_{L^2_sL^2_k}\lesssim\eps_1$, as a consequence of the bootstrap assumptions on $F$ and $\phi$. The bounds in the first line follow in a similar (in fact slightly easier) way from the multiplier bounds \eqref{yar29}.

The proof is similar when $f=\langle \partial_z\phi\partial_vF\rangle$ (one just has to replace $k\eta$ with $k(\xi-\eta)$ in \eqref{yar25}). On the other hand, if $f\in\{\langle P'_v\partial_vF\rangle,\,t\langle P'_z\,\partial_zF\rangle,\,t\langle P'_v\,\partial_zF\rangle\}$ then we use the stronger bounds \eqref{Par1.7} together with \eqref{yar29}--\eqref{yar30}, and estimate the functions in a similar way.

{\bf{Step 2.}} We consider now multiplication by functions in the space $R$. In view of Lemma \ref{Multi0} (i) and \eqref{nar4}, it suffices to prove the multiplier estimates
\begin{equation}\label{yar49}
\frac{A^2_{NR}(t,\xi)}{|\dot{A}_{NR}(t,\xi)|}\frac{\langle t\rangle^{3/4}}{\langle\xi\rangle^{3/4}}\lesssim_\delta \frac{A^2_{NR}(t,\eta)}{|\dot{A}_{NR}(t,\eta)|}\frac{\langle t\rangle^{3/4}}{\langle\eta\rangle^{3/4}}A_R(t,\xi-\eta)\{\langle\xi-\eta\rangle^{-2}+\langle\eta\rangle^{-2}\}
\end{equation}
and
\begin{equation}\label{yar50}
\begin{split}
\frac{A^{3/2}_{NR}(t,\xi)}{|\dot{A}_{NR}(t,\xi)|^{1/2}}\frac{\langle t\rangle^{3/4}}{\langle\xi\rangle^{3/4}}&\lesssim_\delta \big[|(\dot{A}_{NR}/A_{NR})(t,\eta)|^{1/2}+|(\dot{A}_R/A_R)(t,\xi-\eta)|^{1/2}\big]\\
&\times \frac{A^2_{NR}(t,\eta)}{|\dot{A}_{NR}(t,\eta)|}\frac{\langle t\rangle^{3/4}}{\langle\eta\rangle^{3/4}}A_R(t,\xi-\eta)\{\langle\xi-\eta\rangle^{-2}+\langle\eta\rangle^{-2}\}.
\end{split}
\end{equation}
The first bound follows from \cite[Lemma 8.8]{IOJI}, while the second inequality follows using also \eqref{DtVMulti}. This completes the proof of the lemma.
\end{proof}

\section{Bootstrap estimates, I: improved control of  the variables $F$, $V_\ast$, $\varrho_\ast$, $W_\ast$}\label{fimprov}

In this section we use the evolution equations \eqref{Pef}--\eqref{PeH} and energy estimates to prove the main bounds \eqref{boot3} for the bootstrap variables $F$, $V_\ast$, $\varrho_\ast$, and $W_\ast$. 

\subsection{The normalized vorticity $F$} We prove the following:

\begin{proposition}\label{BootImp1}
With the definitions and assumptions in Proposition \ref{MainBootstrap}, we have
\begin{equation}\label{nar1}
\mathcal{E}_F(t)+\mathcal{B}_F(t)\lesssim_\delta\eps_1^3\qquad\text{ for any }t\in[1,T].
\end{equation}
\end{proposition}

\begin{proof} {\bf{Step 1.}} We calculate
\begin{equation}\label{nar1.5}
\begin{split}
\frac{d}{dt}\E_F(t)=&\sum_{k\in \mathbb{Z}}\int_\R 2\dot{A}_k(t,\xi)A_k(t,\xi)\big|\widetilde{F}(t,k,\xi)\big|^2\,d\xi\\
&+2\Re\sum_{k\in \mathbb{Z}}\int_{\R}A_k^2(t,\xi)\partial_t\widetilde{F}(t,k,\xi)\overline{\widetilde{F}(t,k,\xi)}\,d\xi.
\end{split}
\end{equation}
Therefore, since $\partial_tA_k\leq 0$, for any $t\in[1,T]$ we have
\begin{equation*}
\begin{split}
&\E_F(t)+\int_1^t\sum_{k\in \mathbb{Z}}\int_\R 2|\dot{A}_k(s,\xi)|A_k(s,\xi)\big|\widetilde{F}(s,k,\xi)\big|^2\,d\xi ds\\
&=\E_F(1)+\int_1^t\Big\{2\Re\sum_{k\in \mathbb{Z}}\int_{\R}A_k^2(s,\xi)\partial_s\widetilde{F}(s,k,\xi)\overline{\widetilde{F}(s,k,\xi)}\,d\xi\Big\}ds.
\end{split}
\end{equation*}
Since $\E_F(1)\lesssim\eps_1^3$ (see \eqref{boot1}), for \eqref{nar1} it suffices to prove that, for any $t\in[0,T]$.
\begin{equation}\label{nar2}
\Big|2\Re\int_1^t\sum_{k\in \mathbb{Z}}\int_{\R}A_k^2(s,\xi)\partial_s\widetilde{F}(s,k,\xi)\overline{\widetilde{F}(s,k,\xi)}\,d\xi ds\Big|\lesssim_\delta \eps_1^3.
\end{equation}

We examine now the space-time integrals in the left-hand side of \eqref{nar2}, and use the identity \eqref{Pef}. Therefore, recalling the support property of $F$,
\begin{equation}\label{nar8.1}
\begin{split}
&\partial_sF=\mathcal{N}_1+\mathcal{N}_2+\mathcal{N}_3+\mathcal{N}_4,\\
&\mathcal{N}_1:=-\varrho V'\partial_vP_{\neq 0}(\Psi\phi)\,\partial_zF,\qquad \mathcal{N}_2:=\varrho V'\partial_z(\Psi\phi)\partial_vF,\\
&\mathcal{N}_3:=-\dot{V}\,\partial_vF,\qquad \mathcal{N}_4:=V'P'_v\,(\partial_v-t\partial_z)F+\varrho P'_z\,\partial_zF.
\end{split}
\end{equation}

{\bf{Step 2.}} We consider now the contributions of the terms $\mathcal{N}_1$, $\mathcal{N}_2$, and $\mathcal{N}_3$. The estimates are similar to the estimates of the corresponding terms in \cite[Lemmas 4.4, 4.6, and 4.8]{IOJI}. 

{\bf{Substep 2.1.}} We provide all the details only for the term $\mathcal{N}_1$. For any $t\in[1,T]$ we will prove that
\begin{equation}\label{nar11}
\Big|2\Re\int_1^t\sum_{k\in \mathbb{Z}}\int_{\R}A_k^2(s,\xi)\widetilde{\mathcal{N}_1}(s,k,\xi)\overline{\widetilde{F}(s,k,\xi)}\,d\xi ds\Big|\lesssim_\delta \eps_1^3.
\end{equation}

With $h_2=\varrho V'\partial_v\mathbb{P}_{\neq 0}(\Psi\phi)$ as in the proof of Lemma \ref{nar13}, we have
\begin{equation}\label{exN1}
\begin{split}
&\Big|2\Re\int_1^t\sum_{k\in \mathbb{Z}}\int_{\R}A_k^2(s,\xi)\widetilde{\mathcal{N}_1}(s,k,\xi)\overline{\widetilde{F}(s,k,\xi)}\,d\xi ds\Big|\\
&=C\Big|2\Re\Big\{\sum_{k,\ell\in \mathbb{Z}}\int_1^t\int_{\R^2}A_k^2(s,\xi)\widetilde{h_2}(s,k-\ell,\xi-\eta)i\ell\widetilde{F}(s,\ell,\eta)\overline{\widetilde{F}(s,k,\xi)}\,d\xi d\eta ds\Big\}\Big|\\
&=C\Big|\int_1^t\sum_{k,\ell\in \mathbb{Z}}\int_{\R^2}\big[\ell A_k^2(s,\xi)-k A_\ell^2(s,\eta)\big]\widetilde{h_2}(s,k-\ell,\xi-\eta)\widetilde{F}(s,\ell,\eta)\overline{\widetilde{F}(s,k,\xi)}\,d\xi d\eta ds\Big|,
\end{split}
\end{equation}
where the second identity is proved by symmetrization (recall that $h_2$ is real-valued).

We define the sets
\begin{equation}\label{nar18.1}
\begin{split}
R_0:=\Big\{&((k,\xi),(\ell,\eta))\in (\Z\times \R)^2:\\
&\min(\langle k,\xi\rangle,\,\langle\ell,\eta\rangle,\,\langle k-\ell,\xi-\eta\rangle)\geq \frac{\langle k,\xi\rangle+\langle\ell,\eta\rangle+\langle k-\ell,\xi-\eta\rangle}{20}\Big\},\\
\end{split}
\end{equation}
\begin{equation}\label{nar18.2}
R_1:=\Big\{((k,\xi),(\ell,\eta))\in (\Z\times \R)^2:\,\langle k-\ell,\xi-\eta\rangle\leq \frac{\langle k,\xi\rangle+\langle\ell,\eta\rangle+\langle k-\ell,\xi-\eta\rangle}{10}\Big\},
\end{equation}
\begin{equation}\label{nar18.3}
R_2:=\Big\{((k,\xi),(\ell,\eta))\in (\Z\times \R)^2:\,\langle\ell,\eta\rangle\leq \frac{\langle k,\xi\rangle+\langle\ell,\eta\rangle+\langle k-\ell,\xi-\eta\rangle}{10}\Big\},
\end{equation}
\begin{equation}\label{nar18.4}
R_3:=\Big\{((k,\xi),(\ell,\eta))\in (\Z\times \R)^2:\,\langle k,\xi\rangle\leq \frac{\langle k,\xi\rangle+\langle\ell,\eta\rangle+\langle k-\ell,\xi-\eta\rangle}{10}\Big\}.
\end{equation}
Then we define the corresponding integrals for $a\in\{0,1,2,3\}$,
\begin{equation}\label{nar19}
\begin{split}
\mathcal{U}_a:=\int_1^t\sum_{k,\ell\in \mathbb{Z}}\int_{\R^2}&\mathbf{1}_{R_a}((k,\xi),(\ell,\eta))\big|\ell A_k^2(s,\xi)-k A_\ell^2(s,\eta)\big|\,|\widetilde{h_2}(s,k-\ell,\xi-\eta)|\\
&\times|\widetilde{F}(s,\ell,\eta)|\,|\widetilde{F}(s,k,\xi)|\,d\xi d\eta ds.
\end{split}
\end{equation}

To bound the integrals $\mathcal{U}_a$ we use estimates on the weights. More precisely, letting $(m,\rho):=(k-\ell,\xi-\eta)$, $\delta'_0:=\delta_0/200$, and assuming that $m\neq0$ we have the following bounds:

$\bullet\,\,$ If $((k,\xi),(\ell,\eta))\in R_0\cup R_1$, then
\begin{equation}\label{TLXH1.1}
\begin{split}
\frac{(|\rho/m|+\langle t \rangle)\langle \rho \rangle}{\langle t\rangle m^2\langle t-\rho/m \rangle^2}&\big|\ell A_k^2(t,\xi)-kA_{\ell}^2(s,\eta)\big|\lesssim_{\delta}\sqrt{|A_k\dot{A}_k(t,\xi)|}\,\sqrt{|A_{\ell}\dot{A}_{\ell}(t,\eta)|}\,A_{m}(t,\rho) \,e^{-\delta'_0\langle m,\rho \rangle^{1/2}}.
\end{split}
\end{equation}

$\bullet\,\,$ If $((k,\xi),(\ell,\eta))\in R_2$, then
\begin{equation}\label{TLXH1.2}
\begin{split}
\frac{(|\rho/m|+\langle t \rangle)\langle \rho \rangle}{\langle t\rangle m^2\langle t-\rho/m \rangle^2}&\big|\ell A_k^2(t,\xi)-kA_{\ell}^2(s,\eta)\big|\lesssim_{\delta}\sqrt{|A_k\dot{A}_k(t,\xi)|}\,\sqrt{|A_{m}\dot{A}_{m}(t,\rho)|}\,A_{\ell}(t,\eta) \,e^{-\delta'_0\langle \ell,\eta \rangle^{1/2}}.\\
\end{split}
\end{equation}

See \cite[Lemma 8.4]{IOJI} for the proof.

To bound $\mathcal{U}_a$, $a=0,1$, we use \eqref{TLXH1.1}. We remark that $\widetilde{h_2}(t,0,\xi)\equiv0$. Denote $(m,\rho)=(k-\ell,\xi-\eta)$. We can then bound, using \eqref{nar14} and the bootstrapping bounds (\ref{boot2})
\begin{equation*}
\begin{split}
\mathcal{U}_a&\lesssim_{\delta}\int_1^t\sum_{k,\ell\in \mathbb{Z}}\int_{\R^2}\sqrt{|A_k\dot{A}_k(s,\xi)|}\,\big|\widetilde{F}(s,k,\xi)\big|\sqrt{|A_{\ell}\dot{A}_{\ell}(s,\eta)|}\,\big|\widetilde{F}(s,\ell,\eta)\big|\\
&\qquad\times\, \mathbf{1}_{\Z^\ast}(m)\frac{\langle s\rangle m^2\langle s-\rho/m \rangle^2}{(|\rho/m|+\langle s \rangle)\langle \rho \rangle}A_{m}(s,\rho)\big|\widetilde{h_2}(s,m,\rho)\big|e^{-(\delta_0/200)\langle m,\rho\rangle^{1/2}}\,d\xi d\eta ds\\
&\lesssim_{\delta} \Big\|\sqrt{|A_k\dot{A}_k(s,\xi)|}\,\widetilde{F}(s,k,\xi)\Big\|_{L^2_{s}L^2_{k,\xi}}\Big\|\sqrt{|A_{\ell}\dot{A}_{\ell}(s,\eta)|}\,\widetilde{F}(s,\ell,\eta)\Big\|_{L^2_sL^2_{\ell,\eta}}\\
&\qquad\times \Big\|\mathbf{1}_{\Z^\ast}(m)\frac{\langle s\rangle m^2\langle s-\rho/m \rangle^2}{(|\rho/m|+\langle s \rangle)\langle \rho \rangle}A_{m}(s,\rho)e^{-(\delta_0/300)\langle m,\rho\rangle^{1/2}}\widetilde{h_2}(s,m,\rho)\Big\|_{L^{\infty}_sL^2_{m,\rho}}\\
&\lesssim_{\delta}\epsilon_1^3.
\end{split}
\end{equation*}

Similarly, for $a=2$ we use \eqref{TLXH1.2}, \eqref{nar14}, and the bootstrapping bounds (\ref{boot2}), 
\begin{equation*}
\begin{split}
\mathcal{U}_2&\lesssim_{\delta}\int_1^t\sum_{k,\ell\in \mathbb{Z}}\int_{\R^2}\mathbf{1}_{\Z^\ast}(m)\sqrt{|A_{m}\dot{A}_{m}(s,\rho)|}\frac{\langle s\rangle m^2\langle s-\rho/m \rangle^2}{(|\rho/ m|+\langle s \rangle)\langle \rho \rangle}\big|\widetilde{h_2}(s,m,\rho)\big|\\
&\qquad\times\sqrt{|A_k\dot{A}_k(s,\xi)|}\,\big|\widetilde{F}(s,k,\xi)\big|A_{\ell}(s,\eta) e^{-(\delta_0/200)\langle \ell,\eta\rangle^{1/2}}|\widetilde{F}(s,\ell,\eta)|\,d\xi d\eta ds\\
&\lesssim_{\delta} \Big\|\sqrt{|A_k\dot{A}_k(s,\xi)|}\,\widetilde{F}(s,k,\xi)\Big\|_{L^2_{s}L^2_{k,\xi}}\Big\|A_{\ell}(s,\eta)\,e^{-(\delta_0/300)\langle \ell,\eta\rangle^{1/2}}\widetilde{F}(s,\ell,\eta)\Big\|_{L^{\infty}_sL^2_{\ell,\eta}}\\
&\qquad\times \Big\|\mathbf{1}_{\Z^\ast}(m)\sqrt{|A_{m}\dot{A}_{m}(s,\rho)|}\frac{\langle s\rangle m^2\langle s-\rho/m \rangle^2}{(|\rho/ m|+\langle s \rangle)\langle \rho \rangle}\big|\widetilde{h_2}(s,m,\rho)\Big\|_{L^{2}_sL^2_{m,\rho}}\\
&\lesssim_{\delta}\epsilon_1^3.
\end{split}
\end{equation*}
By changes of variables one can prove that $\mathcal{U}_3\lesssim_\delta\eps_1^3$ as well, and the bounds \eqref{nar11} follow.
\smallskip

{\bf{Substep 2.2.}} The contributions of the nonlinearities $\mathcal{N}_2$ and $\mathcal{N}_3$ can be bounded by the same general procedure: energy estimates and symmetrization, bounds on the weights, and $L^2_sL^2\times L^2_sL^2\times L^\infty_sL^2$ estimates to control the space-time integrals. See \cite[Lemmas 4.6, and 4.8]{IOJI} for details; here we summarize only the bounds on the weights we use.

To control $\mathcal{N}_2$ we recall the sets $R_0,R_1,R_2,R_3$ defined in (\ref{nar18.1})-(\ref{nar18.4}). Denote $(m,\rho):=(k-\ell,\xi-\eta)$, $\delta'_0:=\delta_0/200$. Suppose that $m\neq0$.

$\bullet\,\,$ If $((k,\xi),(\ell,\eta))\in R_0\cup R_1$, then
\begin{equation}\label{TLXH3.1}
\begin{split}
\frac{|\rho/m|^2+\langle t \rangle^2}{|m|\langle t\rangle^2\langle t-\rho/m \rangle^2}&\big|\eta A_k^2(t,\xi)-\xi A_{\ell}^2(s,\eta)\big|\lesssim_{\delta}\sqrt{|A_k\dot{A}_k(t,\xi)|}\,\sqrt{|A_{\ell}\dot{A}_{\ell}(t,\eta)|}\,A_{m}(t,\rho) \,e^{-\delta'_0\langle m,\rho \rangle^{1/2}}.
\end{split}
\end{equation}

$\bullet\,\,$ If $((k,\xi),(\ell,\eta))\in R_2$, then
\begin{equation}\label{TLXH3.2}
\begin{split}
\frac{|\rho/m|^2+\langle t \rangle^2}{|m|\langle t\rangle^2\langle t-\rho/m \rangle^2}&\big|\eta A_k^2(t,\xi)-\xi A_{\ell}^2(s,\eta)\big|\lesssim_{\delta}\sqrt{|A_k\dot{A}_k(t,\xi)|}\,\sqrt{|A_{m}\dot{A}_{m}(t,\rho)|}\,A_{\ell}(t,\eta) \,e^{-\delta'_0\langle \ell,\eta \rangle^{1/2}}.\\
\end{split}
\end{equation}

See \cite[Lemma 8.5]{IOJI} for the proof.

To control $\mathcal{N}_3$ we define the sets $R_j^\ast$, $j=\{0,1,2,3\}$, by
\begin{equation}\label{nar19.1}
R_j^\ast:=\big\{((k,\xi),(l,\eta))\in R_j:\,k=l\big\},
\end{equation}
where $R_j$ are as in \eqref{nar18.1}--\eqref{nar18.4}. Denote $\rho:=\xi-\eta$.

$\bullet\,\,$ If $((k,\xi),(k,\eta))\in R^\ast_0\cup R^\ast_1$, then
\begin{equation}\label{TLXH2.1}
\begin{split}
\frac{\big|\eta A_k^2(t,\xi)-\xi A_{k}^2(s,\eta)\big|}{\langle\rho\rangle\langle t\rangle+\langle \rho\rangle^{1/4}\langle t\rangle^{7/4}}\lesssim_{\delta}\sqrt{|A_k\dot{A}_k(t,\xi)|}\,\sqrt{|A_{k}\dot{A}_{k}(t,\eta)|}\,A_{NR}(t,\rho) \,e^{-\delta'_0\langle \rho \rangle^{1/2}}.
\end{split}
\end{equation}

$\bullet\,\,$ If $((k,\xi),(k,\eta))\in R^\ast_2$, then
\begin{equation}\label{TLXH2.2}
\begin{split}
\frac{\big|\eta A_k^2(t,\xi)-\xi A_{k}^2(t,\eta)\big|}{\langle\rho\rangle\langle t\rangle+\langle \rho\rangle^{1/4}\langle t\rangle^{7/4}}\lesssim_{\delta}\sqrt{|A_k\dot{A}_k(t,\xi)|}\,\sqrt{|A_{NR}\dot{A}_{NR}(t,\rho)|}\,A_{k}(t,\eta) \,e^{-\delta'_0\langle k,\eta \rangle^{1/2}}.\\
\end{split}
\end{equation}

See \cite[Lemma 8.6]{IOJI} for the proof.

{\bf{Step 3.}} Finally, we bound  the contribution of the nonlinearity $\mathcal{N}_4$. We will show that
\begin{equation}\label{zar36}
\Big|2\Re\int_1^t\sum_{k\in \mathbb{Z}}\int_{\R}A_k^2(s,\xi)\widetilde{\mathcal{N}_4}(s,k,\xi)\overline{\widetilde{F}(s,k,\xi)}\,d\xi ds\Big|\lesssim_\delta \eps_1^3.
\end{equation}
Let $g_1:=V'P'_v$ and $g_2:=-tV'P'_v+\varrho P'_z$, so $\mathcal{N}_4=g_1\partial_vF+g_2\partial_zF$. As before, after symmetrization, the left-hand side of \eqref{zar36} is dominated by $C(\mathcal{I}+\mathcal{J})$, where 
\begin{equation}\label{zar36.55}
\begin{split}
\mathcal{I}:=&\Big|\int_1^t\sum_{k,\ell\in \mathbb{Z}}\int_{\R^2}\big[\eta A_k^2(s,\xi)-\xi A_\ell^2(s,\eta)\big]\widetilde{g_1}(s,k-\ell,\xi-\eta)\widetilde{F}(s,\ell,\eta)\overline{\widetilde{F}(s,k,\xi)}\,d\xi d\eta ds\Big|,\\
\mathcal{J}:=&\Big|\int_1^t\sum_{k,\ell\in \mathbb{Z}}\int_{\R^2}\big[\ell A_k^2(s,\xi)-k A_\ell^2(s,\eta)\big]\widetilde{g_2}(s,k-\ell,\xi-\eta)\widetilde{F}(s,\ell,\eta)\overline{\widetilde{F}(s,k,\xi)}\,d\xi d\eta ds\Big|.
\end{split}
\end{equation}

To estimate $\mathcal{I}$ we recall the definitions \eqref{nar18.1}--\eqref{nar18.4} and define
\begin{equation}\label{nar76}
\begin{split}
\mathcal{V}_a:=\int_1^t\sum_{k,\ell\in \mathbb{Z}}\int_{\R^2}&\mathbf{1}_{R_a}((k,\xi),(\ell,\eta))\big|\eta A_k^2(s,\xi)-\xi A_\ell^2(s,\eta)\big|\,|\widetilde{g_1}(s,k-\ell,\xi-\eta)|\\
&\times|\widetilde{F}(s,\ell,\eta)|\,|\widetilde{F}(s,k,\xi)|\,d\xi d\eta ds.
\end{split}
\end{equation}
To bound $\mathcal{V}_a$, $a=0,1$, we use \eqref{TLXH3.1}, \eqref{Par1}, and the bootstrapping bounds (\ref{boot2}). Thus, with $(m,\rho)=(k-l,\xi-\eta)$, we have
\begin{equation*}
\begin{split}
\mathcal{V}_a&\lesssim_{\delta}\int_1^t\sum_{k,\ell\in \mathbb{Z}}\int_{\R^2}\sqrt{|A_k\dot{A}_k(s,\xi)|}\,\big|\widetilde{F}(s,k,\xi)\big|\sqrt{|A_{\ell}\dot{A}_{\ell}(s,\eta)|}\,\big|\widetilde{F}(s,\ell,\eta)\big|\\
&\qquad\times\, \mathbf{1}_{\{-1,1\}}(m)\frac{\langle s\rangle^2|m|\langle s-\rho/m \rangle^2}{|\rho/m|^2+\langle s \rangle^2}A_{m}(s,\rho)\big|\widetilde{g_1}(s,m,\rho)\big|e^{-\delta'_0\langle m,\rho\rangle^{1/2}}\,d\xi d\eta ds\\
&\lesssim_{\delta} \Big\|\sqrt{|A_k\dot{A}_k(s,\xi)|}\,\widetilde{F}(s,k,\xi)\Big\|_{L^2_{s}L^2_{k,\xi}}\Big\|\sqrt{|A_{\ell}\dot{A}_{\ell}(s,\eta)|}\,\widetilde{F}(s,\ell,\eta)\Big\|_{L^2_sL^2_{\ell,\eta}}\\
&\qquad\times \Big\|\mathbf{1}_{\{-1,1\}}(m)\frac{\langle s\rangle^2|m|\langle s-\rho/\sigma \rangle^2}{|\rho/m|^2+\langle s \rangle^2}A_{m}(s,\rho)\big|\widetilde{g_1}(s,m,\rho)\big|e^{-(\delta'_0/2)\langle m,\rho\rangle^{1/2}}\Big\|_{L^{\infty}_sL^2_{m,\rho}}\\
&\lesssim_{\delta}\epsilon_1^3.
\end{split}
\end{equation*}

Similarly, for $a=2$ we use \eqref{TLXH3.2}, \eqref{Par1}, and the bootstrapping bounds (\ref{boot2}), 
\begin{equation*}
\begin{split}
\mathcal{V}_2&\lesssim_{\delta}\int_1^t\sum_{k,\ell\in \mathbb{Z}}\int_{\R^2}\mathbf{1}_{\{-1,1\}}(m)\frac{\langle s\rangle^2|m|\langle s-\rho/m\rangle^2}{|\rho/m|^2+\langle s \rangle^2}\sqrt{|A_{m}\dot{A}_{m}(s,\rho)|}\big|\widetilde{g_1}(s,m,\rho)\big|\\
&\qquad\times\sqrt{|A_k\dot{A}_k(s,\xi)|}\,\big|\widetilde{F}(s,k,\xi)\big|A_{\ell}(s,\eta) e^{-\delta'_0\langle \ell,\eta\rangle^{1/2}}|\widetilde{F}(s,\ell,\eta)|\,d\xi d\eta ds\\
&\lesssim_{\delta} \Big\|\sqrt{|A_k\dot{A}_k(s,\xi)|}\,\widetilde{f}(s,k,\xi)\Big\|_{L^2_{s}L^2_{k,\xi}}\Big\|A_{\ell}(s,\eta)\,e^{-(\delta'_0/2)\langle \ell,\eta\rangle^{1/2}}\widetilde{f}(s,\ell,\eta)\Big\|_{L^{\infty}_sL^2_{\ell,\eta}}\\
&\qquad\times \Big\|\mathbf{1}_{\{-1,1\}}(m)\frac{\langle s\rangle^2|m|\langle s-\rho/m\rangle^2}{|\rho/m|^2+\langle s \rangle^2}\sqrt{|A_{m}\dot{A}_{m}(s,\rho)|}\widetilde{g_1}(s,m,\rho)\Big\|_{L^{2}_sL^2_{\sigma,\rho}}\\
&\lesssim_{\delta}\epsilon_1^3.
\end{split}
\end{equation*}
By changes of variables one can prove that $\mathcal{V}_3\lesssim_\delta\eps_1^3$ as well, thus $\mathcal{I}\lesssim_\delta\eps_1^3$.

We notice now that the integral $\mathcal{J}$ in \eqref{zar36.55} is similar to the integral in \eqref{exN1}. Moreover, $g_2$ satisfies the bounds \eqref{Par1}, which are similar to the bounds \eqref{nar14} satisfied by the function $h_2$. The same estimates as in {\bf{Step 1}} show that $\mathcal{J}\lesssim_\delta\eps_1^3$, which completes the proof of \eqref{zar36}.
\end{proof}

\subsection{The coordinate functions $V_\ast$, $\varrho_\ast$, and $W_\ast$}\label{VRhoW} We now prove the following:

\begin{proposition}\label{BootImp3}
With the definitions and assumptions in Proposition \ref{MainBootstrap}, we have
\begin{equation}\label{yar1}
\mathcal{E}_{V_\ast}(t)+\mathcal{E}_{\varrho_\ast}(t)+\mathcal{E}_{W_\ast}(t)+\mathcal{B}_{V_\ast}(t)+\mathcal{B}_{\varrho_\ast}(t)+\mathcal{B}_{W_\ast}(t)\leq\eps_1^2/20\qquad\text{ for any }t\in[1,T].
\end{equation}
\end{proposition}

The rest of the subsection is concerned with the proof of this proposition. Some of the arguments are similar to the arguments in \cite[Section 6]{IOJI}. For the sake of completeness we provide most of the details.

Using the equations \eqref{PeV'}--\eqref{PeH} and the definitions \eqref{rec2}--\eqref{rec3} we calculate
\begin{equation*}
\begin{split}
\frac{d}{dt}&[\mathcal{E}_{W_\ast}+\mathcal{E}_{V_\ast}+\mathcal{E}_{\varrho_\ast}](t)=2\mathcal{K}^{2}\int_{\R}\dot{A}_{NR}(t,\xi)A_{NR}(t,\xi)\big(\langle t\rangle^{3/2}\langle\xi\rangle^{-3/2}\big)\big|\widetilde{W_\ast}(t,\xi)\big|^2\,d\xi\\
&+2\int_\R \dot{A}_R(t,\xi)A_R(t,\xi)\big|\widetilde{V_\ast}(t,\xi)\big|^2\,d\xi+2\int_\R \dot{A}_R(t,\xi)A_R(t,\xi)\big|\widetilde{(\Psi^\dagger\varrho_\ast)}(t,\xi)\big|^2\,d\xi\\
&+\mathcal{K}^{2}2\Re\int_{\R}A^2_{NR}(t,\xi)\big(\langle t\rangle^{3/2}\langle\xi\rangle^{-3/2}\big)\partial_t\widetilde{W_\ast}(t,\xi)\overline{\widetilde{W_\ast}(t,\xi)}\,d\xi\\
&+2\Re\int_\R A^2_R(t,\xi)\partial_t\widetilde{V_\ast}(t,\xi)\overline{\widetilde{V_\ast}(t,\xi)}\,d\xi+2\Re\int_\R A^2_R(t,\xi)\partial_t\widetilde{(\Psi^\dagger\varrho_\ast)}(t,\xi)\overline{\widetilde{(\Psi^\dagger\varrho_\ast)}(t,\xi)}\,d\xi\\
&+\mathcal{K}^{2}\int_{\R}A^2_{NR}(t,\xi)\frac{3}{2}\big(t\langle t\rangle^{-1/2}\langle\xi\rangle^{-3/2}\big)\big|\widetilde{W_\ast}(t,\xi)\big|^2\,d\xi.
\end{split}
\end{equation*}
Therefore, since $\partial_tA_R\leq 0$ and $\partial_tA_{NR}\leq 0$ and using \eqref{rec5}--\eqref{rec6}, for any $t\in[1,T]$ we have
\begin{equation}\label{yar2}
\begin{split}
\mathcal{E}_{V_\ast}(t)&+\mathcal{E}_{\varrho_\ast}(t)+\mathcal{E}_{W_\ast}(t)+\mathcal{B}_{V_\ast}(t)+\mathcal{B}_{\varrho_\ast}(t)+\mathcal{B}_{W_\ast}(t)\\
&=\mathcal{E}_{V_\ast}(1)+\mathcal{E}_{\varrho_\ast}(1)+\mathcal{E}_{W_\ast}(1)-[\mathcal{B}_{V_\ast}(t)+\mathcal{B}_{\varrho_\ast}(t)+\mathcal{B}_{W_\ast}(t)]+\mathcal{L}_1(t)+\mathcal{L}_2(t),
\end{split}
\end{equation}
where
\begin{equation}\label{yar3}
\begin{split}
\mathcal{L}_1(t):&=2\Re\int_1^t\int_\R A^2_R(s,\xi)\partial_s\widetilde{V_\ast}(s,\xi)\overline{\widetilde{V_\ast}(s,\xi)}\,d\xi ds\\
&+2\Re\int_1^t\int_\R A^2_R(s,\xi)\partial_s\widetilde{(\Psi^\dagger\varrho_\ast)}(s,\xi)\overline{\widetilde{(\Psi^\dagger\varrho_\ast)}(s,\xi)}\,d\xi ds,
\end{split}
\end{equation}
\begin{equation}\label{yar4}
\begin{split}
\mathcal{L}_2(t):&=\mathcal{K}^{2}2\Re\int_1^t\int_{\R}A^2_{NR}(s,\xi)\big(\langle s\rangle^{3/2}\langle\xi\rangle^{-3/2}\big)\partial_s\widetilde{W_\ast}(s,\xi)\overline{\widetilde{W_\ast}(s,\xi)}\,d\xi ds\\
&+\mathcal{K}^{2}\int_1^t\int_{\R}A^2_{NR}(s,\xi)\frac{3}{2}\big(s\langle s\rangle^{-1/2}\langle\xi\rangle^{-3/2}\big)\big|\widetilde{W_\ast}(s,\xi)\big|^2\,d\xi ds.
\end{split}
\end{equation}
Since $\mathcal{E}_{V_\ast}(1)+\mathcal{E}_{\varrho_\ast}(1)+\mathcal{E}_{W_\ast}(1)\lesssim\eps_1^3$, for \eqref{yar1} it suffices to prove that, for any $t\in[1,T]$,
\begin{equation}\label{yar6}
-[\mathcal{B}_{V_\ast}(t)+\mathcal{B}_{\varrho_\ast}(t)+\mathcal{B}_{W_\ast}(t)]+\mathcal{L}_1(t)+\mathcal{L}_2(t)\leq \eps_1^2/30.
\end{equation}

To prove \eqref{yar6} we use the equations \eqref{PeV'}--\eqref{PeH}. We extract the quadratic components of $\mathcal{L}_1$ and $\mathcal{L}_2$ (corresponding to the linear terms in the right-hand sides of \eqref{PeV'}--\eqref{PeH}), so we define $\dot{V}'(s,v):=\sqrt{\pi/(2\kappa v)}\dot{V}(s,v)$, and then
\begin{equation}\label{yar7}
\mathcal{L}_{1,2}(t):=2\Re\int_1^t\int_\R A^2_R(s,\xi)\Big\{\frac{\widetilde{W_\ast}(s,\xi)}{s}\overline{\widetilde{V_\ast}(s,\xi)}-\widetilde{\dot{V}'}(s,\xi)\overline{\widetilde{(\Psi^\dagger\varrho_\ast)}(s,\xi)}\Big\}\,d\xi ds,
\end{equation}
and
\begin{equation}\label{yar8}
\begin{split}
\mathcal{L}_{2,2}(t)&:=\mathcal{K}^{2}\int_1^t\int_{\R}A^2_{NR}(s,\xi)\Big\{-\frac{2\langle s\rangle^{3/2}}{s\langle\xi\rangle^{3/2}}|\widetilde{W_\ast}(s,\xi)|^2+\frac{3s/2}{\langle s\rangle^{1/2}\langle\xi\rangle^{3/2}}\big|\widetilde{W_\ast}(s,\xi)\big|^2\Big\}\,d\xi ds\\
&=-\mathcal{K}^{2}\int_1^t\int_{\R}A^2_{NR}(s,\xi)\frac{2+s^2/2}{s\langle\xi\rangle^{3/2}\langle s\rangle^{1/2}}|\widetilde{W_\ast}(s,\xi)|^2\,d\xi ds.
\end{split}
\end{equation}
We examine the identities \eqref{PeV'}--\eqref{PeH} and let
\begin{equation}\label{yar18.5}
\begin{split}
&f_1:=-\dot{V}\partial_vV_\ast,\qquad f_2:=-\dot{V}\partial_v(\Psi^\dagger\varrho_\ast),\qquad g_1:=-\dot{V}\partial_vW_\ast\\
&g_2:=\varrho^2V'[\langle\partial_z\phi\,\partial_vF\rangle-\langle\partial_vP_{\neq0}\phi\,\partial_zF\rangle],\qquad g_3:=\varrho V'\langle P'_v(\partial_v-t\partial_z)F\rangle+\varrho^2 \langle P'_z\,\partial_zF\rangle.
\end{split}
\end{equation}
Notice that
\begin{equation}\label{yar19}
\begin{split}
&\mathcal{L}_1(t)=\mathcal{L}_{1,2}(t)+2\Re\int_1^t\int_\R A^2_R(s,\xi)\big\{\widetilde{f_1}(s,\xi)\overline{\widetilde{V_\ast}(s,\xi)}+\widetilde{f_2}(s,\xi)\overline{\widetilde{(\Psi^\dagger\varrho_\ast)}(s,\xi)}\big\}\,d\xi ds,\\
&\mathcal{L}_2(t)=\mathcal{L}_{2,2}(t)+\sum_{a\in\{1,2,3\}}\mathcal{K}^{2}2\Re\int_1^t\int_{\R}A^2_{NR}(s,\xi)\big(\langle s\rangle^{3/2}\langle\xi\rangle^{-3/2}\big)\widetilde{g_a}(s,\xi)\overline{\widetilde{W_\ast}(s,\xi)}\,d\xi ds.
\end{split}
\end{equation}
The desired bounds \eqref{yar6} follow from Lemmas \ref{yar10} and \ref{yar20} below.

\begin{lemma}\label{yar10}
For any $t\in [1,T]$ we have
\begin{equation}\label{yar11}
-[\mathcal{B}_{V_\ast}(t)+\mathcal{B}_{\varrho_\ast}(t)+\mathcal{B}_{W_\ast}(t)]+\mathcal{L}_{1,2}(t)+\mathcal{L}_{2,2}(t)\leq \eps_1^2/40.
\end{equation}
\end{lemma}

\begin{proof} Since $\mathcal{L}_{2,2}(t)\leq 0$, it suffices to prove that, for any $t\in[1,T]$,
\begin{equation}\label{yar11.4}
\mathcal{L}_{1,2}(t)\leq \mathcal{B}_{V_\ast}(t)+\mathcal{B}_{\varrho_\ast}(t)+\eps_1^2/40.
\end{equation}
Using Cauchy-Schwartz and the definitions, we have
\begin{equation*}
\begin{split}
\mathcal{L}_{1,2}(t)&\leq \frac{1}{2}\mathcal{B}_{V_\ast}(t)+8\int_1^t\int_{\R}\frac{A^3_R(s,\xi)}{|\dot{A}_R(s,\xi)|}\frac{|\widetilde{W_\ast}(s,\xi)|^2}{s^2}\,d\xi ds\\
&+\frac{1}{2}\mathcal{B}_{\varrho_\ast}(t)+8\int_1^t\int_{\R}\frac{A^3_R(s,\xi)}{|\dot{A}_R(s,\xi)|}|\dot{V}'(s,\xi)|^2\,d\xi ds
\end{split}
\end{equation*}

The function $\dot{V}'$ satisfies the bounds \eqref{nar7}, similar to $\dot{V}$. Using also the estimates \eqref{rew11}, it suffices to show that, for any $C_\delta\geq 1$,
\begin{equation*}
\frac{A^3_R(s,\xi)}{s^2|\dot{A}_R(s,\xi)|}\leq A_{NR}(s,\xi)|\dot{A}_{NR}(s,\xi)|(C_\delta^{-1}+\mathcal{K}(\delta)^{2}\langle s\rangle^{3/2}\langle \xi\rangle^{-3/2}),
\end{equation*}
provided that $\mathcal{K}(\delta)$ is taken sufficiently large. This inequality is proved in Lemma 6.2 in \cite{IOJI}, using some of the basic properties of the weights, which are summarized in section \ref{weights}.
\end{proof}

We prove now estimates on the cubic terms. 

\begin{lemma}\label{yar20} 
For any $t\in[1,T]$ and $a\in\{1,2,3\}$ we have
\begin{equation}\label{yar21}
\Big|2\Re\int_1^t\int_\R A^2_R(s,\xi)\widetilde{f_1}(s,\xi)\overline{\widetilde{V_\ast}(s,\xi)}\,d\xi ds\Big|\lesssim_\delta\eps_1^3,
\end{equation}
\begin{equation}\label{yar21.7}
\Big|2\Re\int_1^t\int_\R A^2_R(s,\xi)\widetilde{f_2}(s,\xi)\overline{\widetilde{(\Psi^\dagger\varrho_\ast)}(s,\xi)}\,d\xi ds\Big|\lesssim_\delta\eps_1^3,
\end{equation}
and
\begin{equation}\label{yar22}
\Big|2\Re\int_1^t\int_{\R}A^2_{NR}(s,\xi)\big(\langle s\rangle^{3/2}\langle\xi\rangle^{-3/2}\big)\widetilde{g_a}(s,\xi)\overline{\widetilde{W_\ast}(s,\xi)}\,d\xi ds\Big|\lesssim_\delta\eps_1^3.
\end{equation}
\end{lemma} 

\begin{proof} {\bf{Step 1.}} We start with \eqref{yar21}-\eqref{yar21.7}. The two bounds are similar, so we only provide all the details for the estimate \eqref{yar21.7}. See also \cite[Lemma 6.5]{IOJI} for a similar argument.

We write the left-hand side of \eqref{yar21.7} in the form
\begin{equation*}
\begin{split}
&C\Big|2\Re\int_1^t\int_\R\int_\R A^2_R(s,\xi)\widetilde{\dot{V}}(s,\xi-\eta)(i\eta)\widetilde{(\Psi^\dagger\varrho_\ast)}(s,\eta)\overline{\widetilde{(\Psi^\dagger\varrho_\ast)}(s,\xi)}\,d\xi d\eta ds\Big|\\
&=C\Big|\int_1^t\int_\R\int_\R [\eta A^2_R(s,\xi)-\xi A_R^2(s,\eta)]\widetilde{\dot{V}}(s,\xi-\eta)\widetilde{(\Psi^\dagger\varrho_\ast)}(s,\eta)\overline{\widetilde{(\Psi^\dagger\varrho_\ast)}(s,\xi)}\,d\xi d\eta ds\Big|,
\end{split}
\end{equation*}
using symmetrization and the fact that $\dot{V}$ is real-valued. As in \eqref{nar18.1}--\eqref{nar18.4}, we define the sets
\begin{equation}\label{tol4}
\begin{split}
&S_0:=\Big\{(\xi,\eta)\in\R^2:\,\min(\langle\xi\rangle,\,\langle\eta\rangle,\,\langle\xi-\eta\rangle)\geq \frac{\langle\xi\rangle+\langle\eta\rangle+\langle\xi-\eta\rangle}{20}\Big\},\\
&S_1:=\Big\{(\xi,\eta)\in\R^2:\,\langle\xi-\eta\rangle\leq \frac{\langle\xi\rangle+\langle\eta\rangle+\langle\xi-\eta\rangle}{10}\Big\},\\
&S_2:=\Big\{(\xi,\eta)\in\R^2:\,\langle\eta\rangle\leq \frac{\langle\xi\rangle+\langle\eta\rangle+\langle\xi-\eta\rangle}{10}\Big\},\\
&S_3:=\Big\{(\xi,\eta)\in\R^2:\,\langle\xi\rangle\leq \frac{\langle\xi\rangle+\langle\eta\rangle+\langle\xi-\eta\rangle}{10}\Big\}.
\end{split}
\end{equation}
and the corresponding integrals
\begin{equation}\label{tol5}
\begin{split}
\mathcal{I}_n:=\int_1^t\int_\R\int_\R \mathbf{1}_{S_n}(\xi,\eta)&|\eta A^2_R(s,\xi)-\xi A_R^2(s,\eta)|\,|\widetilde{\dot{V}}(s,\xi-\eta)|\\
&\times|\widetilde{(\Psi^\dagger\varrho_\ast)}(s,\eta)|\,|\widetilde{(\Psi^\dagger\varrho_\ast)}(s,\xi)|\,d\xi d\eta ds.
\end{split}
\end{equation}
For \eqref{yar21.7} it suffices to prove that
\begin{equation}\label{tol6}
\mathcal{I}_n\lesssim_\delta \eps_1^3\qquad\text{ for }n\in\{0,1,2,3\}.
\end{equation}

As in Proposition \ref{BootImp1}, we use estimates on the weights. Letting $\delta'_0=\delta_0/200$, we have:

$\bullet\,\,$ If $(\xi,\eta)\in S_0\cup S_1$, $\rho=\xi-\eta$, $s\geq 1$, $\alpha\in[0,4]$, and $Y\in\{NR,R\}$ then
\begin{equation}\label{tol7}
\begin{split}
|\eta A^2_Y(s,\xi)&\langle\xi\rangle^{-\alpha}-\xi A_Y^2(s,\eta)\langle\eta\rangle^{-\alpha}|\\
&\lesssim_\delta s^{1.6}\frac{\sqrt{|(A_Y\dot{A}_Y)(s,\xi)|}}{\langle\xi\rangle^{\alpha/2}}\frac{\sqrt{|(A_Y\dot{A}_Y)(s,\eta)|}}{\langle\xi\rangle^{\alpha/2}}\cdot A_{NR}(s,\rho)e^{-\delta'_0\langle\rho\rangle^{1/2}}.
\end{split}
\end{equation}

$\bullet\,\,$ If $(\xi,\eta)\in S_2$, $\rho=\xi-\eta$, and $s\geq 1$ then
\begin{equation}\label{TLY2.2}
\langle\eta\rangle A^2_{R}(s,\xi)\lesssim_\delta s^{1.1}\langle\xi\rangle^{0.6}\sqrt{|(A_R\dot{A}_R)(s,\xi)|}\sqrt{|(A_{NR}\dot{A}_{NR})(s,\rho)|}\cdot A_{R}(s,\eta)e^{-\delta'_0\langle\eta\rangle^{1/2}}
\end{equation}
and
\begin{equation}\label{TLY2.3}
\langle\eta\rangle A^2_{NR}(s,\xi)\lesssim_\delta s^{1.1}\langle\xi\rangle^{-0.4}\sqrt{|(A_{NR}\dot{A}_{NR})(s,\xi)|}\sqrt{|(A_{NR}\dot{A}_{NR})(s,\rho)|}\cdot A_{NR}(s,\eta)e^{-\delta'_0\langle\eta\rangle^{1/2}}.
\end{equation}

See \cite[Lemma 8.9]{IOJI} for the proof.

For $n\in\{0,1\}$ we can now estimate, using \eqref{tol7},
\begin{equation*}
\begin{split}
\mathcal{I}_n\lesssim_\delta \Big\|\sqrt{|(A_R\dot{A}_R)(s,\xi)|}&\widetilde{(\Psi^\dagger\varrho_\ast)}(s,\xi)\Big\|_{L^2_sL^2_\xi}\Big\|\sqrt{|(A_R\dot{A}_R)(s,\eta)|}\widetilde{(\Psi^\dagger\varrho_\ast)}(s,\eta)\Big\|_{L^2_sL^2_\eta}\\
&\times\Big\|s^{1.6}A_{NR}(s,\rho)\langle\rho\rangle^2 e^{-\delta'_0\langle\rho\rangle^{1/2}}\widetilde{\dot{V}}(s,\rho)\Big\|_{L^\infty_sL^2_\rho},
\end{split}
\end{equation*}
and the bounds \eqref{tol6} follow for $n\in\{0,1\}$ from \eqref{boot2} and \eqref{nar7}. Similarly, for $n=2$ we use \eqref{TLY2.2} and \eqref{eq:comparisonweights1} to estimate
\begin{equation*}
\begin{split}
\mathcal{I}_2\lesssim_\delta \Big\|\sqrt{|(A_R\dot{A}_R)(s,\xi)|}&\widetilde{(\Psi^\dagger\varrho_\ast)}(s,\xi)\Big\|_{L^2_sL^2_\xi}\Big\|s^{1.1}\langle\rho\rangle^{0.6}\sqrt{|(A_{NR}\dot{A}_{NR})(s,\rho)|} \widetilde{\dot{V}}(s,\rho)\Big\|_{L^2_sL^2_\rho}\\
&\times\Big\|A_R(s,\eta)\langle\eta\rangle e^{-\delta'_0\langle\eta\rangle^{1/2}}\widetilde{(\Psi^\dagger\varrho_\ast)}(s,\eta)\Big\|_{L^\infty_sL^2_\eta},
\end{split}
\end{equation*}
and the desired bounds follow from \eqref{boot2} and \eqref{nar7}. The case $n=3$ is similar, by changes of variables, which completes the proof of \eqref{yar21.7}.

{\bf{Step 2.}} The bounds \eqref{yar22} for $a=1$ are similar, using symmetrization, the bounds \eqref{tol7} with $Y=NR$, and the bounds \eqref{TLY2.3}. See also \cite[Lemma 6.6]{IOJI} for a similar argument. Finally, the bounds \eqref{yar22} for $a\in\{2,3\}$ follow from \eqref{yar24}, \eqref{boot2}, and the Cauchy inequality.
\end{proof}

We can now prove the bounds \eqref{boot4} in Proposition \ref{MainBootstrap}.

\begin{lemma}\label{extras}
For any $t\in[1,T]$ we have
\begin{equation}\label{extras1}
|P'(t)|+\|\Psi^\dagger\cdot P_z'(t)\|_{\mathcal{G}^{\delta_0,1/2}}+\|\Psi^\dagger\cdot P_v'(t)\|_{\mathcal{G}^{\delta_0,1/2}}\lesssim_\delta\eps_1^{3/2}e^{-0.1\delta_0t^{1/2}},
\end{equation}
\begin{equation}\label{extras2}
\langle t\rangle\|W_\ast(t)\|_{\mathcal{G}^{\delta_0,1/2}}+\langle t\rangle^2\|\dot{V}(t)\|_{\mathcal{G}^{\delta_0,1/2}}\lesssim_\delta \eps_1^{3/2},
\end{equation}
\begin{equation}\label{extras3}
\|\langle(\partial_t+\dot{V}\partial_v)F\rangle(t)\|_{\mathcal{G}^{\delta_0,1/2}}\lesssim_\delta \eps_1^{3/2}\langle t\rangle^{-3}.
\end{equation}
\end{lemma}

\begin{proof} The bounds \eqref{extras1} follow from \eqref{Pop3} and the improved bootstrap estimates \eqref{nar1}. We prove now the bounds \eqref{extras2}. The point is to obtain the full $\langle t\rangle^{-2}$ time decay for the functions $W_\ast/t$ and $\dot{V}$, in a weaker topology. Using \eqref{PeH} we write
\begin{equation}\label{rew20}
\begin{split}
&\partial_t(tW_\ast)=-t\dot{V}\partial_vW_\ast+tH_1+tH_2,\\
&H_1:=\varrho^2V'\partial_v(\langle\partial_z\phi\,F\rangle),\qquad H_2:=\varrho V'\langle P'_v(\partial_v-t\partial_z)F\rangle+\varrho^2 \langle P'_z\,\partial_zF\rangle.
\end{split}
\end{equation}
We have $\|\partial_vW_\ast(t)/t\|_{\mathcal{G}^{\delta_0,1/2}}+\|\dot{V}(t)\|_{\mathcal{G}^{\delta_0,1/2}}\lesssim_\delta\eps_1\langle t\rangle^{-7/4}$ (as consequences of \eqref{nar7} and \eqref{rew11}) and $\|\Psi^\dagger P_z'\|_{\mathcal{G}^{\delta_0,1/2}}+\|\Psi^\dagger P_v'\|_{\mathcal{G}^{\delta_0,1/2}}\lesssim_\delta \eps_1e^{-(\delta_0/10)\langle t\rangle^{1/2}}$ (due to \eqref{Par1.7}). Thus, for any $t\in[1,T]$,
\begin{equation}\label{rew21}
\|-t\dot{V}(t)\partial_vW_\ast(t)\|_{\mathcal{G}^{\delta_0,1/2}}+\|tH_2(t)\|_{\mathcal{G}^{\delta_0,1/2}}\lesssim_\delta\eps_1^2\langle t\rangle^{-3/2}.
\end{equation}

We would like to prove now similar bounds for the function $H_1$. Using \eqref{Pephi} we have
\begin{equation*}
\begin{split}
\langle\partial_z\phi\,F\rangle&=\langle\partial_z\phi\cdot\{\varrho^2\partial_z^2\phi+(V')^2(\partial_v-t\partial_z)^2\phi+V''(\partial_v-t\partial_z)\phi+\varrho V'(\partial_v-t\partial_z)\phi\}\rangle\\
&=(V')^2\langle\partial_z\phi\,\partial_v^2\phi\rangle-2t(V')^2\langle\partial_z\phi\,\partial_z\partial_v\phi\rangle+(V''+\varrho V')\langle\partial_z\phi(\partial_v-t\partial_z)\phi\rangle.
\end{split}
\end{equation*}
The main point is the vanishing of the term containing $t^2$. Moreover, it follows from Lemma \ref{nar13} that $\|\langle\nabla\rangle^4\mathbb{P}_{\neq 0}\phi\|_{\mathcal{G}^{\delta_0,1/2}}\lesssim_\delta\eps_1\langle t\rangle^{-2}$. Therefore, for any $t\in[1,T]$,
\begin{equation}\label{rew22}
\|\langle\nabla\rangle^4\langle\partial_z\phi\,F\rangle(t)\|_{\mathcal{G}^{\delta_0,1/2}}+\|H_1(t)\|_{\mathcal{G}^{\delta_0,1/2}}\lesssim_\delta\eps_1^2\langle t\rangle^{-3}.
\end{equation}

The desired estimate for $\langle t\rangle\|W_\ast(t)\|_{\mathcal{G}^{\delta_0,1/2}}$ holds using \eqref{rew20}--\eqref{rew22} and \eqref{boot1}. To bound $\langle t\rangle^2\|\dot{V}(t)\|_{\mathcal{G}^{\delta_0,1/2}}$ we use first \eqref{rew8}, \eqref{rew9}, and \eqref{nar4}, thus $\|\partial_v(e^H\dot{V})\|_{\mathcal{G}^{\delta_0,1/2}}\lesssim_\delta\eps_1^{3/2}\langle t\rangle^{-2}$. The desired bound follows using the uncertainty principle, as in the proof of \eqref{rew10} above.

Finally, to prove \eqref{extras3} we examine the formula \eqref{Pef}. The desired conclusion follows from \eqref{rew22} and \eqref{extras1}.
\end{proof}

\section{Bootstrap estimates, II: improved control of the variable $\phi$}\label{coimprov1}

We now prove the main bounds (\ref{boot3'}) on  the function $\phi$. More precisely,
\begin{proposition}\label{ImpTheta}
With the definitions and assumptions in Proposition \ref{MainBootstrap}, we have
\begin{equation}\label{Imboot3'}
\E_{\phi}(t)+\mathcal{B}_{\phi}(t)\lesssim_{\delta}\epsilon_1^3\leq\epsilon_1^2/2,\qquad{\rm for\,\,any\,\,}t\in[1,T].
\end{equation}
\end{proposition}

The rest of the section is concerned with the proof of Proposition \ref{ImpTheta}, which consists of two main steps. The first step is to control the contribution of low frequencies, where the weight $A_k$ does not play a role, and the estimates follow from standard theory of elliptic equations.  In the second and main step, we commute the elliptic equation (\ref{Pephi}) with the Fourier multiplier $A_k$ and use the estimates for low frequencies obtained in the first step, to obtain the weighted estimates on the high frequencies. The key is to control the commutator terms, and show that these commutator terms are perturbative when the frequency is sufficiently high. 

We begin with an estimate which is effective in controlling low frequencies.
\begin{lemma}\label{ImpThl}
With the definitions and assumptions in Proposition \ref{MainBootstrap}, for $t\in[1,T]$ we have
\begin{equation}\label{Imboot3'l}
\int_{\mathbb{T}\times\mathbb{R}}\big|[\partial_z^2+(\partial_v-t\partial_z)^2](\Psi\phi)(t,z,v)\big|^2dzdv\lesssim \int_{\mathbb{T}\times\mathbb{R}}|F(t,z,v)|^2dzdv\lesssim_\delta\eps_1^3.
\end{equation}
\end{lemma}

\begin{proof}
The last inequality follows from the improved bootstrap bounds in Proposition \ref{BootImp1}. For the first inequality, in view of the changes of variables (\ref{rea21}) and \eqref{u'In} it suffices to show that 
\begin{equation}\label{Imbl3.1}
\int_{B_{4/\vartheta_0}\backslash B_{\vartheta_0/4}}|\psi(t,x,y)|^2+|\nabla \psi(t,x,y)|^2+|\nabla^2\psi(t,x,y)|^2\,dxdy\lesssim \int_{\mathbb{R}^2}|\omega(t,x,y)|^2dxdy,
\end{equation}
which follows by elliptic theory from the defining formula $\Delta\psi=\omega$.
\end{proof}

We now proceed to prove (\ref{Imboot3'}), using (\ref{Imbl3.1}). From the equation (\ref{Pephi}), localizing $\phi$ using $\Psi$, we get that
\begin{equation}\label{Lophi1}
\begin{split}
&\varrho^2\partial_z^2\big(\Psi\phi\big)+(V')^2(\partial_v-t\partial_z)^2\big(\Psi\phi\big)+V''(\partial_v-t\partial_z)\big(\Psi\phi\big)+\varrho V' (\partial_v-t\partial_z)\big(\Psi\phi\big)\\
&=\Psi F+2(V')^2\partial_v\Psi (\partial_v-t\partial_z)\phi+[(V')^2\partial_v^2\Psi+V''\partial_v\Psi+\varrho V'\partial_v\Psi]\cdot\phi\\
&=\Psi F+g_{1}+g_{2},
\end{split}
\end{equation}
where 
\begin{equation}\label{thegj}
\begin{split}
&g_1:=2(V')^2\partial_v\Psi (\partial_v-t\partial_z)\phi,\qquad g_2:=[(V')^2\partial_v^2\Psi+V''\partial_v\Psi+\varrho V'\partial_v\Psi]\phi.
\end{split}
\end{equation}
Let $\mathbb{Z}^\ast:=\mathbb{Z}\setminus\{0\}$ and let $A^\dagger$ denote the operator defined by the Fourier multiplier
\begin{equation}\label{Aast}
A^{\dagger}(t,k,\xi):=\mathbf{1}_{\mathbb{Z}^\ast}(k)A_k(t,\xi)\frac{\langle t\rangle}{\langle t\rangle+|\xi/k|}
\end{equation}
and fix a Gevrey cutoff function $\Psi_0$ supported in $\big[\ubv/5,5\obv\big]$, equal to $1$ in $\big[\ubv/4,4\obv\big]$, and satisfying $\big\|e^{\langle\xi\rangle^{3/4}}\widetilde{\Psi_0}(\xi)\big\|_{L^\infty}\lesssim 1$. We apply the operator $\Psi_0A^\dagger$ to (\ref{Lophi1}) to obtain
\begin{equation}\label{Lophi1.1}
\begin{split}
\varrho^2\partial_z^2(\Psi_0A^\dagger)\big(\Psi\phi\big)&+(V')^2(\partial_v-t\partial_z)^2(\Psi_0A^\dagger)\big(\Psi\phi\big)+(V''+\varrho V')(\partial_v-t\partial_z)(\Psi_0A^\dagger)\big(\Psi\phi\big)\\
&=(\Psi_0A^\dagger)\big(\Psi F\big)+\Psi_0A^\dagger g_1+\Psi_0A^\dagger g_2+\sum_{j=1}^3\mathcal{C}_j.
\end{split}
\end{equation}
The commutator terms $\mathcal{C}_j$ are defined as
\begin{equation}\label{defC}
\begin{split}
&\mathcal{C}_1:=\varrho^2\partial_z^2(\Psi_0A^\dagger)\big(\Psi\phi\big)-(\Psi_0A^\dagger)\big[\varrho^2\partial_z^2\big(\Psi\phi\big)\big];\\
&\mathcal{C}_2:=(V')^2(\partial_v-t\partial_z)^2(\Psi_0A^\dagger)\big(\Psi\phi\big)-(\Psi_0A^\dagger)\big[(V')^2(\partial_v-t\partial_z)^2\big(\Psi\phi\big)\big];\\
&\mathcal{C}_3:=(V''+\varrho V')(\partial_v-t\partial_z)(\Psi_0A^\dagger)\big(\Psi\phi\big)-(\Psi_0A^\dagger)\big[(V''+\varrho V')(\partial_v-t\partial_z)\big(\Psi\phi\big)\big].
\end{split}
\end{equation}

By examining the definitions, we notice that, for any $t\in[1,T]$,
\begin{equation}\label{cros1}
\mathcal{E}_\phi(t)\approx\big\|[\partial_z^2+(\partial_v-t\partial_z)^2]A^{\dagger}\big(\Psi\phi\big)(t,z,v)\big\|_{L^2_{z,v}}.
\end{equation}
We will use the elliptic equation \eqref{Lophi1.1} to bound the energy functional $\mathcal{E}_\phi(t)$. 

Similarly, to control the space-time integrals $\mathcal{B}_\phi(t)$ we would like to apply suitable weighted operators to the identity \eqref{Lophi1}. We have to be slightly careful, because our weights need to have suitable smoothness in $\xi$ in order to estimate the resulting commutator terms. Therefore, we define $\mu^\#(t,\xi)$ for $t\ge 0$ and $\xi\geq 0$ as follows:
\begin{equation}\label{muD}
\mu^{\#}(t,\xi):=\left\{\begin{array}{ll}
                            0&{\rm if}\,\,|\xi|\leq \delta^{-10},\\
                            \delta^2& {\rm if}\,\,|\xi|>\delta^{-10} \,\,{\rm and}\,\,t<t_{k_0(\xi),\xi},\\
                            \frac{\delta^2}{1+\delta^2|t-\xi/k|}&{\rm if}\,\,|\xi|>\delta^{-10}\,\,{\rm and}\,\,t\in I_{k,\xi}, \,\,k\in\{1,2,\dots,k_0(\xi)\},\\
                            0&{\rm if}\,\,t>2|\xi|.
                         \end{array}\right.
\end{equation}
Then we define $\mu^\#(t,\xi):=\mu^\#(t,|\xi|)$ if $\xi\leq 0$. Compare with the definitions in subsection \ref{weightsdefin}, in particular the formulas \eqref{reb8}. Then we define
 \begin{equation}\label{muaD}
 \mu^\ast(t,\xi):=\int_{\mathbb{R}}\mu^{\#}(t,\rho)\frac{1}{d_0L_{\delta'}(t,\xi)}\varphi\bigg(\frac{\xi-\rho}{L_{\delta'}(t,\xi)}\bigg)\,d\rho,\qquad L_{\delta'}(t,\xi):=1+\frac{\delta'\langle\xi\rangle}{\langle\xi\rangle^{1/2}+\delta' t},
 \end{equation}
 where $\delta'=\delta'(\delta)\in(0,1)$ is chosen sufficiently small. Compare with \eqref{dor1}. Finally, we define
 \begin{equation}\label{mu1}
 \begin{split}
 \mu_k(t,\xi)&:=\frac{\langle k,\xi\rangle^{1/2}}{\langle t\rangle^{1+\sigma_0}}+\frac{\mu^{\ast}(t,\xi)}{1+e^{\sqrt{\delta}(|k|^{1/2}-\langle\xi\rangle^{1/2})}b_k(t,\xi)},\\
 \mu_R(t,\xi)&:=\frac{\langle\xi\rangle^{1/2}}{\langle t\rangle^{1+\sigma_0}}+\mu^{\ast}(t,\xi).
 \end{split}
 \end{equation}
These definitions are motivated by the formulas \eqref{TLX3.5}--\eqref{eq:A_kxi}. All the required properties of the weights $\mu_k$ and $\mu_R$ are proved in subsection \ref{musection}.

Let $B^\dagger$ denote the operator defined by the Fourier multiplier
\begin{equation}\label{Bast}
B^{\dagger}(t,k,\xi):= \mu_k(t,\xi)^{1/2}A^\dagger(t,k,\xi).
\end{equation}
We apply the operator $\Psi_0 B^\dagger$ to equation (\ref{Lophi1}) to obtain
\begin{equation}\label{Lophi1.11}
\begin{split}
\varrho^2\partial_z^2(\Psi_0B^\dagger)\big(\Psi\phi\big)&+(V')^2(\partial_v-t\partial_z)^2(\Psi_0B^\dagger)\big(\Psi\phi\big)+(V''+\varrho V')(\partial_v-t\partial_z)(\Psi_0B^\dagger)\big(\Psi\phi\big)\\
&=(\Psi_0B^\dagger)\big(\Psi F\big)+\Psi_0B^\dagger g_1+\Psi_0B^\dagger g_2+\sum_{j=1}^3\mathcal{D}_j.
\end{split}
\end{equation}
The commutator terms $\mathcal{D}_j$ are defined as
\begin{equation}\label{defD}
\begin{split}
&\mathcal{D}_1:=\varrho^2\partial_z^2(\Psi_0B^\dagger)\big(\Psi\phi\big)-(\Psi_0B^\dagger)\big[\varrho^2\partial_z^2\big(\Psi\phi\big)\big];\\
&\mathcal{D}_2:=(V')^2(\partial_v-t\partial_z)^2(\Psi_0B^\dagger)\big(\Psi\phi\big)-(\Psi_0B^\dagger)\big[(V')^2(\partial_v-t\partial_z)^2\big(\Psi\phi\big)\big];\\
&\mathcal{D}_3:=(V''+\varrho V')(\partial_v-t\partial_z)(\Psi_0B^\dagger)\big(\Psi\phi\big)-(\Psi_0B^\dagger)\big[(V''+\varrho V')(\partial_v-t\partial_z)\big(\Psi\phi\big)\big].
\end{split}
\end{equation}

In view of Lemma \ref{Lmu1} (i) we notice that that for $t\in[1,T]$
\begin{equation}\label{w0.5}
\begin{split}
\E_{\phi}(t)&\approx_{\delta}\E'_{\phi}(t):=\|[\partial_z^2+(\partial_v-t\partial_z)^2]A^\dagger(\Psi\phi)(t)\|^2_{L^2(\mathbb{T}\times\mathbb{R})}\\
\mathcal{B}_{\phi}(t)&\approx_{\delta}\B'_{\phi}(t):=\int_1^t\|[\partial_z^2+(\partial_v-s\partial_z)^2]B^\dagger(\Psi\phi)(s)\|^2_{L^2(\mathbb{T}\times\mathbb{R})}ds.
\end{split}
\end{equation}
Using \eqref{Lophi1.1}, \eqref{Lophi1.11}, and a change of variables as in the proof of Lemma \ref{ImpThl}, we have
\begin{equation}\label{w1.0}
\begin{split}
\big\|[\partial_z^2+(\partial_v-t\partial_z)^2](\Psi_0A^\dagger)&(\Psi\phi)(t)\big\|^2_{L^2(\mathbb{T}\times\mathbb{R})}+\int_1^t\big\|[\partial_z^2+(\partial_v-s\partial_z)^2](\Psi_0B^\dagger)(\Psi\phi)(s)\big\|^2_{L^2(\mathbb{T}\times\mathbb{R})}ds\\
  &\lesssim \sum_{h\in\{\Psi F, g_1,g_2\}}\Big[\|\Psi_0A^\dagger h(t)\|^2_{L^2(\mathbb{T}\times\mathbb{R})}+\int_1^t\|\Psi_0B^\dagger h(s)\|^2_{L^2(\mathbb{T}\times\mathbb{R})}ds\Big]\\
&+\sum_{j=1}^3\Big[\|\mathcal{C}_j(t)\|^2_{L^2(\mathbb{T}\times\mathbb{R})}+\int_1^t\|\mathcal{D}_j(s)\|^2_{L^2(\mathbb{T}\times\mathbb{R})}ds\Big],
\end{split}
\end{equation}
for any $t\in[1,T]$. In view of \eqref{w0.5}-(\ref{w1.0}), to prove (\ref{Imboot3'l}) it suffices to prove that
\begin{equation}\label{wg1}
\sum_{h\in\{\Psi F, g_1,g_2\}}\Big[\|A^\dagger h(t)\|^2_{L^2(\mathbb{T}\times\mathbb{R})}+\int_1^t\|B^\dagger h(s)\|^2_{L^2(\mathbb{T}\times\mathbb{R})}ds\Big]\lesssim C(\delta)\epsilon_1^3
\end{equation}
and
\begin{equation}\label{wCD}
\begin{split}
\sum_{j=1}^4\Big[\|\mathcal{C}_j(t)&\|^2_{L^2(\mathbb{T}\times\mathbb{R})}+\int_1^t\|\mathcal{D}_j(s)\|^2_{L^2(\mathbb{T}\times\mathbb{R})}ds\Big]\lesssim C(\delta)\epsilon_1^3+\sqrt{\delta}\big[\E'_{\phi}(t)+\B'_{\phi}(t)\big],
\end{split}
\end{equation}
 for any $t\in[1,T]$, where the commutator terms $\mathcal{C}_1,\mathcal{C}_2,\mathcal{C}_3,\mathcal{D}_1,\mathcal{D}_2,\mathcal{D}_3$ are defined in \eqref{defC} and \eqref{defD}, and the additional commutator terms $\mathcal{C}_4,\mathcal{D}_4$ are defined by
\begin{equation}\label{defCD4}
\begin{split}
\mathcal{C}_4:=&[\partial_z^2+(\partial_v-t\partial_z)^2](\Psi_0A^\dagger)(\Psi\phi)-[\partial_z^2+(\partial_v-t\partial_z)^2](A^\dagger\Psi_0)(\Psi\phi),\\
\mathcal{D}_4:=&[\partial_z^2+(\partial_v-t\partial_z)^2](\Psi_0B^\dagger)(\Psi\phi)-[\partial_z^2+(\partial_v-t\partial_z)^2](B^\dagger\Psi_0)(\Psi\phi).
\end{split}
\end{equation}

The bounds for \eqref{wg1} for $h=\Psi F$ follow from the improved bootstrap bounds \eqref{nar1} on $F$ and the bilinear weighted bounds in Lemma \ref{lm:Multi} (ii). In the rest of the section we first prove the bounds \eqref{wg1} for $g_1,g_2$ and then we prove the commutator estimates \eqref{wCD}.

\subsection{Bounds \eqref{wg1} on the terms $A^{\dagger}g_j, \,B^{\dagger}g_j$, $j=1,2,$} We prove now the bounds (\ref{wg1}) for the functions $g_1$ and $g_2$. The main difficulty is that the boostrap assumption on $\phi$ gives information on the localized stream function $\Psi\phi$, not on $\phi$ itself. To compensate, we need to use the compact support assumption of the normalized vorticity function $F$.

Let $\phi_k(t,v)$ be the $k-$th Fourier mode of $\phi$ in $z$, i.e.,
\begin{equation}\label{phik}
\phi_k(t,v)=\frac{1}{2\pi}\int_{\mathbb{T}}\phi(t,z,v)e^{-ikz}dz.
\end{equation}
We define similarly $f_k(t,v)$ as the $k-$th Fourier mode of $f$ in $z$. In addition, let $\psi'_k(t,r),\omega'_k(t,r)$ denote the $k-$th Fourier mode of $\psi'(t,r,\theta),\omega'(t,r,\theta)$ in $\theta$, respectively. To calculate $\phi$ we use the relations
\begin{equation}\label{prpv}
\psi'_k(t,r)=\phi_k(t,v(t,r))e^{-iktv(t,r)}\qquad {\rm and}\qquad \omega'_k(t,r)=F_k(t,v(t,r))e^{-iktv(t,r)},
\end{equation}
which are consequences of the change of variables (\ref{PCoC})-(\ref{rea21}). Taking Fourier transform in the equation (\ref{St}) for $\psi'$ and using (\ref{prpv}), we obtain that
\begin{equation}\label{Stk}
\partial_r^2\psi'_k(t,r)+\frac{1}{r}\partial_r\psi'_k(t,r)-\frac{k^2\psi'_k}{r^2}=F_k(t,v(t,r))e^{-iktv(t,r)}.
\end{equation}
Define for $k\in\mathbb{Z}\backslash\{0\}, r,r'\in(0,\infty)$,
\begin{equation}\label{G.k}
G_k(r,r'):=\left\{\begin{array}{ll}
                         -\frac{r'}{2|k|}\big(\frac{r}{r'}\big)^{|k|}&{\rm for}\,\,r<r';\\
                          -\frac{r'}{2|k|}\big(\frac{r'}{r}\big)^{|k|}&{\rm for}\,\,r>r'.
                        \end{array}\right.
\end{equation}
The kernel $G_k$ is the fundamental solution to (\ref{Stk}) when $k\in\mathbb{Z}\backslash\{0\}$ and we can express $\psi'_k$, $k\in\mathbb{Z}\backslash\{0\}$, as 
\begin{equation}\label{phi.k1}
\psi'_k(t,r)=\int_{\mathbb{R}}G_k(r,r')F_k(t,v(t,r'))e^{-iktv(t,r')}dr'.
\end{equation}
Let $\mathcal{G}_k(t,v,v')$ for $k\in\mathbb{Z}\backslash\{0\}, v,v'\in[\underline{v}/10,10\overline{v}]$ be defined by
\begin{equation}\label{G.k1}
\mathcal{G}_k(t,v(t,r),v(t,r'))=G_k(r,r')\qquad {\rm for }\,\,r, r'\in(0,\infty).
\end{equation}
Using (\ref{prpv}) and the change of variable $r'\to v(t,r')$ in the integral of (\ref{phi.k1}), we get
\begin{equation}\label{phi.k2}
\phi_k(t,v)=e^{iktv}\int_{\mathbb{R}}\mathcal{G}_k(t,v,v')F_k(t,v')e^{-iktv'}\frac{1}{V'(t,v')}\,dv'.
\end{equation}
This is our main formula we use to estimate the functions $g_1$ and $g_2$. We first prove suitable bounds on $\mathcal{G}_k$. 

\begin{lemma}\label{G.k4}
Assume $\Theta,\Theta':\mathbb{R}\to\mathbb{C}$ satisfy $\big\|e^{\langle\xi\rangle^{2/3}}\widetilde{\Theta}(\xi)\big\|_{L^\infty}+\big\|e^{\langle\xi\rangle^{2/3}}\widetilde{\Theta'}(\xi)\big\|_{L^\infty} \lesssim 1$ and have separated supports, i.e. assume that there is $a\in[\underline{v}/4,4\overline{v}]$ and a constant $c>0$ such that $\Theta$ is supported in $[\underline{v}/8,a-c]$ and $\Theta'$ is supported in $[a+c,8\overline{v}]$ (or vice versa). Let
\begin{equation}\label{G.k5}
\mathcal{G}^{\flat}_k(t,v,v'):=\Theta(v)\,\mathcal{G}_k(t,v,v')\,\Theta'(v').
\end{equation}
Then there is a small constant $c_0>0$ (independent of $\delta,k$) such that for any $t\in[0,T]$ we have
\begin{equation}\label{G.k6}
\begin{split}
&\int_{\mathbb{R}^2}A_R^2(t,\xi)A_R^2(t,\eta)\big|\widetilde{\mathcal{G}^{\flat}_k}(t,\xi,\eta)\big|^2 d\xi d\eta\\
&+\,\int_0^t\int_{\mathbb{R}^2}\Big\{\big|\dot{A}_RA_R(s,\xi)\big|A_{R}^2(t,\eta)+\big|\dot{A}_RA_R(s,\eta)\big|A_R^2(t,\xi)\Big\}\big|\widetilde{\mathcal{G}^{\flat}_k}(s,\xi,\eta)\big|^2 d\xi d\eta ds\lesssim e^{-c_0|k|}.
\end{split}
\end{equation}
\end{lemma} 
 
\begin{proof}
By the definitions \eqref{rea22'} and \eqref{G.k}, we have
\begin{equation}\label{Gkb}
\mathcal{G}_k(t,v,v')=\left\{ \begin{array}{rl}
                               -\frac{1}{2|k|}\varrho(t,v')^{|k|-1}\varrho(t,v)^{-|k|}&\qquad {\rm if}\,\,v\ge v'>0;\\
                               &\\
                               -\frac{1}{2|k|}\varrho(t,v)^{|k|}\varrho(t,v')^{-1-|k|} &\qquad {\rm if}\,\,0<v\leq v'.
                               \end{array}\right.              
\end{equation}
To fix ideas, assume that $\Theta$ is supported in $[\underline{v}/8,a-c]$ and $\Theta'$ is supported in $[a+c,8\overline{v}]$ (the other case is similar). Then, using also \eqref{rea22'},
\begin{equation}\label{G.k8}
\begin{split}
\Theta(v)\,&\mathcal{G}_k(t,v,v')\,\Theta'(v')=-\frac{\Theta(v)\Theta'(v')}{2|k|}\frac{(2\pi v/\kappa)^{|k|/2}[1+\sqrt{\kappa/(2\pi v)}\,\varrho_{\ast}(t,v)]^{|k|}}{(2\pi v'/\kappa)^{(|k|+1)/2}[1+\sqrt{\kappa/(2\pi v')}\,\varrho_{\ast}(t,v')]^{|k|+1}}\\
&=-\frac{\sqrt{\kappa/(2\pi)}}{2|k|}\frac{(a-c/2)^{|k|/2}}{(a+c/2)^{(|k|+1)/2}}\times\frac{\Theta(v)v^{|k|/2}}{(a-c/2)^{|k|/2}}[1+\sqrt{\kappa/(2\pi v)}\,\Psi^\dagger(v)\varrho_{\ast}(t,v)]^{|k|}\\
&\times \frac{\Theta'(v')(a+c/2)^{(|k|+1)/2}}{(v')^{(|k|+1)/2}[1+\sqrt{\kappa/(2\pi v')}\,\Psi^\dagger(v')\varrho_{\ast}(t,v')]^{|k|+1}}.
\end{split}
\end{equation}
Thus $\Theta(v)\,\mathcal{G}_k(t,v,v')\,\Theta'(v')$ has product structure in the variables $v$ and $v'$. Moreover, it follows from Lemma \ref{lm:Gevrey} that, for any $\xi\in\mathbb{R}$,
\begin{equation*}
\Big|\mathcal{F}\Big\{\frac{\Theta(v)v^{|k|/2}}{(a-c/2)^{|k|/2}}\Big\}(\xi)\Big|+\Big|\mathcal{F}\Big\{\frac{\Theta'(v')(a+c/2)^{(|k|+1)/2}}{(v')^{(|k|+1)/2}}\Big\}(\xi)\Big|\lesssim e^{-\mu|\xi|^{2/3}}.
\end{equation*}
Notice that the factor $\frac{(a-c/2)^{|k|/2}}{(a+c/2)^{(|k|+1)/2}}$ provides exponential decay in $|k|$. Recall that $\|\Psi^\dagger\varrho_\ast\|_{R}\lesssim \eps_1$, due to the bootstrap assumption \eqref{boot2}. The desired conclusion follows from the algebra property of the space $R$ in Lemma \ref{nar8} (i), provided that $\eps_1$ is sufficiently small relative to $\delta$.
\end{proof}

We can now complete the proof of the bounds \eqref{wg1}.

\begin{lemma}\label{lemmawg}
With $g_j$, $j\in\{1,2\}$ as in \eqref{thegj}, for any $t\in[1,T]$ we have
\begin{equation}\label{wg1*}
\|A^\dagger g_j(t)\|^2_{L^2(\mathbb{T}\times\mathbb{R})}+\int_1^t\|B^\dagger g_j(s)\|^2_{L^2(\mathbb{T}\times\mathbb{R})}ds\lesssim_\delta\epsilon_1^3.
\end{equation}
\end{lemma}

\begin{proof} Let $F'(t,z,v):=F(t,z,v)/V'(t,v)$. As in Lemma \ref{nar8}, using Lemma \ref{lm:Multi} (ii), Lemma \ref{Multi0} and the bounds \eqref{nar4}, the function $F'$ satisfies similar bounds as the function $F$,
\begin{equation}\label{w.06}
\sum_{k\in\mathbb{Z}}\int_{\R}A_k^2(t,\eta)\big|\widetilde{F'}(t,k,\eta)\big|^2d\eta+\sum_{k\in\mathbb{Z}}\int_1^t\int_{\R}|\dot{A}_kA_k(s,\eta)|\big|\widetilde{F'}(s,k,\eta)\big|^2\,d\eta ds\lesssim_{\delta}\epsilon_1^3,
\end{equation}
for any $t\in[1,T]$. 

{\bf{Step 1.}} Let $\Theta$ denote a Gevrey cutoff function supported in $[\underline{v}/3.4,\underline{v}/1.6]\cup[1.6\overline{v},3.4\overline{v}]$ and equal to $1$ in $[\underline{v}/3.2,\underline{v}/1.8]\cup[1.8\overline{v},3.2\overline{v}]$ (such that $\Theta(v)g_j(t,z,v)=g_j(t,z,v)$, $j\in\{1,2\}$), and let $\Theta'$ denote a Gevrey cutoff function supported in  $[\underline{v}/1.4,1.4\overline{v}]$ and equal to $1$ in $[\underline{v}/1.2,1.2\overline{v}]$ (such that $\Theta'(v')F'(t,z',v')=F'(t,z',v')$, see \eqref{Psupp}). As in Lemma \ref{G.k4} we assume that $\big\|e^{\langle\xi\rangle^{2/3}}\widetilde{\Theta}(\xi)\big\|_{L^\infty}+\big\|e^{\langle\xi\rangle^{2/3}}\widetilde{\Theta'}(\xi)\big\|_{L^\infty} \lesssim 1$ and define $\mathcal{G}^{\flat}_k$ as in \eqref{G.k5}. 

In view of \eqref{phi.k2} and taking the Fourier transform in $v$, we obtain
\begin{equation}\label{phi.k3}
\big|\widetilde{(\Theta\phi)}(t,k,\xi)\big|\lesssim\int_{\mathbb{R}}\big|\widetilde{F'}(t,k,\eta)\big|\big|\widetilde{G_k^{\flat}}(t,\xi-kt,kt-\eta)\big|\,d\eta.
\end{equation}

Using Lemma \ref{lm:Multi} we have
\begin{equation}\label{w.07}
\begin{split}
A_k(t,\xi)&\lesssim_{\delta}A_k(t,\eta)A_R(t,\xi-\eta)e^{-(\lambda(t)/20)\min(\langle\xi-\eta\rangle, \langle k,\eta\rangle)^{1/2}}\\
              &\lesssim_{\delta}A_k(t,\eta)A_R(t,\xi-kt)A_R(t,kt-\eta)e^{-(\lambda(t)/20)[\min(\langle\xi-\eta\rangle, \langle k,\eta\rangle)+\min(\langle\xi-kt\rangle,\langle kt-\eta\rangle)]^{1/2}}
\end{split}
\end{equation}
and
\begin{equation}\label{w.08}
\begin{split}
|\dot{A}_kA_k(t,\xi)|^{1/2}&\lesssim_{\delta}\big\{|\dot{A}_RA_R(t,\xi-\eta)|^{1/2}A_k(t,\eta)+|\dot{A}_kA_k(t,\eta)|^{1/2}A_R(t,\xi-\eta)\big\}\\
&\qquad\qquad \times e^{-(\lambda(t)/30)\min(\langle\xi-\eta\rangle, \langle k,\eta\rangle)^{1/2}}\\
              &\lesssim_{\delta}\left\{|\dot{A}_RA_R(t,\xi-kt)|^{1/2}A_k(t,\eta)A_R(t,kt-\eta)\right.\\
              &\qquad\qquad+|\dot{A}_RA_R(t,kt-\eta)|^{1/2}A_k(t,\eta)A_R(t,\xi-kt)\\
               &\qquad\qquad\left.+|\dot{A}_kA_k(t,\eta)|^{1/2}A_R(t,\xi-kt)A_R(t,kt-\eta)\right\}\\
               &\qquad\qquad\times e^{-(\lambda(t)/30)\min(\langle\xi-\eta\rangle, \langle k,\eta\rangle)^{1/2}}e^{-(\lambda(t)/30)\min(\langle\xi-kt\rangle,\langle kt-\eta\rangle)^{1/2}}.
\end{split}
\end{equation}
Moreover, it is easy to see that for any $k\in\mathbb{Z}\backslash\{0\}, \xi, \eta\in \R, t\ge 1$, we have
\begin{equation}\label{w.09}
\begin{split}
\frac{|k|t}{|k|t+|\xi|}\langle\xi-kt\rangle&\lesssim_{\delta}e^{\delta\min(\langle\xi-\eta\rangle, \langle k,\eta\rangle)^{1/2}}e^{\delta\min(\langle\xi-kt\rangle,\langle kt-\eta\rangle)^{1/2}},\\
e^{\delta\min(\langle\xi-kt\rangle,\langle kt\rangle)^{1/2}}&\lesssim e^{10\delta\min(\langle\xi-\eta\rangle, \langle k,\eta\rangle)^{1/2}}e^{10\delta\min(\langle\xi-kt\rangle,\langle kt-\eta\rangle)^{1/2}}
\end{split}
\end{equation}

{\bf{Step 2.}} We would like to show now that for each $t\in[1,T]$ we have
\begin{equation}\label{w.11}
\begin{split}
\sum_{k\in\mathbb{Z}\setminus\{0\}}&\int_{\R}A_k^2(t,\xi)\nu(t,k,\xi)^2e^{2\delta|k|}\big|\widetilde{(\Theta\phi)}(t,k,\xi)\big|^2\,d\xi\\
&+\sum_{k\in\mathbb{Z}\setminus\{0\}}\int_0^t\int_{\R}|\dot{A}_kA_k(s,\xi)|\nu(s,k,\xi)^2e^{2\delta|k|}\big|\widetilde{(\Theta\phi)}(s,k,\xi)\big|^2 d\xi ds\lesssim_{\delta}\epsilon_1^3,
\end{split}
\end{equation}
where
\begin{equation}\label{w.115}
\nu(t,k,\xi):=\frac{|k|t}{|k|t+|\xi|}\langle\xi-kt\rangle e^{\delta\min(\langle\xi-kt\rangle,\langle kt\rangle)^{1/2}}.
\end{equation}

Indeed, the first term in the left-hand side of \eqref{w.11} is bounded by
\begin{equation*}
\begin{split}
\sup_{\|P\|_{L^2_kL^2_\xi}\leq 1}&\sum_{k\in\mathbb{Z}\setminus\{0\}}\int_{\R}A_k(t,\xi)\nu(t,k,\xi)e^{\delta|k|}\big|\widetilde{(\Theta\phi)}(t,k,\xi)\big||P(k,\xi)|\,d\xi\\
&\lesssim_\delta \sup_{\|P\|_{L^2_kL^2_\xi}\leq 1}\sum_{k\in\mathbb{Z}\setminus\{0\}}\int_{\R^2}A_k(t,\eta)\big|\widetilde{F'}(t,k,\eta)\big|A_R(t,\xi-kt)A_R(t,kt-\eta)\\
&\qquad\qquad\times\big|\widetilde{G_k^{\flat}}(t,\xi-kt,kt-\eta)\big|e^{\delta|k|}|P(k,\xi)|\,d\eta d\xi,
\end{split}
\end{equation*}
using \eqref{phi.k3}, \eqref{w.07}, and \eqref{w.09}. The desired inequality follows using the Cauchy-Schwartz inequality, and the bounds \eqref{G.k6} and \eqref{w.06}. The second term in \eqref{w.11} can be bounded in a similar way, using \eqref{w.08} instead of \eqref{w.07}, which completes the proof of \eqref{w.11}.

{\bf{Step 3.}} We prove now the bounds \eqref{wg1*}. Notice that
\begin{equation*}
g_1=2(V')^2\partial_v\Psi\cdot (\partial_v-t\partial_z)(\Theta\phi),
\end{equation*}
since the additional cutoff function is equal to $1$ on the support of $\partial_v\Psi$. To estimate $g_1$ we would like to use Lemma \ref{Multi0} (ii). The bilinear weighted bounds we need are
\begin{equation*}
\begin{split}
A_k(t,\xi)&\frac{\langle t\rangle}{|\xi/k|+\langle t\rangle}\lesssim_\delta A_R(t,\xi-\eta)\cdot A_k(t,\eta)\frac{\langle t\rangle}{|\eta/k|+\langle t\rangle}e^{-\delta'_0\min(\langle\xi-\eta\rangle,\langle k,\eta\rangle)^{1/2}},
\end{split}
\end{equation*}
for any $t\geq 1$, $\xi,\eta\in\mathbb{R}$, and $k\in\mathbb{Z}\setminus\{0\}$ (which follows from \eqref{TLX7}) and \eqref{vfc30.7}. The desired bounds on $g_1$ follow from \eqref{w.11} and Lemma \ref{nar8} (ii).

Finally we consider the function $g_2=[(V')^2\partial_v^2\Psi+V''\partial_v\Psi+\varrho V'\partial_v\Psi]\cdot\Theta\phi$. The proof is the same as before, using Lemma \ref{Multi0} (ii), \eqref{w.11}, and Lemma \ref{nar8} (ii). The additional weighted bounds we need in this case are
\begin{equation*}
\begin{split}
A_k(t,\xi)&\frac{\langle t\rangle}{|\xi/k|+\langle t\rangle}\lesssim_\delta \frac{A_R(t,\xi-\eta)}{\langle\xi-\eta\rangle}\cdot A_k(t,\eta)\frac{\langle t\rangle}{|\eta/k|+\langle t\rangle}\langle\eta-kt\rangle e^{-\delta'_0\min(\langle\xi-\eta\rangle,\langle k,\eta\rangle)^{1/2}},
\end{split}
\end{equation*}
which follow easily from \eqref{TLX7}. This completes the proof of the lemma.
\end{proof}

\subsection{Bounds on the terms $\mathcal{C}_a$, $a\in\{1,2,3,4\}$}
In this subsection we prove the bounds (\ref{wCD}) for the commutator terms $\mathcal{C}_a$. The idea is to exploit the cancellation in these commutators to obtain suitable smallness, using either the more favorable bounds \eqref{Imboot3'l} for small frequencies, or the gain in regularity from \eqref{A-A2} for high frequencies.

We start by decomposing $\mathcal{C}_2=\mathcal{C}'_2+\mathcal{C}''_2$ and $\mathcal{C}_3=\mathcal{C}'_3+\mathcal{C}''_3$ where
\begin{equation}\label{ront2}
\begin{split}
\mathcal{C}'_2&:=(V')^2\Psi_0(\partial_v-t\partial_z)^2A^\dagger\big(\Psi\phi\big)-(\Psi_0A^\dagger)\big[(V')^2(\partial_v-t\partial_z)^2\big(\Psi\phi\big)\big],\\
\mathcal{C}''_2&:=2(V')^2\partial_v\Psi_0\cdot(\partial_v-t\partial_z)A^\dagger\big(\Psi\phi\big)+(V')^2\partial_v^2\Psi_0\cdot A^\dagger\big(\Psi\phi\big),
\end{split}
\end{equation}
and
\begin{equation}\label{ront3}
\begin{split}
\mathcal{C}'_3&:=(V''+\varrho V')\Psi_0(\partial_v-t\partial_z)A^\dagger\big(\Psi\phi\big)-(\Psi_0A^\dagger)\big[(V''+\varrho V')(\partial_v-t\partial_z)\big(\Psi\phi\big)\big],\\
\mathcal{C}''_3&:=(V''+\varrho V')\partial_v\Psi_0\cdot A^\dagger\big(\Psi\phi\big).
\end{split}
\end{equation}

We bound first the commutators $\mathcal{C}_1$, $\mathcal{C}'_2$, and $\mathcal{C}'_3$.

\begin{lemma}\label{wCDc1}
For $t\in[1,T]$ we have
\begin{equation}\label{wCDc1.0}
\|\mathcal{C}_1(t)\|^2_{L^2(\mathbb{T}\times\mathbb{R})}+\|\mathcal{C}'_2(t)\|^2_{L^2(\mathbb{T}\times\mathbb{R})}+\|\mathcal{C}'_3(t)\|^2_{L^2(\mathbb{T}\times\mathbb{R})}\lesssim C(\delta)\epsilon_1^3+\delta\mathcal{E}'_\phi.
\end{equation}
\end{lemma}

\begin{proof} {\bf{Step 1.}} Fix a Gevrey cutoff function $\Psi'_0$ supported in $\big[\ubv/6,6\obv\big]$, equal to $1$ in $\big[\ubv/5,5\obv\big]$, and satisfying $\big\|e^{\langle\xi\rangle^{3/4}}\widetilde{\Psi'_0}(\xi)\big\|_{L^\infty}\lesssim 1$. Let $h_1:=\Psi'_0\varrho^2$. In view of (\ref{defC}) we have
\begin{equation*}
\begin{split}
\widetilde{\mathcal{C}_1}(t,k,\xi)=C\int_{\mathbb{R}^2}\Big\{&\widetilde{h_1}(t,\xi-\eta)k^2\widetilde{\Psi_0}(\eta-\rho)A^\dagger(t,k,\rho)\widetilde{(\Psi\phi)}(t,k,\rho)\\
&-\widetilde{\Psi_0}(\xi-\eta)A^\dagger(t,k,\eta)\widetilde{h_1}(t,\eta-\rho)k^2\widetilde{(\Psi\phi)}(t,k,\rho)\Big\}\, d\eta d\rho.
\end{split}
\end{equation*}
Therefore, after changes of variables,
\begin{equation*}
|\widetilde{\mathcal{C}_1}(t,k,\xi)|\lesssim\int_{\mathbb{R}^2}|\widetilde{h_1}(t,\xi-\eta)||\widetilde{\Psi_0}(\eta-\rho)||\widetilde{G}(t,k,\rho)| K_1(t,k;\xi,\eta,\rho)\, d\eta d\rho,
\end{equation*}
where $G:=[\partial_z^2+(\partial_v-t\partial_z)^2](\Psi\phi)$ and
\begin{equation}\label{kern1}
K_1(t,k;\xi,\eta,\rho):=\frac{\mathbf{1}_{\Z^\ast}(k)k^2}{k^2+(\rho-tk)^2}\Big|\frac{A_k(t,\rho)\langle t\rangle}{\langle t\rangle+|\rho/k|}-\frac{A_k(t,\rho+\xi-\eta)\langle t\rangle}{\langle t\rangle+|(\rho+\xi-\eta)/k|}\Big|.
\end{equation}
One can estimate $|\widetilde{\mathcal{C}'_2}(t,k,\xi)|$ and $|\widetilde{\mathcal{C}'_3}(t,k,\xi)|$ is a similar way, thus
\begin{equation*}
|\widetilde{\mathcal{C}'_a}(t,k,\xi)|\lesssim\int_{\mathbb{R}^2}|\widetilde{h_a}(t,\xi-\eta)||\widetilde{\Psi_0}(\eta-\rho)||\widetilde{G}(t,k,\rho)| K_a(t,k;\xi,\eta,\rho)\, d\eta d\rho,
\end{equation*}
where
\begin{equation}\label{kern2}
K_2(t,k;\xi,\eta,\rho):=\frac{\mathbf{1}_{\Z^\ast}(k)(\rho-tk)^2}{k^2+(\rho-tk)^2}\Big|\frac{A_k(t,\rho)\langle t\rangle}{\langle t\rangle+|\rho/k|}-\frac{A_k(t,\rho+\xi-\eta)\langle t\rangle}{\langle t\rangle+|(\rho+\xi-\eta)/k|}\Big|,
\end{equation}
\begin{equation}\label{kern3}
K_3(t,k;\xi,\eta,\rho):=\frac{\mathbf{1}_{\Z^\ast}(k)\langle\xi-\eta\rangle(\rho-tk)}{k^2+(\rho-tk)^2}\Big|\frac{A_k(t,\rho)\langle t\rangle}{\langle t\rangle+|\rho/k|}-\frac{A_k(t,\rho+\xi-\eta)\langle t\rangle}{\langle t\rangle+|(\rho+\xi-\eta)/k|}\Big|,
\end{equation}
and
\begin{equation}\label{kern5}
h_2:=\Psi'_0(V')^2,\qquad h_3:=\langle\partial_v\rangle^{-1}[\Psi'_0(V''+\varrho V')].
\end{equation}

Notice that $K_1+K_2+K_3\lesssim K$, where
\begin{equation}\label{ront5}
K(t,k;\xi,\eta,\rho):=\mathbf{1}_{\Z^\ast}(k)\Big[1+\frac{\langle\xi-\eta\rangle}{|k|\langle\rho/k-t\rangle}\Big]\Big|\frac{A_k(t,\rho)\langle t\rangle}{\langle t\rangle+|\rho/k|}-\frac{A_k(t,\rho+\xi-\eta)\langle t\rangle}{\langle t\rangle+|(\rho+\xi-\eta)/k|}\Big|.
\end{equation}
Notice also that $\|h_a\|_R\lesssim 1$ for $a\in\{1,2,3\}$, due to \eqref{nar4}. For \eqref{wCDc1.0} it suffices to prove that
\begin{equation}\label{ront6}
\Big\|\int_{\mathbb{R}^2}|\widetilde{h_a}(t,\xi-\eta)||\widetilde{\Psi_0}(\eta-\rho)||\widetilde{G}(t,k,\rho)| K(t,k;\xi,\eta,\rho)\, d\eta d\rho\Big\|^2_{L^2_{k,\xi}}\lesssim C(\delta)\epsilon_1^3+\delta\mathcal{E}'_\phi,
\end{equation}
for any $t\in[1,T]$ and $a\in\{1,2,3\}$.

{\bf Step 2.} Recall the definition of the sets $R_i^\ast$, $i\in\{0,1,2,3\}$, in \eqref{nar18.1}--\eqref{nar18.4} and \eqref{nar19.1}. Assume first that $((k,\rho),(k,\rho+\xi-\eta))\in R_1^\ast$. Then, using \eqref{A-A2},
\begin{equation*}
K(t,k;\xi,\eta,\rho)\lesssim \mathbf{1}_{\Z^\ast}(k)\frac{A_k(t,\rho)\langle t\rangle}{\langle t\rangle+|\rho/k|}A_R(t,\xi-\eta)e^{-(\delta_0/50)\langle\xi-\eta\rangle^{1/2}}\big[\sqrt\delta+C_0(\delta)\langle k,\rho\rangle^{-1/8}\big].
\end{equation*}
Therefore
\begin{equation}\label{ront7}
\begin{split}
\Big\|\int_{\mathbb{R}^2}&\mathbf{1}_{R_1^\ast}((k,\rho),(k,\rho+\xi-\eta))|\widetilde{h_a}(t,\xi-\eta)||\widetilde{\Psi_0}(\eta-\rho)||\widetilde{G}(t,k,\rho)| K(t,k;\xi,\eta,\rho)\, d\eta d\rho\Big\|_{L^2_{k,\xi}}\\
&\lesssim \|\widetilde{h_a}(t,\nu)A_R(t,\nu)\|_{L^2_\nu}\Big\|\mathbf{1}_{\Z^\ast}(k)\frac{A_k(t,\rho)\langle t\rangle}{\langle t\rangle+|\rho/k|}\widetilde{G}(t,k,\rho)\big[\sqrt\delta+C_0(\delta)\langle k,\rho\rangle^{-1/8}\big]\Big\|_{L^2_{k,\rho}}\\
&\lesssim C(\delta)\epsilon_1^{3/2}+\sqrt{\delta}(\mathcal{E}'_\phi)^{1/2},
\end{split}
\end{equation}
where the last inequality  follows from the bounds $\|h_a\|_R\lesssim 1$ and
\begin{equation}\label{ront8}
\Big\|\mathbf{1}_{\Z^\ast}(k)\frac{A_k(t,\rho)\langle t\rangle}{\langle t\rangle+|\rho/k|}\widetilde{G}(t,k,\rho)\big[\sqrt\delta+C_0(\delta)\langle k,\rho\rangle^{-1/8}\big]\Big\|_{L^2_{k,\rho}}\lesssim \sqrt{\delta}(\mathcal{E}'_\phi)^{1/2}+C(\delta)\eps_1^{3/2}.
\end{equation}
The bounds \eqref{ront8} follow from Lemma \ref{ImpThl} and the definition \eqref{w0.5}.

{\bf{Step 3.}} In the remaining case when $\langle\xi-\eta\rangle\geq\langle k,\rho\rangle/10$ we use the bounds \eqref{TLX7} to estimate
\begin{equation*}
\begin{split}
K(t,k;\xi,\eta,\rho)&\lesssim \mathbf{1}_{\Z^\ast}(k)\langle\xi-\eta\rangle\frac{A_k(t,\rho)\langle t\rangle}{\langle t\rangle+|\rho/k|}\\
&+C(\delta)\mathbf{1}_{\Z^\ast}(k)\Big[1+\frac{\langle\xi-\eta\rangle}{|k|\langle\rho/k-t\rangle}\Big]\frac{A_k(t,\rho)A_R(t,\xi-\eta)e^{-(\delta_0/100)\langle k,\rho\rangle^{1/2}}\langle t\rangle}{\langle t\rangle+|(\rho+\xi-\eta)/k|}.
\end{split}
\end{equation*}
Moreover, with $\delta'_0=\delta_0/200$ as before, it is easy to see that
\begin{equation}\label{runt13.5}
\Big[1+\frac{\langle\xi-\eta\rangle}{|k|\langle\rho/k-t\rangle}\Big]\frac{\langle t\rangle}{\langle t\rangle+|(\rho+\xi-\eta)/k|}\lesssim e^{\delta'_0\langle k,\rho\rangle^{1/2}}.
\end{equation}
Therefore, if $k\neq 0$ and $\langle\xi-\eta\rangle\geq\langle k,\rho\rangle/10$ then
\begin{equation}\label{ront9}
K(t,k;\xi,\eta,\rho)\lesssim_\delta \mathbf{1}_{\Z^\ast}(k)\frac{\langle\xi-\eta\rangle^2}{\langle k,\rho\rangle}\frac{A_k(t,\rho)\langle t\rangle}{\langle t\rangle+|\rho/k|}+\mathbf{1}_{\Z^\ast}(k)A_k(t,\rho)A_R(t,\xi-\eta)e^{-\delta'_0\langle k,\rho\rangle^{1/2}}.
\end{equation}
We notice that if $((k,\rho),(k,\rho+\xi-\eta))\notin R_1^\ast$ then $\langle\xi-\eta\rangle\geq\langle k,\rho\rangle/10$, so the bounds \eqref{ront9} hold. As in \eqref{ront7} one can then estimate
\begin{equation*}
\begin{split}
\Big\|\int_{\mathbb{R}^2}&\mathbf{1}_{{}^cR_1^\ast}((k,\rho),(k,\rho+\xi-\eta))|\widetilde{h_a}(t,\xi-\eta)||\widetilde{\Psi_0}(\eta-\rho)||\widetilde{G}(t,k,\rho)| K(t,k;\xi,\eta,\rho)\, d\eta d\rho\Big\|_{L^2_{k,\xi}}\\
&\lesssim C(\delta)\epsilon_1^{3/2}+\sqrt{\delta}(\mathcal{E}'_\phi)^{1/2},
\end{split}
\end{equation*} 
and the desired conclusion \eqref{ront6} follows.
\end{proof}

We bound now the remaining commutators $\mathcal{C}_4$, $\mathcal{C}''_2$, and $\mathcal{C}''_3$.

\begin{lemma}\label{wCDc2}
For $t\in[1,T]$ we have
\begin{equation}\label{ront10}
\|\mathcal{C}_4(t)\|^2_{L^2(\mathbb{T}\times\mathbb{R})}+\|\mathcal{C}''_2(t)\|^2_{L^2(\mathbb{T}\times\mathbb{R})}+\|\mathcal{C}''_3(t)\|^2_{L^2(\mathbb{T}\times\mathbb{R})}\lesssim C(\delta)\epsilon_1^3+\delta\mathcal{E}'_\phi.
\end{equation}
\end{lemma}

\begin{proof} With $\Psi'_0$ as in the proof of Lemma \ref{wCDc1}, we notice first that $\|\Psi'_0 V''\|_{L^\infty}+\|\Psi'_0 V'\|_{L^\infty}+\|\Psi'_0 \varrho\|_{L^\infty}\lesssim 1$, as a consequence of \eqref{rew3}. Therefore, if we let
\begin{equation}\label{ront11}
\mathcal{C}''':=\partial^a_v\Psi_0\cdot A^\dagger(\partial_v-t\partial_z)^b(\Psi\phi)-A^\dagger[\partial^a_v\Psi_0\cdot (\partial_v-t\partial_z)^b(\Psi\phi)],
\end{equation}
where $a\in\{1,2\},\,b\in\{0,1\}$, then it suffices to prove that
\begin{equation}\label{ront12}
\|\mathcal{C}'''(t)\|_{L^2(\mathbb{T}\times\mathbb{R})}+\|\mathcal{C}_4(t)\|_{L^2(\mathbb{T}\times\mathbb{R})}\lesssim C(\delta)\epsilon_1^{3/2}+\sqrt{\delta}(\mathcal{E}'_\phi)^{1/2}.
\end{equation}

Let $P(\rho):=e^{-\langle\rho\rangle^{3/4}/2}$. As in the proof of Lemma \ref{wCDc1} we have
\begin{equation}\label{ront13}
|\widetilde{\mathcal{C}'''}(t,k,\xi)|+|\widetilde{\mathcal{C}_4}(t,k,\xi)|\lesssim \int_{\mathbb{R}}P(\xi-\eta)|\widetilde{G}(t,k,\eta)|L(t,k;\xi,\eta)\,d\eta,
\end{equation}
where $G=[\partial_z^2+(\partial_v-t\partial_z)^2](\Psi\phi)$ as before and
\begin{equation}\label{ront14}
L(t,k;\xi,\eta):=\mathbf{1}_{\Z^\ast}(k)\langle\xi-\eta\rangle^2\Big|\frac{A_k(t,\xi)\langle t\rangle}{\langle t\rangle+|\xi/k|}-\frac{A_k(t,\eta)\langle t\rangle}{\langle t\rangle+|\eta/k|}\Big|.
\end{equation}

If $((k,\xi),(k,\eta))\in R_1^\ast$ then we use \eqref{A-A2} to estimate
\begin{equation*}
L(t,k;\xi,\eta)\lesssim \mathbf{1}_{\Z^\ast}(k)\frac{A_k(t,\eta)\langle t\rangle}{\langle t\rangle+|\eta/k|}A_R(t,\xi-\eta)e^{-(\delta_0/50)\langle\xi-\eta\rangle^{1/2}}\big[\sqrt\delta+C_0(\delta)\langle k,\eta\rangle^{-1/8}\big].
\end{equation*}
Therefore, using \eqref{ront8},
\begin{equation}\label{ront14.5}
\begin{split}
\Big\|\int_{\mathbb{R}^2}\mathbf{1}_{R_1^\ast}((k,\xi),(k,\eta))&P(\xi-\eta)|\widetilde{G}(t,k,\eta)| L(t,k;\xi,\eta)\, d\eta\Big\|_{L^2_{k,\xi}}\\
&\lesssim \Big\|\mathbf{1}_{\Z^\ast}(k)\frac{A_k(t,\eta)\langle t\rangle}{\langle t\rangle+|\eta/k|}\widetilde{G}(t,k,\eta)\big[\sqrt\delta+C_0(\delta)\langle k,\rho\rangle^{-1/8}\big]\Big\|_{L^2_{k,\eta}}\\
&\lesssim C(\delta)\epsilon_1^{3/2}+\sqrt{\delta}(\mathcal{E}'_\phi)^{1/2}.
\end{split}
\end{equation}

On the other hand, if $\langle\xi-\eta\rangle\geq\langle k,\eta\rangle/10$ then we can use the definitions \eqref{dor4} to estimate
\begin{equation*}
L(t,k;\xi,\eta)\lesssim \mathbf{1}_{\Z^\ast}(k)e^{10\delta_0\langle\xi-\eta\rangle^{1/2}}\langle k,\eta\rangle^{-1}.
\end{equation*}
Using \eqref{Imboot3'l} we then estimate
\begin{equation}\label{ront15}
\begin{split}
\Big\|\int_{\mathbb{R}^2}\mathbf{1}_{{}^cR_1^\ast}((k,\xi),(k,\eta))&P(\xi-\eta)|\widetilde{G}(t,k,\eta)| L(t,k;\xi,\eta)\, d\eta\Big\|_{L^2_{k,\xi}}\lesssim \|\widetilde{G}(t,k,\eta)\|_{L^2_{k,\eta}}\lesssim_\delta\epsilon_1^{3/2}.
\end{split}
\end{equation}
The desired bounds \eqref{ront12} follow from \eqref{ront14.5}--\eqref{ront15}.
\end{proof}

\subsection{Bounds on the terms $\mathcal{D}_a$, $a\in\{1,2,3,4\}$} Finally, we prove the bounds (\ref{wCD}) for the commutator terms $\mathcal{D}_a$. As before, we decompose $\mathcal{D}_2=\mathcal{D}'_2+\mathcal{D}''_2$ and $\mathcal{D}_3=\mathcal{D}'_3+\mathcal{D}''_3$ where
\begin{equation}\label{runt2}
\begin{split}
\mathcal{D}'_2&:=(V')^2\Psi_0(\partial_v-t\partial_z)^2B^\dagger\big(\Psi\phi\big)-(\Psi_0B^\dagger)\big[(V')^2(\partial_v-t\partial_z)^2\big(\Psi\phi\big)\big],\\
\mathcal{D}''_2&:=2(V')^2\partial_v\Psi_0\cdot(\partial_v-t\partial_z)B^\dagger\big(\Psi\phi\big)+(V')^2\partial_v^2\Psi_0\cdot B^\dagger\big(\Psi\phi\big),
\end{split}
\end{equation}
and
\begin{equation}\label{runt3}
\begin{split}
\mathcal{D}'_3&:=(V''+\varrho V')\Psi_0(\partial_v-t\partial_z)B^\dagger\big(\Psi\phi\big)-(\Psi_0B^\dagger)\big[(V''+\varrho V')(\partial_v-t\partial_z)\big(\Psi\phi\big)\big],\\
\mathcal{D}''_3&:=(V''+\varrho V')\partial_v\Psi_0\cdot B^\dagger\big(\Psi\phi\big).
\end{split}
\end{equation}

As before, we bound first the commutators $\mathcal{D}_1$, $\mathcal{D}'_2$, and $\mathcal{D}'_3$.

\begin{lemma}\label{runt5}
For $t\in[1,T]$ we have
\begin{equation}\label{runt6}
\int_1^t\big\{\|\mathcal{D}_1(s)\|^2_{L^2(\mathbb{T}\times\mathbb{R})}+\|\mathcal{D}'_2(s)\|^2_{L^2(\mathbb{T}\times\mathbb{R})}+\|\mathcal{D}'_3(s)\|^2_{L^2(\mathbb{T}\times\mathbb{R})}\big\}\,ds\lesssim C(\delta)\epsilon_1^3+\delta(\mathcal{E}'_\phi+\mathcal{B}'_\phi).
\end{equation}
\end{lemma}

\begin{proof} As in the proof of Lemma \ref{wCDc1}, we can estimate the Fourier transforms $|\widetilde{\mathcal{D}_1}(s,k,\xi)|$, $|\widetilde{\mathcal{D}'_2}(s,k,\xi)|$, $|\widetilde{\mathcal{D}'_3}(s,k,\xi)|$. With $\Psi'_0$ as in the proof of Lemma \ref{wCDc1}, let, as before,
\begin{equation}\label{runt8}
G=[\partial_z^2+(\partial_v-t\partial_z)^2](\Psi\phi),\quad h_1=\Psi'_0\varrho^2,\quad h_2=\Psi'_0(V')^2,\quad h_3=\langle\partial_v\rangle^{-1}[\Psi'_0(V''+\varrho V')].
\end{equation}
With $B^\dagger$ defined as in \eqref{Bast}, \eqref{Aast}, and \eqref{mu1}, let
\begin{equation}\label{runt12}
K'(s,k;\xi,\eta,\rho):=\mathbf{1}_{\Z^\ast}(k)\Big[1+\frac{\langle\xi-\eta\rangle}{|k|\langle\rho/k-s\rangle}\Big]\big|B^\dagger(s,k,\rho)-B^\dagger(s,k,\rho+\xi-\eta)\big|.
\end{equation}
For \eqref{runt6} it suffices to prove that, for any $t\in[1,T]$ and $a\in\{1,2,3\}$,
\begin{equation}\label{runt40}
\begin{split}
\int_1^t\Big\|\int_{\mathbb{R}^2}|\widetilde{h_a}(s,\xi-\eta)||\widetilde{\Psi_0}(\eta-\rho)||\widetilde{G}(s,k,\rho)| K'(s,k;\xi,\eta,\rho)\, d\eta d\rho\Big\|^2_{L^2_{k,\xi}}\,ds\\
\lesssim C(\delta)\epsilon_1^3+\delta(\mathcal{E}'_\phi+\mathcal{B}'_\phi).
\end{split}
\end{equation}

As in the proof of Lemma \ref{wCDc1}, assume first that $((k,\rho),(k,\rho+\xi-\eta))\in R_1^\ast$. Then
\begin{equation*}
K'(s,k;\xi,\eta,\rho)\lesssim \mathbf{1}_{\Z^\ast}(k)\frac{\sqrt{\mu_k(s,\rho)}A_k(s,\rho)\langle s\rangle}{\langle s\rangle+|\rho/k|}A_R(s,\xi-\eta)e^{-(\delta_0/50)\langle\xi-\eta\rangle^{1/2}}\big[\sqrt\delta+C'(\delta)\langle k,\rho\rangle^{-1/8}\big],
\end{equation*}
using the definition and Lemmas \ref{A-A} and \ref{Lmu3}. Moreover
\begin{equation}\label{runt41}
\begin{split}
\Big\|\mathbf{1}_{\Z^\ast}(k)\frac{\sqrt{\mu_k(s,\rho)}A_k(s,\rho)\langle s\rangle}{\langle s\rangle+|\rho/k|}\widetilde{G}(s,k,\rho)\big[\sqrt\delta+C'(\delta)\langle k,\rho\rangle^{-1/8}\big]\Big\|_{L^2_sL^2_{k,\rho}}\\
\lesssim \sqrt{\delta}(\mathcal{B}'_\phi)^{1/2}+C''(\delta)\eps_1^{3/2},
\end{split}
\end{equation}
due to the definition \eqref{w0.5}, the bounds \eqref{Imboot3'l}, and the observation that $A_k(s,\rho)\lesssim_\delta 1$, $\mu_k(s,\rho)\lesssim_\delta \langle s\rangle^{-1-\sigma_0}$ if $\langle k,\rho\rangle\lesssim_\delta 1$. Therefore
\begin{equation}\label{runt13}
\begin{split}
&\Big\|\int_{\mathbb{R}^2}\mathbf{1}_{R_1^\ast}((k,\rho),(k,\rho+\xi-\eta))|\widetilde{h_a}(s,\xi-\eta)||\widetilde{\Psi_0}(\eta-\rho)||\widetilde{G}(s,k,\rho)| K'(s,k;\xi,\eta,\rho)\, d\eta d\rho\Big\|_{L^2_sL^2_{k,\xi}}\\
&\lesssim \|\widetilde{h_a}(s,\nu)A_R(s,\nu)\|_{L^\infty_sL^2_\nu}\Big\|\mathbf{1}_{\Z^\ast}(k)\frac{\sqrt{\mu_k(s,\rho)}A_k(s,\rho)\langle s\rangle}{\langle s\rangle+|\rho/k|}\widetilde{G}(s,k,\rho)\big[\sqrt\delta+C'(\delta)\langle k,\rho\rangle^{-1/8}\big]\Big\|_{L^2_sL^2_{k,\rho}}\\
&\lesssim C(\delta)\epsilon_1^{3/2}+\sqrt{\delta}(\mathcal{B}'_\phi)^{1/2},
\end{split}
\end{equation}
where the last inequality uses also the bounds $\|h_a\|_R\lesssim 1$.

On the other hand, if $\langle\xi-\eta\rangle\geq\langle k,\rho\rangle/10$ we use the bounds \eqref{TLX7} and \eqref{Lmu3.1} to estimate
\begin{equation*}
\begin{split}
K'(s,&k;\xi,\eta,\rho)\lesssim \mathbf{1}_{\Z^\ast}(k)\langle\xi-\eta\rangle\frac{\sqrt{\mu_k(s,\rho)}A_k(s,\rho)\langle s\rangle}{\langle s\rangle+|\rho/k|}\\
&+C(\delta)\mathbf{1}_{\Z^\ast}(k)\Big[1+\frac{\langle\xi-\eta\rangle}{|k|\langle\rho/k-s\rangle}\Big]\frac{\sqrt{\mu_R(s,\xi-\eta)}A_k(s,\rho)A_R(s,\xi-\eta)e^{-(\delta_0/100)\langle k,\rho\rangle^{1/2}}\langle s\rangle}{\langle s\rangle+|(\rho+\xi-\eta)/k|}.
\end{split}
\end{equation*}
Using again \eqref{runt13.5} it follows that if $\delta'_0=\delta_0/200$, $k\neq 0$, and $\langle\xi-\eta\rangle\geq\langle k,\rho\rangle/10$ then
\begin{equation*}
\begin{split}
K'(s,k;\xi,\eta,\rho)&\lesssim_\delta \mathbf{1}_{\Z^\ast}(k)\frac{\langle\xi-\eta\rangle^2}{\langle k,\rho\rangle}\frac{\sqrt{\mu_k(s,\rho)}A_k(s,\rho)\langle s\rangle}{\langle s\rangle+|\rho/k|}\\
&+\mathbf{1}_{\Z^\ast}(k)\sqrt{\mu_R(s,\xi-\eta)}A_R(s,\xi-\eta)A_k(s,\rho)e^{-\delta'_0\langle k,\rho\rangle^{1/2}}.
\end{split}
\end{equation*}
Therefore, using \eqref{ront8} and \eqref{runt41},
\begin{equation*}
\begin{split}
\Big\|\int_{\mathbb{R}^2}&\mathbf{1}_{{}^cR_1^\ast}((k,\rho),(k,\rho+\xi-\eta))|\widetilde{h_a}(s,\xi-\eta)||\widetilde{\Psi_0}(\eta-\rho)||\widetilde{G}(s,k,\rho)| K'(s,k;\xi,\eta,\rho)\, d\eta d\rho\Big\|_{L^2_sL^2_{k,\xi}}\\
&\lesssim_\delta \|\widetilde{h_a}(s,\nu)A_R(s,\nu)\|_{L^\infty_sL^2_\nu}\Big\|\mathbf{1}_{\Z^\ast}(k)\frac{\sqrt{\mu_k(s,\rho)}A_k(s,\rho)\langle s\rangle}{\langle s\rangle+|\rho/k|}\widetilde{G}(s,k,\rho)\langle k,\rho\rangle^{-1}\Big\|_{L^2_sL^2_{k,\rho}}\\\\
&+\|\widetilde{h_a}(s,\nu)\sqrt{\mu_R(s,\nu)}A_R(s,\nu)\|_{L^2_sL^2_\nu}\big\|\mathbf{1}_{\Z^\ast}(k)A_k(s,\rho)\widetilde{G}(s,k,\rho)e^{-(\delta'_0/2)\langle k,\rho\rangle^{1/2}}\big\|_{L^\infty_sL^2_{k,\rho}}\\
&\lesssim C(\delta)\epsilon_1^{3/2}+\sqrt{\delta}(\mathcal{B}'_\phi)^{1/2}+\sqrt{\delta}(\mathcal{E}'_\phi)^{1/2}.
\end{split}
\end{equation*} 
The desired conclusion \eqref{runt40} follows using also \eqref{runt13}.
\end{proof}

We bound now the remaining commutators $\mathcal{D}_4$, $\mathcal{D}''_2$, and $\mathcal{D}''_3$.

\begin{lemma}\label{wCDc5}
For $t\in[1,T]$ we have
\begin{equation}\label{runt50}
\int_1^t\big\{\|\mathcal{D}_4(s)\|^2_{L^2(\mathbb{T}\times\mathbb{R})}+\|\mathcal{D}''_2(s)\|^2_{L^2(\mathbb{T}\times\mathbb{R})}+\|\mathcal{D}''_3(s)\|^2_{L^2(\mathbb{T}\times\mathbb{R})}\big\}\,ds\lesssim C(\delta)\epsilon_1^3+\delta(\mathcal{E}'_\phi+\mathcal{B}'_\phi).
\end{equation}
\end{lemma}

\begin{proof} As in the proof of Lemma \ref{wCDc2}, it suffices to show that, for any $t\in[1,T]$,
\begin{equation}\label{runt60}
\begin{split}
\int_1^t\Big\|\int_{\mathbb{R}}P(\xi-\eta)|\widetilde{G}(s,k,\eta)|L'(s,k;\xi,\eta)\,d\eta\Big\|^2_{L^2_{k,\xi}}\,ds\lesssim C(\delta)\epsilon_1^3+\delta(\mathcal{E}'_\phi+\mathcal{B}'_\phi),
\end{split}
\end{equation}
where $P(\rho):=e^{-\langle\rho\rangle^{3/4}/2}$, $G=[\partial_z^2+(\partial_v-t\partial_z)^2](\Psi\phi)$ as before, and
\begin{equation}\label{runt61}
L'(s,k;\xi,\eta):=\mathbf{1}_{\Z^\ast}(k)\langle\xi-\eta\rangle^2\Big|\frac{\sqrt{\mu_k(s,\xi)}A_k(s,\xi)\langle s\rangle}{\langle s\rangle+|\xi/k|}-\frac{\sqrt{\mu_k(s,\eta)}A_k(s,\eta)\langle s\rangle}{\langle s\rangle+|\eta/k|}\Big|.
\end{equation}

If $((k,\xi),(k,\eta))\in R_1^\ast$ then we use \eqref{A-A2} and Lemma \ref{Lmu3} to estimate
\begin{equation*}
L'(s,k;\xi,\eta)\lesssim \mathbf{1}_{\Z^\ast}(k)\frac{\sqrt{\mu_k(s,\eta)}A_k(s,\eta)\langle s\rangle}{\langle s\rangle+|\eta/k|}A_R(s,\xi-\eta)e^{-(\delta_0/50)\langle\xi-\eta\rangle^{1/2}}\big[\sqrt\delta+C_0(\delta)\langle k,\eta\rangle^{-1/8}\big].
\end{equation*}
On the other hand, if $\langle\xi-\eta\rangle\geq\langle k,\eta\rangle/10$ then we can use the definitions \eqref{dor4} to estimate
\begin{equation*}
L'(s,k;\xi,\eta)\lesssim \mathbf{1}_{\Z^\ast}(k)e^{10\delta_0\langle\xi-\eta\rangle^{1/2}}\langle t\rangle^{-(1+\sigma_0)/2}\langle k,\eta\rangle^{-1}.
\end{equation*}
The desired bounds \eqref{runt60} follow in the same way as in the proof of Lemma \ref{runt5}.
\end{proof}

\section{The time-dependent imbalanced weights: definitions and properties}\label{weights}

\subsection{Definitions}\label{weightsdefin}

In this subsection we recall the precise definitions of the weights $w_\ast, b_\ast, A_\ast$, $\ast\in\{NR,R,k\}$, $k\in\mathbb{Z}$, from \cite{IOJI}. We define first the functions $w_{NR},w_R:[0,\infty)\times\mathbb{R}\to [0,1]$ which model the non-resonant and resonant growth respectively. Take small $\delta>0$ with $\delta\ll \delta_0$, which is still much larger than $ \epsilon$. For $|\eta|\leq\delta^{-10}$ we define simply
\begin{equation}\label{reb1}
w_{NR}(t,\eta):=1,\qquad w_R(t,\eta):=1.
\end{equation}
For $\eta>\delta^{-10}$ we define $k_0(\eta):=\lfloor\sqrt{\delta^3\eta}\rfloor$. For $l\in\{1,\ldots,k_0(\eta)\}$ we define
\begin{equation}\label{reb2}
t_{l,\eta}:=\frac{1}{2}\big(\frac{\eta}{l+1}+\frac{\eta}{l}\big),\qquad t_{0,\eta}:=2\eta,\qquad I_{l,\eta}:=[t_{l,\eta},\,t_{l-1,\eta}].
\end{equation}
Notice that $|I_{l,\eta}|\sim \frac{\eta}{l^2}$ and
\begin{equation*}
\delta^{-3/2}\sqrt{\eta}/2\leq t_{k_0(\eta),\eta}\leq\ldots\leq t_{l,\eta}\leq\eta/l\leq t_{l-1,\eta}\leq\ldots\leq t_{0,\eta}=2\eta.
\end{equation*}

We define
\begin{equation}\label{reb3}
w_{NR}(t,\eta):=1,\,w_{R}(t,\eta):=1\qquad\text{ if }\,\,t\geq t_{0,\eta}=2\eta.
\end{equation}
Then we define, for $k\in\{1,\ldots,k_0(\eta)\}$,
\begin{equation}\label{reb5}
\begin{split}
w_{NR}(t,\eta)&:=\Big(\frac{1+\delta^2|t-\eta/k|}{1+\delta^2|t_{k-1,\eta}-\eta/k|}\Big)^{\delta_0}w_{NR}(t_{k-1,\eta},\eta)\qquad\text{ if }t\in[\eta/k,t_{k-1,\eta}],\\
w_{NR}(t,\eta)&:=\Big(\frac{1}{1+\delta^2|t-\eta/k|}\Big)^{1+\delta_0}w_{NR}(\eta/k,\eta)\qquad\text{ if }t\in[t_{k,\eta},\eta/k].
\end{split}
\end{equation}
We define also the weight $w_R$ by the formula
\begin{equation}\label{reb5.5}
w_R(t,\eta):=
\begin{cases}
w_{NR}(t,\eta)\frac{1+\delta^2|t-\eta/k|}{1+\delta^2\eta/(8k^2)}\qquad&\text{ if }|t-\eta/k|\leq\eta/(8k^2)\\
w_{NR}(t,\eta)\qquad&\text{ if }t\in I_{k,\eta},\,|t-\eta/k|\geq\eta/(8k^2),
\end{cases}
\end{equation}
for any $k\in\{1,\ldots,k_0(\eta)\}$. Notice that
\begin{equation}\label{reb4}
\frac{w_{NR}(t_{k,\eta},\eta)}{w_{NR}(t_{k-1,\eta},\eta)}\approx \Big(\frac{k^2}{\delta^2\eta}\Big)^{1+2\delta_0},\qquad w_{R}(t_{k,\eta},\eta)=w_{NR}(t_{k,\eta},\eta).
\end{equation}
Moreover, notice that for $t\in I_{k,\eta}$,
\begin{equation}\label{reb7}
w_{R}(t,\eta)\approx w_{NR}(t,\eta)\left[\frac{k^2}{\delta^2\eta}\left(1+\delta^2|t-\eta/k|\right)\right],
\end{equation}
and
\begin{equation}\label{reb8}
\frac{\partial_tw_{NR}(t,\eta)}{w_{NR}(t,\eta)}\approx\frac{\partial_tw_R(t,\eta)}{w_R(t,\eta)}\approx \frac{\delta^2}{1+\delta^2\left|t-\eta/k\right|}.
\end{equation}

We observe that 
\begin{equation}\label{reb8.5}
e^{(J_2-J_1)\ln(A/J_2^2)}\leq\prod_{j=J_1+1}^{J_2}\frac{A}{j^2}\leq e^{(J_2-J_1)\ln(A/J_2^2)+4(J_2-J_1)}
\end{equation}
provided that $1\leq J_1+1\leq J_2$. In particular, for $\eta>\delta^{-10}$,
\begin{equation}\label{reb8.6}
\begin{split}
&w_{NR}(t_{k_0(\eta),\eta},\eta)=w_R(t_{k_0(\eta),\eta},\eta)\in[X_\delta(\eta)^4,X_\delta(\eta)^{1/4}],\\
&X_\delta(\eta):=e^{-\delta^{3/2}\ln(\delta^{-1})\sqrt\eta}.
\end{split}
\end{equation}

For small values of $t\leq t_{k_0(\eta),\eta}$ we define the weights $w_{NR}$ and $w_R$ by the formulas 
\begin{equation}\label{reb9}
w_{NR}(t,\eta)=w_R(t,\eta):=(e^{-\delta\sqrt\eta})^\beta w_{NR}(t_{k_0(\eta),\eta},\eta)^{1-\beta}
\end{equation}
if $t=(1-\beta)t_{k_0(\eta),\eta}$, $\beta\in[0,1]$. We notice that
\begin{equation}\label{reb9.5}
\frac{w_{NR}(t_1,\eta)}{w_{NR}(t_2,\eta)}\lesssim e^{4\delta^{5/2}|t_1-t_2|}\qquad\text{ for any }t_1\in[0,t_{k_0(\eta),\eta}],\,t_2\in[0,\infty).
\end{equation}

If $\eta<-\delta^{-10}$, then we define $w_R(t,\eta):=w_R(t,|\eta|)$, $w_{NR}(t,\eta):=w_{NR}(t,|\eta|)$.

We define now the weights $w_k(t,\eta)$ which crucially distinguish the way resonant and non-resonant modes grow around the critical times $\eta/k$. If $|\eta|\leq\delta^{-10}$, then we define $w_k(t,\eta)\equiv 1$. If $\eta>\delta^{-10}$ and $1\leq k\leq \sqrt{\delta^3\eta}$, then we define
\begin{equation}\label{eq:resonantweight}
w_k(t,\eta):=\left\{\begin{array}{lll}
w_{NR}(t,\eta)&{\rm \,if\,}&t\not\in I_{k,\eta},\\
w_R(t,\eta)&{\rm \,if\,}&t\in I_{k,\eta}.
\end{array}\right.
\end{equation}
If $\eta>\delta^{-10}$, $k\not\in\big[1,\,\sqrt{\delta^3\eta}\,\big]$, we define $w_k(t,\eta)\equiv w_{NR}(t,\eta)$. If $\eta<-\delta^{-10}$ then we define
\begin{equation}\label{reb10}
w_k(\eta):=w_{-k}(-\eta).
\end{equation}
In particular $w_k(\eta)=w_{NR}(\eta)$ if $k\eta\leq 0$. 

\subsubsection{The functions $b_\ast$ and $A_\ast$} The functions $w_{NR}$, $w_{R}$ and $w_k$ have the right size but lack optimal smoothness in the frequency parameter $\eta$, mainly due to the jump discontinuities of the function $k_0(\eta)$. The smoothness of the weight is important in the analysis of the transport terms, as it leads to smaller loss of derivatives after symmetrization in the energy functionals.

To correct this problem we mollify the weights $w_\ast$. We fix $\varphi: \R \to [0,1]$ an even smooth function supported in $[-8/5,8/5]$ and equal to $1$ in $[-5/4,5/4]$ and let $d_0:=\int_\mathbb{R}\varphi(x)\,dx$. For $k\in\mathbb{Z}$ and $\ast\in \{NR,R,k\}$ let
\begin{equation}\label{dor1}
\begin{split}
b_\ast(t,\xi)&:=\int_\R w_\ast(t,\rho)\varphi\Big(\frac{\xi-\rho}{L_{\delta'}(t,\xi)}\Big)\frac{1}{d_0L_{\delta'}(t,\xi)}\,d\rho,\\
L_{\delta'}(t,\xi)&:=1+\frac{\delta'\langle\xi\rangle}{\langle\xi\rangle^{1/2}+\delta' t},\qquad\delta'\in[0,1].
\end{split}
\end{equation}
In other words, the functions $b_\ast(t,\xi)$ are obtained by averaging $w_\ast(t,\rho)$ over intervals of length $\approx L_{\delta'}(t,\xi)$ around the point $\xi$. The length $L_{\delta'}(t,\xi)$ in \eqref{dor1} is chosen to optimize the smoothness in $\xi$ of the functions $b_\ast(t,.)$, while not changing significantly the size of the weights. The parameter $\delta'$ is fixed sufficiently small, depending only on $\delta$. 

We can now finally define our main weights $A_{NR}$, $A_R$, and $A_k$. We define first the decreasing function $\lambda:[0,\infty)\to[\delta_0,3\delta_0/2]$ by
\begin{equation}\label{dor2}
\lambda(0)=\frac{3}{2}\delta_0,\,\,\,\,\lambda'(t)=-\frac{\delta_0\sigma_0^2}{\langle t\rangle^{1+\sigma_0}},
\end{equation}
for small positive constant $\sigma_0$ (say $\sigma_0=0.01$). Then we define
\begin{equation}\label{dor3}
A_R(t,\xi):=\frac{e^{\lambda(t)\langle\xi\rangle^{1/2}}}{b_R(t,\xi)}e^{\sqrt{\delta}\langle\xi\rangle^{1/2}},\qquad A_{NR}(t,\xi):=\frac{e^{\lambda(t)\langle\xi\rangle^{1/2}}}{b_{NR}(t,\xi)}e^{\sqrt{\delta}\langle\xi\rangle^{1/2}},
\end{equation}
and, for any $k\in\mathbb{Z}$,
\begin{equation}\label{dor4}
A_k(t,\xi):=e^{\lambda(t)\langle k,\xi\rangle^{1/2}}\Big(\frac{e^{\sqrt{\delta}\langle\xi\rangle^{1/2}}}{b_k(t,\xi)}+e^{\sqrt{\delta}|k|^{1/2}}\Big).
\end{equation}

\subsection{Properties of the weights} In this subsection we collect several bounds on the weights $w_\ast$, $b_\ast$, and $A_\ast$, see \cite{IOJI} for proofs. We start with a lemma (Lemma 7.1 in \cite{IOJI}):

\begin{lemma}\label{comparisonweights}
For all $t\ge0$ and $\xi,\,\eta\in \R$, and $k\in\mathbb{Z}$ we have
\begin{equation}\label{eq:comparisonweights1}
\frac{w_{NR}(t,\xi)}{w_{NR}(t,\eta)}+\frac{w_{R}(t,\xi)}{w_{R}(t,\eta)}+\frac{w_{k}(t,\xi)}{w_{k}(t,\eta)}\lesssim_\delta e^{\sqrt{\delta} |\eta-\xi|^{1/2}}.
\end{equation}
Moreover, if $|\xi-\eta|\leq 10L_1(t,\eta)$ (see \eqref{dor1} for the definition of $L_1(t,\eta)$) then we have the stronger bounds
\begin{equation}\label{dor6}
\frac{w_{NR}(t,\xi)}{w_{NR}(t,\eta)}+\frac{w_{R}(t,\xi)}{w_{R}(t,\eta)}+\frac{w_{k}(t,\xi)}{w_{k}(t,\eta)}\lesssim_\delta 1.
\end{equation}
Finally, if $\min(|\xi|,|\eta|)\geq 2\delta^{-10}$, $|\xi-\eta|\leq \min(|\xi|,|\eta|)/3$, and $t\geq \max(t_{k_0(\xi)-4,\xi},t_{k_0(\eta)-4,\eta})$ then we also have the stronger bounds
\begin{equation}\label{ReSm2}
\max\Big\{\frac{w_{NR}(t,\xi)}{w_{NR}(t,\eta)},\,\frac{w_{R}(t,\xi)}{w_{R}(t,\eta)},\,\frac{w_{k}(t,\xi)}{w_{k}(t,\eta)}\Big\}\leq e^{\sqrt{\delta} |\eta-\xi|^{1/2}}.
\end{equation}
\end{lemma}

We also need estimates on the functions $b_\ast$ defined in \eqref{dor1}, see also Lemma 7.2 in \cite{IOJI}.

\begin{lemma}\label{bweights}
For $t\geq 0$, $\xi,\eta\in\R$, $k\in\Z$, and $\ast\in\{NR,R,k\}$ we have
\begin{equation}\label{dor20}
b_\ast(t,\xi)\approx_\delta w_\ast(t,\xi),
\end{equation}
\begin{equation}\label{dor21}
|\partial_\xi b_\ast(t,\xi)|\lesssim_\delta b_\ast(t,\xi)\frac{1}{L_{\delta'}(t,\xi)},
\end{equation}
and
\begin{equation}\label{dor22}
\frac{b_\ast(t,\xi)}{b_\ast(t,\eta)}\lesssim_\delta e^{\sqrt{\delta} |\eta-\xi|^{1/2}}.
\end{equation}
\end{lemma}

We recall now several bounds on the main weights $A_{NR}, A_R, A_k$, see Lemma 7.3 in \cite{IOJI}.

\begin{lemma}\label{A_kA_ell}
(i) Assume $t\in[0,\infty)$, $k\in\mathbb{Z}$, and $\ast\in\{NR,R,k\}$. Then, for any $\xi,\eta\in\mathbb{R}$,
\begin{equation}\label{vfc25.1}
\frac{A_\ast(t,\xi)}{A_\ast(t,\eta)}\lesssim_\delta e^{(\lambda(t)+4\sqrt{\delta})|\xi-\eta|^{1/2}}.
\end{equation}
Moreover, if $\xi,\eta\in\mathbb{R}$ satisfy $|\eta|\geq |\xi|/4$ (or $|(k,\eta)|\geq|(k,\xi)|/4$ if $\ast=k$) then
\begin{equation}\label{vfc25}
\frac{A_\ast(t,\xi)}{A_\ast(t,\eta)}\lesssim_\delta e^{0.9\lambda(t)|\xi-\eta|^{1/2}}.
\end{equation}

(ii) Assume $t\in[0,\infty)$, $k,\ell\in\mathbb{Z}$ and $\xi,\eta\in\mathbb{R}$ satisfy $|(\ell,\eta)|\geq |(k,\xi)|/4$. If $t\not\in I_{k,\xi}$ or if $t\in I_{k,\xi}\cap I_{\ell,\eta}$, then
\begin{equation}\label{vfc26}
\frac{A_k(t,\xi)}{A_\ell(t,\eta)}\lesssim_\delta e^{0.9\lambda(t)|(k-\ell,\xi-\eta)|^{1/2}}.
\end{equation}
If $t\in I_{k,\xi}$ and $t\not\in I_{\ell,\eta}$, then
\begin{equation}\label{vfc27}
\frac{A_k(t,\xi)}{A_\ell(t,\eta)}\lesssim_\delta \frac{|\xi|}{k^2}\frac{1}{1+\big|t-\xi/k\big|} e^{0.9\lambda(t)|(k-\ell,\xi-\eta)|^{1/2}}.
\end{equation}
\end{lemma}

We also need estimates on the time derivatives of the weights $A_\ast$, see Lemma 7.4 in \cite{IOJI}. 

\begin{lemma}\label{lm:CDW}
(i) For all $t\ge 0,$ $\rho\in\mathbb{R}$, and $\ast\in\{NR,R\}$ we have
\begin{equation}\label{TLX3.5}
\frac{-\dot{A}_\ast(t,\rho)}{A_\ast(t,\rho)}\approx_\delta\left[\frac{\langle\rho\rangle^{1/2}}{\langle t\rangle^{1+\sigma_0}}+\frac{\partial_tw_\ast(t,\rho)}{w_\ast(t,\rho)}\right].
\end{equation}
and, for any $k\in\Z$,
\begin{equation}\label{eq:A_kxi}
\frac{-\dot{A}_k(t,\rho)}{A_k(t,\rho)}\approx_\delta\left[\frac{\langle k,\rho\rangle^{1/2}}{\langle t\rangle^{1+\sigma_0}}+\frac{\partial_tw_k(t,\rho)}{w_k(t,\rho)}\frac{1}{1+e^{\sqrt\delta(|k|^{1/2}-\langle\rho\rangle^{1/2})}w_k(t,\rho)}\right].
\end{equation}

(ii) For all $t\ge 0,$ $\xi,\,\eta\in\mathbb{R}$, and $\ast\in\{NR,R\}$ we have
\begin{equation}\label{vfc30}
\big|(\dot{A}_\ast/A_{\ast})(t,\xi)\big|\lesssim_\delta \big|(\dot{A}_\ast/A_{\ast})(t,\eta)\big|e^{4\sqrt{\delta}|\xi-\eta|^{1/2}}.
\end{equation}
Moreover, if $k,\ell\in\mathbb{Z}$ then
\begin{equation}\label{eq:CDW}
\big|(\dot{A}_k/A_k)(t,\xi)\big|\lesssim_\delta \big|(\dot{A}_\ell/A_{\ell})(t,\eta)\big|e^{4\sqrt{\delta}|k-\ell,\xi-\eta|^{1/2}}.
\end{equation}
Finally, if $\rho\in\mathbb{R}$ and $k\in\mathbb{Z}$ satisfy $|k|\leq\langle\rho\rangle+10$ then
\begin{equation}\label{vfc30.5}
\big|(\dot{A}_k/A_k)(t,\rho)\big|\approx_\delta\big|(\dot{A}_{NR}/A_{NR})(t,\rho)\big|\approx_\delta\big|(\dot{A}_R/A_R)(t,\rho)\big|.
\end{equation}
\end{lemma}

\subsection{Refined smoothness of the weights}
To bound the commutator terms in (\ref{wCD}), we need more refined smoothness properties of the weights $A_k(t,\xi)$ in $\xi\in\mathbb{R}$ than those proved in \cite{IOJI}. We start with a proposition that contains the main estimates necessary to control the commutator terms $\mathcal{C}_j$, $j\in\{1,2,3,4\}$, see (\ref{defC}) and (\ref{wCD}).

\begin{lemma}\label{A-A}
If $\xi,\eta\in\mathbb{R}$, $t\geq 0$, $k\in\mathbb{Z}$, and $\langle\xi-\eta\rangle\leq (\langle k,\xi\rangle+\langle k,\eta\rangle)/8$ then
 \begin{equation}\label{A-A2}
\big|A_k(t,\xi)-A_k(t,\eta)\big|\lesssim  A_R(t,\xi-\eta)A_k(t,\eta) e^{-(\lambda(t)/40)\langle \xi-\eta\rangle^{1/2}}\Big[\frac{C_0(\delta)}{\langle k,\xi\rangle^{1/8}}+\sqrt{\delta}\Big],
\end{equation}
for some constant $C_0(\delta)\gg 1$.
\end{lemma}

\begin{proof}
We sometimes use the following elementary inequality: if $a,b\in\mathbb{R}^n$ and $\beta\in[0,1]$ then
\begin{equation}\label{b>a}
    \langle b\rangle\ge \beta \langle a-b\rangle \qquad{\rm implies} \qquad \langle a\rangle^{1/2} \leq \langle b\rangle ^{1/2}+\big(1-\sqrt{\beta}/2\big)\langle a-b\rangle^{1/2}.
    \end{equation} 

To prove \eqref{A-A2} we can assume that $|\eta-\xi|\leq\langle k,\xi\rangle/100$; otherwise (\ref{A-A2}) follows from (\ref{TLX7}). By symmetry, we can also assume that $\xi\geq 0$.
By the definitions of the weights, we can write
\begin{equation}\label{A-A2.1}
\begin{split}
A_k(t,\xi)&-A_k(t,\eta)=\Big[e^{\lambda(t)\langle k,\xi\rangle^{1/2}}-e^{\lambda(t)\langle k,\eta\rangle^{1/2}}\Big]\Big[\frac{e^{\sqrt{\delta}\langle\xi\rangle^{1/2}}}{b_k(t,\xi)}+e^{\sqrt{\delta}|k|^{1/2}}\Big]\\
&+e^{\lambda(t)\langle k,\eta\rangle^{1/2}}\bigg[\frac{e^{\sqrt{\delta}\langle\xi\rangle^{1/2}}}{b_k(t,\xi)}-\frac{e^{\sqrt{\delta}\langle\eta\rangle^{1/2}}}{b_k(t,\xi)}\bigg]+e^{\lambda(t)\langle k,\eta\rangle^{1/2}}e^{\sqrt{\delta}\langle\eta\rangle^{1/2}}\bigg[\frac{1}{b_k(t,\xi)}-\frac{1}{b_k(t,\eta)}\bigg]\\
&:=\mathcal{T}_1+\mathcal{T}_2+\mathcal{T}_3.
\end{split}
\end{equation}
It suffices to prove that for each $i\in\{1,2,3\}$,
\begin{equation}\label{A-A2.2}
\big|\mathcal{T}_i\big|\lesssim A_R(t,\xi-\eta)A_k(t,\eta) e^{-(\lambda(t)/40)\langle \xi-\eta\rangle^{1/2}}\Big[\frac{C_0(\delta)}{\langle k,\xi\rangle^{1/8}}+\sqrt{\delta}\Big].
\end{equation}

The proof of (\ref{A-A2.2}) for $i\in\{1,2\}$ follows easily from (\ref{eq:comparisonweights1})-(\ref{dor20}). For the case $i=3$ we prove the following stronger bounds: if $\langle\xi-\eta\rangle\leq (\langle k,\xi\rangle+\langle k,\eta\rangle)/8$ and $|k|\leq 3|\xi|$ then
\begin{equation}\label{mur9}
\big|b_k(t,\xi)-b_k(t,\eta)\big|\lesssim b_k(t,\xi)\langle\xi-\eta\rangle e^{2\sqrt\delta|\xi-\eta|^{1/2}}\Big[\frac{C'_0(\delta)}{\langle k,\xi\rangle^{1/8}}+\sqrt{\delta}\Big].
\end{equation}
Indeed this follows from \eqref{dor22} unless $|\xi-\eta|\leq \langle k,\xi\rangle/100$ and $\langle k,\xi\rangle\geq \delta^{-12}$. In this case, we may assume that $\xi>0$ and the bounds \eqref{mur9} follow from \eqref{dor21} if $t\leq\xi^{3/4}$. In the remaining case when $t\geq \xi^{3/4}$ and $\max(100|\xi-\eta|,\delta^{-12})\leq \langle k,\xi\rangle$ we use \eqref{ReSm2} to bound
\begin{equation}\label{A-A2.3}
\begin{split}
\big|b_k(t,\xi)&-b_k(t,\eta)\big|=\frac{1}{d_0}\bigg|\int_{\mathbb{R}}\big[w_k(t,\xi+L_{\delta'}(t,\xi)\rho)-w_k(t,\eta+L_{\delta'}(t,\eta)\rho)\big]\varphi(\rho)\,d\rho\bigg|\\
 &\lesssim \int_{\mathbb{R}}\Big(e^{\sqrt{\delta}|\xi-\eta|^{1/2}+\sqrt{\delta}|L_{\delta'}(t,\xi)\rho-L_{\delta'}(t,\eta)\rho|^{1/2}}-1\Big)w_k(t,\xi+L_{\delta'}(t,\xi)\rho)\varphi(\rho)\,d\rho\\
 &\lesssim \sqrt{\delta}\langle\xi-\eta\rangle\,e^{2\sqrt\delta|\xi-\eta|^{1/2}}b_k(t,\xi).
\end{split}
\end{equation}
The bounds (\ref{mur9}) now follow.
\end{proof}

\subsubsection{The functions $\mu_k$ and $\mu_R$}\label{musection}

To bound the commutator terms $\mathcal{D}_j$ in (\ref{defD}), we need to prove smoothness of the weights $\mu_Y(t,\xi)$ in $\xi\in\mathbb{R}$. We prove two lemmas.

\begin{lemma}\label{Lmu1}
(i) For $t\ge0, \xi\in\R, k\in\mathbb{Z}$, we have
\begin{equation}\label{mudA1}
\mu_k(t,\xi)\approx_{\delta}\big|(\dot{A}_k/A_k)(t,\xi)\big|\qquad \mathrm{ and }\qquad\mu_R(t,\xi)\approx_{\delta}\big|(\dot{A}_R/A_R)(t,\xi)\big|
\end{equation}

(ii) Suppose that $\xi,\eta\in\mathbb{R}$, $t\ge 0$ satisfy
\begin{equation}\label{Lmu2.0}
|\xi|,|\eta|\geq 2\delta^{-10}, \qquad t\leq \min\{2|\xi|,2|\eta|\}, \qquad |\eta-\xi|\leq \min\{10L_1(t,\xi),10L_1(t,\eta)\}.
\end{equation}
Then, for any $k\in\mathbb{Z}$,
\begin{equation}\label{Lmu2.1}
\mu^\ast(t,\xi)\approx_\delta\mu^\ast(t,\eta)\quad\text{ thus }\quad\mu_R(t,\xi)\approx_{\delta}\mu_R(t,\eta)\text{ and }\mu_k(t,\xi)\approx_{\delta}\mu_k(t,\eta).
\end{equation}
\end{lemma}

\begin{proof} (i) We use the definitions \eqref{muD}-\eqref{mu1} and the formulas \eqref{reb8} and \eqref{TLX3.5}-\eqref{eq:A_kxi}. If $|\xi|\leq \delta^{-11}$ then $w_R(t,\xi)\approx_\delta w_k(t,\xi)\approx_\delta 1$, and it s easy to see that
\begin{equation*}
\big|(\dot{A}_R/A_R)(t,\xi)\big|\approx_\delta\langle t\rangle^{-1-\sigma_0}\approx_\delta\mu_R(t,\xi),\qquad \big|(\dot{A}_k/A_k)(t,\xi)\big|\approx_\delta\langle k\rangle^{1/2}\langle t\rangle^{-1-\sigma_0}\approx_\delta\mu_k(t,\xi),
\end{equation*}
as claimed. Moreover, if $|\xi|\geq \delta^{-11}$ and $t\geq 3|\xi|/2$ then 
\begin{equation*}
\big|(\dot{A}_R/A_R)(t,\xi)\big|\approx_\delta\frac{\langle\xi\rangle^{1/2}}{\langle t\rangle^{1+\sigma_0}}\approx_\delta\mu_R(t,\xi),\qquad \big|(\dot{A}_k/A_k)(t,\xi)\big|\approx_\delta\frac{\langle k,\xi\rangle^{1/2}}{\langle t\rangle^{1+\sigma_0}}\approx_\delta\mu_k(t,\xi),
\end{equation*}
which again gives \eqref{mudA1}. Finally, in view of \eqref{dor20}, it remains to show that 
\begin{equation}\label{mudA2}
\mu^\#(t,\xi)\approx_{\delta}\mu^{\ast}(t,\xi), \qquad{\rm for}\,\,0\leq t\leq7|\xi|/4, \,\,|\xi|\geq\delta^{-11}.
\end{equation}
This follows since $\mu^\#(t,\xi)\approx_{\delta}\mu^{\#}(t,\eta)$ if $|\eta-\xi|\leq L_1(t,\xi)$, see the more precise analysis in part (ii) below.

(ii) In view of \eqref{dor20} and \eqref{dor6}, it suffices to prove that $\mu^\ast(t,\xi)\approx_\delta\mu^\ast(t,\eta)$. We may assume that $\xi,\eta\geq 0$. If $t\leq t_{k_0(\xi),\xi}$ then $L_1(t,\xi)\leq 2\langle\xi\rangle^{1/2}$, thus $|\eta-\xi|\leq 20\langle\xi\rangle^{1/2}$. In particular $\mu^\#(t,\rho)\approx_\delta 1$ if $|\rho-\xi|\leq 50\langle\xi\rangle^{1/2}$, so $\mu^\ast(t,\xi)\approx_\delta \mu^\ast(t,\xi)\approx_\delta 1$, and the desired bounds \eqref{Lmu2.1} follow. The same argument works also if $t\leq t_{k_0(\eta),\eta}$.

Assume that $t\geq \max(t_{k_0(\xi),\xi},t_{k_0(\eta),\eta})$. Then $t\in I_{a',\xi}\cap I_{a,\eta}$ for suitable $a'\in\{1,2,\dots,k_0(\xi)\}$ and $a\in \{1,2,\dots,k_0(\eta)\}$. The condition (\ref{Lmu2.0}) implies 
$|\eta-\xi|\leq30a.$
Thus, $|a-a'|\leq 1$. Notice that $\mu^\#(t,\rho)\approx \mu^\#(t,\xi)\approx\frac{\delta^2}{1+\delta^2|t-\xi/a'|}$ if $|\xi-\rho|\leq 2L_1(t,\xi)$. Thus 
\begin{equation}\label{mur1}
\mu^\#(t,\xi)\approx \mu^\ast(t,\xi)\qquad\text{ and }\qquad\mu^\#(t,\eta)\approx \mu^\ast(t,\eta).
\end{equation}

If $a'=a$, then 
\begin{equation}\label{Lmu2.2}
\frac{\mu^\#(t,\xi)}{\mu^\#(t,\eta)}=\frac{1+\delta^2|t-\eta/a|}{1+\delta^2|t-\xi/a|}\leq 1+\delta^2\bigg|\frac{\eta-\xi}{a}\bigg|\lesssim1.
\end{equation}
Similarly $\mu^\#(t,\eta)/\mu^\#(t,\xi)\lesssim 1$, and the desired estimates in \eqref{Lmu2.1} follow using also \eqref{mur1}. On the other hand, if $a\neq a'$, then 
$$\Big|t-\frac{\xi}{a'}\Big|\gtrsim \frac{|\xi|}{|a'|^2} \qquad {\rm and}\qquad \Big|t-\frac{\eta}{a}\Big|\gtrsim \frac{|\eta|}{a^2},$$
and consequently 
\begin{equation}\label{Lmu2.4}
\mu^\#(t,\xi)\approx \frac{|a'|^2}{|\xi|}\approx \frac{a^2}{|\eta|}\approx \mu^\#(t,\eta).
\end{equation}
The bounds (\ref{Lmu2.1}) follow using again \eqref{mur1}.
\end{proof}

The following lemma plays an important role in controlling the commutator terms $\mathcal{D}_j$, $j\in\{1,2,3,4\}$, see (\ref{defD}) and \eqref{defCD4}.

\begin{lemma}\label{Lmu3}
Assume $\xi,\eta\in\mathbb{R}$, $k\in\mathbb{Z}$, and $t\ge 0$. Then:

(i) if $\langle\xi-\eta\rangle\geq \langle k,\eta\rangle/100$, then
\begin{equation}\label{Lmu3.1}
\mu_k(t,\xi)+\mu_k(t,\eta)\lesssim_{\delta} \mu_R(t,\xi-\eta)\,e^{4\sqrt{\delta}|k,\eta|^{1/2}};
\end{equation}

(ii) if $\langle\xi-\eta\rangle\leq (\langle k,\xi\rangle+\langle k,\eta\rangle)/8$, then there is $C_1(\delta)\gg 1$ such that
\begin{equation}\label{Lmu3.2}
\big|\mu_k(t,\xi)-\mu_k(t,\eta)\big|\lesssim \langle\xi-\eta\rangle\mu_k(t,\eta)\,e^{4\sqrt{\delta}|\xi-\eta|^{1/2}}\Big[\frac{C_1(\delta)}{\langle k,\xi\rangle^{1/8}}+\sqrt{\delta}\Big].
\end{equation}

(iii) in all cases we have
\begin{equation}\label{runt30}
\mu_k(t,\xi)\lesssim_\delta \mu_k(t,\eta)e^{6\sqrt\delta |\xi-\eta|^{1/2}}.
\end{equation}
\end{lemma}

\begin{proof}
(i) By the definitions \eqref{mu1}, it suffices to prove that
\begin{equation}\label{Lmu3.11}
\frac{\langle k,\xi\rangle^{1/2}}{\langle t\rangle^{1+\sigma_0}}+\frac{\langle k,\eta\rangle^{1/2}}{\langle t\rangle^{1+\sigma_0}}\lesssim_{\delta} \frac{\langle \xi-\eta\rangle^{1/2}}{\langle t\rangle^{1+\sigma_0}}e^{4\sqrt{\delta}|k,\eta|^{1/2}}
\end{equation}
and
\begin{equation}\label{Lmu3.12}
\mu^{\ast}(t,\xi)+\mu^{\ast}(t,\eta)\lesssim_{\delta} \mu_R(t,\xi-\eta)\,e^{4\sqrt{\delta}|\eta|^{1/2}}.
\end{equation}
The bounds (\ref{Lmu3.11}) are easy, and so we focus on (\ref{Lmu3.12}).  Since
\begin{equation}\label{Lmu3.111}
\mu^{\ast}(t,\xi)+\mu^{\ast}(t,\eta)\lesssim\delta^2\qquad {\rm and}\qquad \mu_R(t,\xi-\eta)\ge\frac{\langle\xi-\eta\rangle^{1/2}}{\langle t\rangle^{1+\sigma_0}},
\end{equation}
it suffices to consider the case $t\ge\delta^{-12}\langle \eta\rangle$. In this case, in view of the definitions (\ref{muD})-(\ref{muaD}), 
$\mu^{\ast}(t,\eta)=0.$  We distinguish two cases.

{\bf Case 1.} Suppose that $t\geq 3|\xi|/2$ and $t\ge\delta^{-12}\langle \eta\rangle$. From (\ref{muD})-(\ref{muaD}), we get $\mu^{\ast}(t,\xi)\lesssim \langle\xi\rangle^{-1}$. Since the left-hand side of (\ref{Lmu3.12}) vanishes if $t>\delta^{-1}\langle \xi-\eta\rangle$, we can assume that $t\leq \delta^{-1}\langle \xi-\eta\rangle$. Then from $t\ge\delta^{-12}\langle \eta\rangle$, it follows that $\langle\xi-\eta\rangle\approx \langle \xi\rangle$ and  
\begin{equation}
\frac{1}{\langle\xi\rangle}\lesssim_{\delta}\frac{\langle \xi-\eta\rangle^{1/2}}{\langle t\rangle^{1+\sigma_0}}\frac{1}{\langle \xi-\eta\rangle^{1/4}},
\end{equation}
from which (\ref{Lmu3.12}) follows easily.

{\bf Case 2.} Now assume that $t\leq 3|\xi|/2$ and $t\ge\delta^{-12}\langle \eta\rangle$.  If $|\eta|\geq t^{1/10}$ then \eqref{Lmu3.12} follows from \eqref{Lmu3.111}. On the other hand, if $|\eta|\leq t^{1/10}$ and $t\leq |\xi|^{9/10}$ then \eqref{Lmu3.12} follows from \eqref{Lmu2.1}.

In the remaining case $|\xi|\geq 2\delta^{-11}$, $|\xi|^{9/10}\leq t\leq 3|\xi|/2$, $|\eta|\leq 2|\xi|^{1/10}$ we prove that
\begin{equation}\label{mur2}
\mu^\#(t,\xi)\lesssim_\delta\mu^\#(t,\xi-\eta)e^{\sqrt{\delta}|\eta|^{1/2}}.
\end{equation}
This suffices to prove \eqref{Lmu3.12}, due to \eqref{mudA2}. To prove \eqref{mur2}, we may assume that $\xi>0$ and $t\in I_{a,\xi}$ for some $a\in[1,4|\xi|^{1/10}]$. If $|t-\xi/a|\geq\xi/(20a^2)$ then $\mu^\#(t,\xi)\approx_\delta a^2/\xi$, and it follows easily that $\mu^\#(t,\xi)\lesssim_\delta\mu^\#(t,\xi-\eta)$, which is better than needed. On the other hand, if $|t-\xi/a|\leq\xi/(20a^2)$ then $t\in I_{a,\xi-\eta}$ as well, and we estimate
\begin{equation*}
\frac{\mu^\#(t,\xi)}{\mu^\#(t,\xi-\eta)}=\frac{1+\delta^2|t-(\xi-\eta)/a|}{1+\delta^2|t-\xi/a|}\leq 1+\delta^2\bigg|\frac{\eta}{a}\bigg|\lesssim_\delta e^{\sqrt{\delta}|\eta|^{1/2}},
\end{equation*}
and \eqref{mur2} follows in this last case.

(ii) We now prove (\ref{Lmu3.2}). Since $\langle k,\xi\rangle\approx \langle k,\eta\rangle$, by the definitions it suffices to prove that
\begin{equation}\label{Lmu3.21}
\Big|\frac{\langle k,\xi\rangle^{1/2}}{\langle t\rangle^{1+\sigma_0}}-\frac{\langle k,\eta\rangle^{1/2}}{\langle t\rangle^{1+\sigma_0}}\Big|\lesssim\langle\xi-\eta\rangle\frac{\langle k,\eta\rangle^{1/2}}{\langle t\rangle^{1+\sigma_0}}\frac{1}{\langle k,\eta\rangle^{1/8}}
\end{equation}
and
\begin{equation}\label{Lmu3.22}
\begin{split}
\Big|\frac{\mu^{\ast}(t,\xi)}{1+e^{\sqrt{\delta}(|k|^{1/2}-\langle\xi\rangle^{1/2})}b_k(t,\xi)}&-\frac{\mu^{\ast}(t,\eta)}{1+e^{\sqrt{\delta}(|k|^{1/2}-\langle\eta\rangle^{1/2})}b_k(t,\eta)}\Big|\\
&\lesssim \langle\xi-\eta\rangle\mu_k(t,\eta)\,e^{4\sqrt{\delta}|\xi-\eta|^{1/2}}\Big[\frac{C_2(\delta)}{\langle k,\eta\rangle^{1/8}}+\sqrt{\delta}\Big].
\end{split}
\end{equation}
The bounds (\ref{Lmu3.21}) follow easily. To prove (\ref{Lmu3.22}) we first eliminate some of the simpler cases. 

These bounds follow easily if $\langle t\rangle\geq 4\langle k,\eta\rangle$ (the left-hand side is $0$) or if ($\langle t\rangle\leq 4\langle k,\eta\rangle$ and $|\xi-\eta|\geq\langle k,\eta\rangle/100$) or if $\langle t\rangle \leq \delta^{-12}\langle k,\eta\rangle^{1/4}$ , due to the lower bound $\mu_k(t,\eta)\ge \frac{\langle k,\eta\rangle^{1/2}}{\langle t\rangle^{1+\sigma_0}}$. On the other hand, if $\delta^{-12}\langle k,\eta\rangle^{1/4}\leq\langle t\rangle\leq 4\langle k,\eta\rangle$ and $|k|\geq 2\max(|\xi|,|\eta|)$ then both terms in the left-hand side are bounded by $C_\delta e^{-\delta |k|}$, since $b_k(t,\rho)\gtrsim_\delta e^{-\delta\sqrt{|\rho|}}$ for any $\rho\in\mathbb{R}$, and the bounds \eqref{Lmu3.22} follow. Finally, if
\begin{equation}\label{mur7}
\delta^{-12}\langle k,\eta\rangle^{1/4}\leq\langle t\rangle\leq 4\langle k,\eta\rangle,\qquad |\xi-\eta|\leq\langle k,\eta\rangle/100,\qquad |k|\leq 2\max(|\xi|,|\eta|)
\end{equation}
then $\langle\xi\rangle\approx\langle\eta\rangle$ and we estimate the left-hand side of \eqref{Lmu3.22} by $I+II$ where
\begin{equation}\label{Lmu3.25}
\begin{split}
I:=&\mu^{\ast}(t,\eta)\Big|\frac{1}{1+e^{\sqrt{\delta}(|k|^{1/2}-\langle\xi\rangle^{1/2})}b_k(t,\xi)}-\frac{1}{1+e^{\sqrt{\delta}(|k|^{1/2}-\langle\eta\rangle^{1/2})}b_k(t,\eta)}\Big|,\\
II:=&\frac{|\mu^{\ast}(t,\xi)-\mu^{\ast}(t,\eta)|}{1+e^{\sqrt{\delta}(|k|^{1/2}-\langle\xi\rangle^{1/2})}b_k(t,\xi)}.
\end{split}
\end{equation}

Using \eqref{dor22} and \eqref{mur9} we estimate
\begin{equation*}
\begin{split}
I&\lesssim\frac{\mu^{\ast}(t,\eta)e^{\sqrt{\delta}(|k|^{1/2}-\langle\xi\rangle^{1/2})}\big\{|b_k(t,\xi)-b_k(t,\eta)|+|1-e^{\sqrt{\delta}(\langle\xi\rangle^{1/2}-\langle\eta\rangle^{1/2})}|b_k(t,\eta)\big\}}{[1+e^{\sqrt{\delta}(|k|^{1/2}-\langle\eta\rangle^{1/2})}b_k(t,\eta)][1+e^{\sqrt{\delta}(|k|^{1/2}-\langle\xi\rangle^{1/2})}b_k(t,\xi)]}\\
&\lesssim\frac{\mu^{\ast}(t,\eta)}{1+e^{\sqrt{\delta}(|k|^{1/2}-\langle\eta\rangle^{1/2})}b_k(t,\eta)}\Big\{\frac{|b_k(t,\xi)-b_k(t,\eta)|}{b_k(t,\xi)}+\frac{C(\delta)e^{\sqrt{\delta}|\xi-\eta|^{1/2}}b_k(t,\eta)}{\langle\xi\rangle^{1/2}b_k(t,\xi)}\Big\}\\
&\lesssim \mu_k(t,\eta)\langle\xi-\eta\rangle e^{2\sqrt\delta|\xi-\eta|^{1/2}}\Big[\frac{C'(\delta)}{\langle \xi\rangle^{1/8}}+\sqrt{\delta}\Big].
\end{split}
\end{equation*}
This is consistent with the desired estimates \eqref{Lmu3.22}. 

To control $II$ it suffices to show that
\begin{equation}\label{mur11}
\begin{split}
|\mu^{\ast}(t,\xi)-\mu^{\ast}(t,\eta)|\lesssim \langle\xi-\eta\rangle e^{2\sqrt{\delta}|\xi-\eta|^{1/2}}\Big[\frac{C''(\delta)}{\langle \eta\rangle^{1/8}}+\sqrt{\delta}\Big]\big(\mu^\ast(t,\eta)+\langle \eta\rangle^{1/2}/\langle t\rangle^{1+\sigma_0}\big)
\end{split}
\end{equation}
provided that $\xi,\eta\in\mathbb{R}$ and $t\geq 0$ satisfy $\delta^{-12}\langle\eta\rangle^{1/4}\leq\langle t\rangle\leq 20\langle\eta\rangle$ and $|\xi-\eta|\leq\langle\eta\rangle/10$.

{\bf Case 1.} Assume first that $t\leq\langle\eta\rangle^{7/8}$. If $|\xi-\eta|\geq\langle\eta\rangle^{1/10}$ then \eqref{mur11} follows easily. Assume that $|\xi-\eta|\leq\langle\eta\rangle^{1/10}$. For $\rho$ between $\xi$ and $\eta$ we have
\begin{equation}\label{Lmu3.26}
\big|\partial_{\rho}\mu^{\ast}(t,\rho)\big|\lesssim_\delta\mu^{\ast}(t,\rho)\frac{1}{L_{\delta'}(t,\rho)}\lesssim_\delta \mu^{\ast}(t,\rho)\langle\rho\rangle^{-1/8},
\end{equation}
using the definition (\ref{muaD}) and Lemma \ref{Lmu1} (ii). The desired bounds \eqref{mur11} follow using again Lemma \ref{Lmu1} (ii) and the observation that $|\xi-\eta|\leq 2L_1(t,\eta)$.

{\bf Case 2.} Suppose now that $t\geq\langle\eta\rangle^{7/8}$. We may assume again that $|\xi-\eta|\leq\langle\eta\rangle^{1/10}$ and $\eta\geq \delta^{-12}$. If $t\geq 5\eta/4$ then $\mu^{\ast}(t,\xi)+\mu^{\ast}(t,\eta)\lesssim \langle\eta\rangle^{-1}$, and the desired bounds \eqref{mur11} follow. 

Finally, assume that $\eta^{7/8}\leq t\leq 5\eta/4$. Therefore there is $a\in[1,2\eta^{1/8}]\cap\mathbb{Z}$ such that $t\in I_{a,\eta}$. If $|t-\eta/a|\geq \eta/(8a^2)$ then the definitions \eqref{muD}--\eqref{muaD} show that
\begin{equation*}
\mu^{\ast}(t,\xi)+\mu^{\ast}(t,\eta)\lesssim a^2\langle\eta\rangle^{-1}\lesssim\langle\eta\rangle^{-3/4},
\end{equation*}
and the desired bounds \eqref{mur11} follow. On the other hand, if $|t-\eta/a|\leq \eta/(8a^2)$ then $|t-\xi/a|\leq \eta/(6a^2)$ (due to the assumptions $|\xi-\eta|\leq\eta^{1/10}$ and $\eta\geq \delta^{-12}$), thus $t\in I_{a,\xi}$. In fact, $t\in I_{a,\rho}$ for any $\rho$ satisfying $|\rho-\eta|\leq 4L_{\delta'}(t,\eta)$ or $|\rho-\xi|\leq 4L_{\delta'}(t,\xi)$. Therefore, using \eqref{muD}--\eqref{muaD} we estimate
\begin{equation}\label{Lmu3.31}
\begin{split}
\big|\mu^{\ast}(t,\xi)&-\mu^{\ast}(t,\eta)\big|=\frac{1}{d_0}\Big|\int_{\mathbb{R}}\big[\mu^\#(t,\xi+L_{\delta'}(t,\xi)\rho)-\mu^\#(t,\eta+L_{\delta'}(t,\eta)\rho)\big]\varphi(\rho)\,d\rho\Big|\\
&\leq\frac{1}{d_0}\Big|\int_{\mathbb{R}}\mu^\#(t,\eta+L_{\delta'}(t,\eta)\rho)\varphi(\rho)\cdot \delta^2|(\xi+L_{\delta'}(t,\xi)\rho)-(\eta+L_{\delta'}(t,\eta)\rho)|\,d\rho\Big|\\
&\lesssim \delta^2\,\langle\xi-\eta\rangle^{1/2}\mu^\ast(t,\eta)
\end{split}
\end{equation}
which completes the proof of (\ref{Lmu3.22}). 

(iii) The bounds \eqref{runt30} follow from \eqref{Lmu3.2} if $\langle\xi-\eta\rangle\leq (\langle k,\xi\rangle+\langle k,\eta\rangle)/8$. On the other hand, if $\langle\xi-\eta\rangle\geq (\langle k,\xi\rangle+\langle k,\eta\rangle)/8$ then we prove the stronger bounds
\begin{equation}\label{runt30.1}
\frac{\langle k,\xi\rangle^{1/2}}{\langle t\rangle^{1+\sigma_0}}+\mu^{\ast}(t,\xi)\lesssim_\delta \frac{\langle k,\eta\rangle^{1/2}}{\langle t\rangle^{1+\sigma_0}}e^{6\sqrt\delta |\xi-\eta|^{1/2}}.
\end{equation}
The bound on the first term in the left-hand side is elementary. For the second term, we notice that it is nontrivial only if $|\xi|\geq\delta^{-9}$ and $t\leq 4|\xi|$. The assumption  $\langle\xi-\eta\rangle\geq (\langle k,\xi\rangle+\langle k,\eta\rangle)/8$ then gives $|\xi-\eta|\geq|\xi|/100$, and the bound on the second term in \eqref{runt30.1} is clear.
\end{proof}

\bigskip
\begin{center}
{\bf Acknowledgement}\\

We are grateful to Vladimir Sverak for suggesting us to consider this problem.
\end{center}

\end{document}